\documentclass[a4paper]{amsart}
\usepackage{amsmath}
\usepackage{amsthm}
\usepackage{amssymb}
\usepackage{enumerate}
\usepackage{latexsym}
\usepackage{empheq}
\newtheorem{Thm}{Theorem}[section]
\newtheorem{Prop}[Thm]{Proposition}
\newtheorem{Cor}[Thm]{Corollary}
\newtheorem{Lem}[Thm]{Lemma}
\newtheorem{Clm}[Thm]{Claim}
\newtheorem{Conj}[Thm]{Conjecture}
\newtheorem{Ma}{Main Theorem 1}[section]
\newtheorem{Mb}{Main Theorem 1'}[section]
\newtheorem{Md}{Main Theorem 2}[section]
\newtheorem{Me}{Main Theorem 3}[section]

\theoremstyle{definition}
\newtheorem{notation}[Thm]{Notation}
\newtheorem{Def}[Thm]{Definition}
\theoremstyle{remark}
\newtheorem{Rem}{Remark}[section]
\newtheorem{Example}{Example}[section]
\newcommand{\Ric}{\mathop{\mathrm{Ric}}\nolimits}
\newcommand{\tr}{\mathop{\mathrm{tr}}\nolimits}
\newcommand{\Imag}{\mathop{\mathrm{Im}}\nolimits}
\newcommand{\Id}{\mathop{\mathrm{Id}}\nolimits}
\newcommand{\Vol}{\mathop{\mathrm{Vol}}\nolimits}
\newcommand{\diam}{\mathop{\mathrm{diam}}\nolimits}
\newcommand{\Card}{\mathop{\mathrm{Card}}\nolimits}
\newcommand{\Span}{\mathop{\mathrm{Span}}\nolimits}
\newcommand{\rad}{\mathop{\mathrm{rad}}\nolimits}
\newcommand{\LIP}{\mathop{\mathrm{LIP}}\nolimits}
\newcommand{\Test}{\mathop{\mathrm{Test}}\nolimits}
\newcommand{\Hess}{\mathop{\mathrm{Hess}}\nolimits}
\newcommand{\Scal}{\mathop{\mathrm{Scal}}\nolimits}
\newcommand{\Ca}{\mathop{1}\nolimits}
\newcommand{\Cb}{\mathop{2}\nolimits}
\newcommand{\Ck}{\mathop{3}\nolimits}
\newcommand{\Cm}{\mathop{4}\nolimits}
\newcommand{\Da}{\mathop{1}\nolimits}
\newcommand{\Db}{\mathop{2}\nolimits}
\newcommand{\Dp}{\mathop{3}\nolimits}
\newcommand{\Dq}{\mathop{4}\nolimits}
\newcommand{\Dr}{\mathop{5}\nolimits}

\title[eigenvalue pinching without positive Ricci]{Sphere theorems and eigenvalue pinching without positive Ricci curvature assumption}
\author{Masayuki Aino}
\address{Graduate School of Mathematics, Nagoya University, Chikusa-Ku Nagoya, 464-8602, Japan}
\email{m16100c@math.nagoya-u.ac.jp}
\begin{document}
\maketitle
\begin{abstract}
Considering the almost rigidity of the Obata theorem, we generalize Petersen and Aubry's sphere theorem about eigenvalue pinching without assuming the positivity of Ricci curvature, only assuming $\Ric\geq-Kg$ and $\diam\leq D$ for some positive constants $K>0$ and $D>0$.
\end{abstract} 
\tableofcontents
\section{Introduction}
The main aim of this article is to give the almost rigidity result about the Obata theorem without assuming the positivity of Ricci curvature, only assuming $\Ric\geq-Kg$ and $\diam\leq D$ for some positive constants $K>0$ and $D>0$.
The Obata theorem  is as follows:
\begin{Thm}[\cite{Ob}] \label{Ob}
Take an integer $n\geq 2$.
Let $(M,g)$ be a closed Riemannian manifold of dimension $n$. If $(M,g)$ admits a non-constant function $f\in C^\infty(M)$ with $\nabla^2 f + f g=0$, then $(M,g)$ is isometric to the standard sphere of radius $1$.
\end{Thm}
We emphasize that the Obata theorem itself does not require the assumption of positive Ricci curvature.
The Obata theorem implies the celebrated Lichnerowicz-Obata theorem.
In the following, $\lambda_k(g)$ denotes the $k$-th eigenvalue of the Laplacian acting on functions.
\begin{Thm}
Take an integer $n\geq 2$.
Let $(M,g)$ be an $n$-dimensional closed Riemannian manifold. If $\Ric \geq (n-1) g$, then $\lambda_1(g)\geq n$.
The equality holds if and only if $(M,g)$ is isometric to the standard sphere of radius $1$.
\end{Thm}

Petersen \cite{Pe1} and Aubry \cite{Au} showed the stability result of the Lichnerowicz-Obata theorem.
In the following, $d_{GH}$ denotes the Gromov-Hausdorff distance function (see Definition \ref{DGH} for the definition of the Gromov-Hausdorff distance).
\begin{Thm}[\cite{Au}, \cite{Pe1}]\label{PA}
Given an integer $n\geq 2$ and an $\epsilon>0$, there exists $\delta(n,\epsilon)>0$ such that if $(M,g)$ is an $n$-dimensional closed Riemannian manifold with $\Ric \geq (n-1) g$ and $\lambda_n(g)\leq n+\delta$, then $d_{GH}(M,S^n)\leq \epsilon$.
\end{Thm}
Note that Petersen considered the pinching condition on $\lambda_{n+1}$ and Aubry improved it.
Aubry also showed that the assumption $\lambda_{n-1}\leq n+\delta$ is not enough to get the Gromov-Hausdorff closeness to the standard sphere.
In the proof of Theorem 1.3, the following vector bundle and its connection play an important role.
For a Riemannian manifold $(M,g)$, we put $E=E_M:=TM \oplus \mathbb{R}e$, where $\mathbb{R}e$ is a rank $1$ trivial bundle on $M$.
We consider the product metric $\langle \cdot,\cdot \rangle_E$ and the following connection $\nabla^E$ on $E$:
\begin{equation*}
\begin{split}
\langle X+f e, Y+h e\rangle_E:=&g(X,Y)+f h,\\
\nabla^E_Y (X+f e):=&\nabla_Y X+ f Y+(Y f-g(Y,X))e,
\end{split}
\end{equation*}
for any $X,Y\in \Gamma(TM)$ and $f,h\in C^\infty(M)$.
Define $\bar{\Delta}^E:=(\nabla^E)^\ast \nabla^E$. See Definition \ref{def1} for details.
When $M$ is closed, we consider the eigenvalues of $\bar{\Delta}^E$:
\begin{equation*}
0\leq \lambda_1(\bar{\Delta}^E)\leq\lambda_2(\bar{\Delta}^E)\leq \cdots  \to \infty.
\end{equation*}
It is not difficult to prove that $\lambda_1(\bar{\Delta}^E)=0$ holds if and only if there exists a non-constant solution of the Obata equation $\nabla^2 f+f g=0$.

The following six conditions are mutually equivalent for any sequence of $n$-dimensional closed Riemannian manifolds $\{(M_i,g_i)\}_{i\in \mathbb{N}}$ with $\Ric \geq (n-1)g_i$:
\begin{itemize}
\item[(i)] $\Vol(M_i)\to \Vol(S^n)$,
\item[(ii)] $\rad(M_i)\to \pi$,
\item[(iii)] $\lambda_{n+1}(g_i)\to n$,
\item[(iv)] $\lambda_{n}(g_i)\to n$,
\item[(v)] $\lambda_{n}(\bar{\Delta}^{E},M_i)\to 0$,
\item[(vi)] $d_{GH}(M_i,S^n)\to 0$,
\end{itemize}
when $i\to \infty$,
where we defined $\rad (M)=\inf \{r\in\mathbb{R}: M=B_r(x)\text{ for some $x\in M$}\}$ for any Riemannian manifold $(M,g)$ (see \cite{Au}, \cite{CC3}, \cite{Co1}, \cite{Co2}, \cite{Co3}, \cite{Ho}, \cite{Pe1}).
In this paper, we show that  we only need to assume $\Ric\geq -Kg_i$, $\diam(M_i)\leq D$ to show that (v) implies (vi).
\begin{Ma}
Given an integer $n\geq 2$ and positive real numbers $\epsilon>0$, $K>0$ and $D>0$, there exists $\delta(n,K,D,\epsilon)>0$ such that if $(M,g)$ is an $n$-dimensional closed Riemannian manifold with  $\Ric \geq-K g$, $\diam(M)\leq D$ and $\lambda_n(\bar{\Delta}^E)\leq\delta$, then $d_{GH}(M,S^n)\leq \epsilon$.
\end{Ma}
\begin{Rem}
In fact, we prove
$$
d_{GH}(M,S^n)\leq C(n,K,D)\lambda_n(\bar{\Delta}^E)^{\frac{1}{1000n^2}}.
$$
See Theorem \ref{p43b}.
This estimate might be far from the optimal one.
\end{Rem}

This result is equivalent to the following result (see section 4).
\begin{Mb}
Given an integer $n\geq 2$ and positive real numbers $\epsilon>0$, $K>0$ and $D>0$, there exists $\delta(n,K,D,\epsilon)>0$ such that the following property holds.
Let $(M,g)$ be an $n$-dimensional closed Riemannian manifold with $\Ric \geq-K g$ and $\diam(M)\leq D$. If there exists an $n$-dimensional subspace $V$ of $C^\infty(M)$ such that  $\|\nabla^2 f+fg\|_2\leq \delta\|f\|_2$ holds for all $f\in V$, then
$d_{GH}(M,S^n)\leq \epsilon$.
\end{Mb}
The equivalence of these theorems corresponds to the equivalence of the conditions (iv) and (v).

The key of the proof of these main theorems is the following:
\begin{Prop}\label{thm14}
Take an integer $n\geq 2$ and positive real numbers  $K>0$ and $D>0$.
Let $(M,g)$ be an $n$-dimensional closed Riemannian manifold with $\Ric \geq-K g$ and $\diam(M)\leq D$.
Suppose that a non-zero function $f\in C^\infty(M)$ satisfies $\|\nabla^2 f+f g\|_{2}\leq \delta\|f\|_{2}$ for a sufficiently small $\delta>0$.
Then, there exist a non-constant function $f_1$ and a point $p\in M$ such that the following properties hold.
Putting $h\colon M\to \mathbb{R}$ as $h=\sqrt{n+1}\|f_1\|_{L^2}\cos d(p,\cdot)$, we have
\begin{align}
\label{f} \|f_1-h\|_{\infty}&\leq C(n,K,D) \delta^\frac{1}{48n}\|f_1\|_{2},\\
\label{grad}\|\nabla f_1-\nabla h\|_{2}&\leq  C(n,K,D) \delta^\frac{1}{48n}\|f_1\|_{2},\\
\label{hess}\|\nabla^2 f_1+f_1 g\|_{2}&\leq C(n,K,D)\delta^\frac{1}{2}\|f_1\|_{2},
\end{align}
for some positive constant $C(n,K,D)>0$.
\end{Prop}

For the proof of Proposition \ref{thm14}, see Lemma \ref{p2c} and Proposition \ref{p2i} (see also Lemma \ref{add} and Definition \ref{Def1}).
The conclusions (\ref{f}), (\ref{grad}) and (\ref{hess}) correspond to the assumptions (2.2), (2.3) and (2.4) in Cheeger-Colding's paper \cite[p.198]{CC2}, respectively.
These kind of properties are usually deduced using the condition that some geometrical quantity is almost optimal to the assumption of Ricci curvature, e.g., $\Ric\geq (n-1)g$ (see \cite{Co1} and \cite{Pe1}).
Proposition \ref{thm14} claims that, for our case, we do not need to assume (\ref{f}) and (\ref{grad}) once we know $\|\nabla^2 f+f g\|_2\leq \delta\|f\|_{2}$, even if we only assumed $\Ric \geq-K g$ and $\diam(M)\leq D$.



We next consider constant scalar curvature metrics.
The following theorem is known (see \cite[Theorem 24]{Kh}).
\begin{Thm}
Take an integer $n\geq 2$.
Let $(M,g)$ be an $n$-dimensional closed Riemannian manifold with $\Scal_g=n(n-1)$, where $\Scal_g$ denotes its scalar curvature. If $(M,g)$ admits a non-constant function $f\in C^\infty(M)$ with $\nabla^2 f + \frac{\Delta f}{n} g=0$, then $(M,g)$ is isometric to the standard sphere of radius $1$.
\end{Thm}
Note that the condition $\nabla^2 f + \frac{\Delta f}{n} g=0$ is weaker than the condition $\nabla^2 f + f g=0$.
By the Bochner formula, the pinching condition $\|\nabla^2 f +\frac{\Delta f}{n}g\|_2\leq \delta\|\Delta f\|_2$ is equivalent to the condition $\int_M \Ric(\nabla f,\nabla f)\,d\mu_g\geq\left(\frac{n-1}{n}-\delta^2\right)\int_M (\Delta f)^2\,d\mu_g$.
Thus, we consider the Riemannian invariants $\Omega_k(g)$ introduced in \cite{Ai}:
\begin{equation*}
\Omega_k(g):=\inf\Big\{\sup_{f\in V\setminus \{0\}}\frac{\int_M \Ric(\nabla f,\nabla f)\,d\mu_g}{\int_M (\Delta f)^2\,d\mu_g}: V\text{ is a $k$-dimensional subspace of $H$}\Big\},
\end{equation*}
where we put
\begin{equation*}
H:=\{f\in W^{2,2}(M):\int_M f\,d\mu_g=0\}.
\end{equation*}
Note that we always have $\Omega_k\leq \frac{n-1}{n}$, and for the $n$-dimensional standard sphere, we have $\Omega_1=\cdots=\Omega_{n+1}=\frac{n-1}{n}$.
Using these Riemannian invariants, we show the following result as a corollary of Main Theorem 1'.
\begin{Md}
Given an integer $n\geq 2$ and positive real numbers $\epsilon>0$, $K>0$ and $D>0$, there exists $\delta(n,K,D,\epsilon)>0$ such that if $(M,g)$ is an $n$-dimensional closed Riemannian manifold with $\Scal_g=n(n-1)$, $\Ric \geq-K g$, $\diam(M)\leq D$ and $\Omega_{n}(g)\geq \frac{n-1}{n}-\delta^2$, then
$d_{GH}(M,S^n)\leq \epsilon$.
\end{Md}

By the following theorem due to Cheeger-Colding \cite[Theorem A.1.12]{CC1}, we get $M$ is diffeomorphic to the $n$-dimensional standard sphere if $\epsilon$ is sufficiently small in our main theorems.
\begin{Thm}[\cite{CC1}]
Take an integer $n\geq 2$ and a positive real number $K>0$.
Let $\{(M_i,g_i)\}_{i\in \mathbb{N}}$ be a sequence of $n$-dimensional closed Riemannian manifolds with $\Ric_{g_i} \geq-K g_i$ and
converges to an $n$-dimensional closed Riemannian manifold $(M,g)$ in Gromov-Hausdorff topology.
Then, for sufficiently large $i$, $M_i$ is diffeomorphic to $M$.
\end{Thm}

In section 5, we give an application of Main Theorem 1 to the almost umbilical manifolds.
We show that $\lambda_{n+1}(\bar{\Delta}^E)$ is small for the almost umbilical manifolds:
\begin{Prop}\label{umb}
Let \((M,g)\) be an $n$-dimensional oriented closed Riemannian manifold and $\iota \colon (M,g)\to \mathbb{R}^{n+1}$ an isometric immersion.
Then, we have
\begin{equation*}
\lambda_{n+1}(\bar{\Delta}^E)\leq 2\|A- \Id \|_2^2,
\end{equation*}
where $A\in \Gamma(T^\ast M\otimes T M)$ denotes the second fundamental form.
\end{Prop}
Let $h:=\frac{1}{n}\tr A$ denotes the mean curvature.
Several authors considered whether the manifold is close to the standard sphere when $A-h\Id$ is small in certain senses (\cite{DM}, \cite{DM2}, \cite{DRG}, \cite{Pz}, \cite{RS}, \cite{SX} and\cite{SX2}).
Perez \cite{Pz} and Roth-Scheuer \cite{RS} considered the assumption that $\|A-h\Id\|_p$ is small for some $p>n$ and showed the closeness to the standard sphere in several meaning. In particular, they showed that $M$ is Hausdorff close to the standard sphere under such an assumption (see Definition \ref{Dhau} for the definition of the Hausdorff distance).
De Rosa and Gioffr\`{e} \cite{DRG} considered a more general situation, called anisotropic nearly umbilical hypersurfaces, and showed a $W^{2,p}$-approximation under the convexity assumption.

In general, the condition $A-\overline{h}\Id$ is small is stronger than the condition $A-h\Id$ is small,
where $\overline{h}$ denotes the average value of the mean curvature:
\begin{equation*}
\overline{h}:=\frac{1}{\Vol(M)}\int_M h\,d\mu_g
\end{equation*}
because we have $|A-\overline{h}\Id|^2=|A-h \Id|^2+n|h-\overline{h}|^2$.
However, using the method in \cite{DT}, Perez \cite{Pz} showed that $\|A-\overline{h}\Id\|_2\leq C\|A-h\Id\|_2$ under the assumption $\Ric\geq 0$.
Cheng-Zhou \cite{CZ} generalized it under the assumption of the lower bound of $\frac{1}{\lambda_1}\Ric$.
Note that Perez also considered the case $\|A-h\Id\|_p$ is small for some $1<p\leq n$ and showed $M$ is Hausdorff close to the standard sphere under the assumption $\Ric\geq 0$ (see \cite[Corollary 2.5, Proposition 3.2]{Pz}).

Combining Cheng-Zhou result and Main Theorem 1, we have the following theorem:
\begin{Me}
Given an integer $n\geq 2$ and positive real numbers $\epsilon>0$, $K>0$, there exists $\delta(n,K,\epsilon)>0$ such that the following property holds.
Let $(M,g)$ be an $n$-dimensional oriented closed Riemannian manifold with $\Ric\geq -K\lambda_1(g)g$, and $\iota \colon (M,g)\to \mathbb{R}^{n+1}$ an isometric immersion.
If $\|A-h \Id\|_2\leq \delta\|h\|_2$,
then we have the following properties:
\begin{itemize}
\item[(i)] $d_{GH}(M,S^n(1/\|h\|_2))\leq \epsilon/\|h\|_2$ and $M$ is diffeomorphic to the $n$-dimensional standard sphere.
\item[(ii)] $d_{H}(\iota(M),S_M)\leq \epsilon/\|h\|_2$, where $d_{H}$ denotes the Hausdorff distance and $S_M$ is defined by
\begin{equation*}
S_M:=\left\{a+\frac{1}{\Vol(M)}\int_M \iota(x) \,d\mu_g(x) \in\mathbb{R}^{n+1}\colon a\in \mathbb{R}^{n+1} \text{ with } |a|=\frac{1}{\|h\|_2} \right\}.
\end{equation*}
\item[(iii)] There exists a measurable subset $E\subset M$ such that $\Vol(M\backslash E)\leq \epsilon \Vol(M)$ and $\iota$ is injective on $E$.
\end{itemize}
\end{Me}
\begin{Rem}
We give several remarks on our assumptions.
\begin{itemize}
\item[(i)] The condition $\Ric\geq -K\lambda_1(g)g$ is homothetically invariant.
\item[(ii)] If $\Ric\geq 0$, then the assumption $\Ric\geq -K\lambda_1(g)g$ is always satisfied for any $K>0$. 
\item[(iii)] If $\Ric\geq -Kg$ and $\diam(M)\leq D$, then we have
$\Ric\geq -C(n,K,D)\lambda_1(g) g$ by the Li-Yau estimate $\lambda_1(\tilde{g})\geq C(n,K,D)>0$ (see \cite[p.116]{SY}).
\item[(iv)] We prove the theorem under slightly weaker assumptions (Theorem \ref{um7}).
In particular, if $\Ric\geq -Kg$ and $\diam(M)\leq D$ or $\Ric\geq 0$, then the assumption ``oriented'' can be removed.
\item[(v)] By the Reilly inequality $\lambda_1(g)\leq n\|h\|_2$ \cite{Rei} (see also \cite[section 3]{AG0}), we have
$\Ric\geq -n K \|h\|_2 g$ if $\Ric\geq -K\lambda_1(g)g$.
We do not know whether we can replace the assumption $\Ric\geq - K\lambda_1(g) g$ to the weaker assumption $\Ric\geq -n K \|h\|_2 g$.
However, once we know that $h$ is almost constant in $L^2$-sense, the assumption $\Ric\geq -n K \|h\|_2 g$ is enough (Theorem \ref{um4.5}).
\end{itemize}
\end{Rem}

The structure of this paper is as follows.

In section 2, we consider a pinching condition on a single function.
Comparing what we get from the Cheeger-Colding segment inequality and what we get from spectral method, we obtain a nice property that a modified function from the originally pinched one is an almost cosine function of the distance as we mentioned in Proposition \ref{thm14}.
We also show the small excess like property, i.e., the converse of the triangle inequality almost holds for a certain triple of points.
This section is the core of our paper.

In section 3, we consider Petersen's method \cite{Pe1}, which gives an approximation map to the standard sphere by using eigenfunctions.
In subsection 3.1, we give some preparations to apply Petersen's method to our problem.
In subsection 3.2, we give a modification of Petersen's argument for our problem.

In section 4, we complete the proof of Main Theorem 1 using Aubry's method \cite{Au}, which gives a pinching condition on the first $(n+1)$-th eigenvalue using a pinching condition on the first $n$-th eigenvalue, where $n$ is the dimension of the manifold.
In the section, we give a modification of Aubry's argument for our problem.
We show the equivalence of our pinching condition on the functions in Main Theorem 1' and the pinching condition on the eigenvalues of a certain Laplacian on the vector bundle $E$, which is a direct sum of a tangent bundle and a rank $1$ trivial bundle in Main Theorem 1.
Taking a wedge product of $n$ eigensections of $E$, 
we give a new almost $(n+1)$-th eigensection through the isomorphism $\bigwedge^n E\cong E$ if the manifold is orientable, and show that the unorientable case cannot occur under our pinching condition by a mapping degree argument.

In section 5, we give an application of Main Theorem 1 to the almost umbilical manifolds.
We show that the almost umbilical manifolds satisfy our pinching condition in Main Theorem 1.

In section 6, we give some examples of collapsing sequences of Riemannian manifolds to show non-continuity of the value of $\lambda_k(\bar{\Delta}^E)$  in the measured Gromov-Hausdorff topology,
and consider the converse of our main theorem for non-collapsing cases.
\begin{sloppypar}
{\bf Acknowledgments}.\ 
I am grateful to my supervisor, Professor Shinichiroh Matsuo, for his advice, also Professor Osamu Kobayashi.
Especially, I would like to thank Professor Erwann Aubry for his valuable suggestions.
I also thank Professor Shouhei Honda for his useful advice and kindly explaining his works to me. Appendix B is based on his suggestion.
Part of this work was done during my stay at the University of C\^{o}te d'Azur. 
I am grateful to the referee for careful reading of the paper and making valuable suggestions.
This work was supported by JSPS Overseas Challenge Program for Young Researchers and by JSPS Research Fellowships for Young Scientists.
\end{sloppypar}


\section{Single function pinching}
In this section, we consider the pinching condition ``$\|\nabla^2 f + fg\|_2$ is small'' for a smooth function, and show that the modified one is an almost cosine function (Proposition \ref{p2i}) and that the Riemannian manifold has the small excess like property (Proposition \ref{p2k}), which asserts that the converse of the triangle inequality almost holds for a certain triple of points, under the pinching condition.
\subsection{Basic definitions}
We first recall some basic definitions and fix our convention.
\begin{Def}[Hausdorff distance]\label{Dhau}
Let $(X,d)$ be a metric space.
For each point $x_0\in X$, subsets $A,B\subset X$ and $r>0$, define
\begin{align*}
d(x_0,A):=&\inf\{d(x_0,a):a\in A\},\\
B_{r}(x_0):=&\{x\in X: d(x,x_0)<r\},\\
B_{r}(A):=&\{x\in X:d(x,A)<r\},\\
d_{H,d}(A,B):=&\inf\{\epsilon>0:A\subset B_{\epsilon}(B) \text{ and } B\subset B_{\epsilon}(A)\}
\end{align*}
We call $d_{H,d}$ the Hausdorff distance.
\end{Def}
If the distance $d$ is clear from the context, we simply write $d_H$ for the Hausdorff distance.
The Hausdorff distance defines a metric on the collection of compact subsets of $X$.
\begin{Def}[Gromov-Hausdorff distance]\label{DGH}
Let $(X,d_X),(Y,d_Y)$ be metric spaces.
Define
\begin{align*}
d_{GH}(X,Y):=\inf\Big\{d_{H,d}(X,Y): &\text{ $d$ is a metric on $X\coprod Y$ such that}\\
&\qquad\qquad\quad\text{$d|_X=d_X$ and $d|_Y=d_Y$}\Big\}.
\end{align*}
\end{Def}
The Gromov-Hausdorff distance defines a metric on the set of isometry classes of compact metric spaces (see \cite[Proposition 11.1.3]{Pe3}).

\begin{Def}[$\epsilon$-Hausdorff approximation map]\label{hap}
Let $(X,d_X),(Y,d_Y)$ be metric spaces.
We say that a map $f\colon X\to Y$ is an $\epsilon$-Hausdorff approximation map for an $\epsilon>0$ if the following two conditions hold.
\begin{itemize}
\item[(i)] For all $a,b\in X$, we have $|d_X(a,b)-d_Y(f(a),f(b))|< \epsilon$,
\item[(ii)] $f(X)$ is $\epsilon$-dense in $Y$, i.e., for all $y\in Y$, there exists $x\in X$ with $d_Y(f(x),y)< \epsilon$.
\end{itemize}
\end{Def}
If there exists an  $\epsilon$-Hausdorff approximation map $f\colon X\to Y$, then we can show that $d_{GH}(X,Y)\leq 3\epsilon/2$ by considering the following metric $d$ on $X\coprod Y$.
\begin{empheq}[left={d(a,b)=\empheqlbrace}]{align*}
&\qquad d_X(a,b)&& (a,b\in X),\\
&\,\frac{\epsilon}{2} +\inf_{x\in X}(d_X(a,x)+d_Y(f(x),b))&&(a\in X,\,b\in Y),\\
&\qquad d_Y(a,b)&&(a,b\in Y).
\end{empheq}
If $d_{GH}(X,Y)< \epsilon$, then there exists a $2\epsilon$-Hausdorff approximation map from $X$ to $Y$.

Let $C(u_1,\ldots,u_l)>0$ denotes a positive function depending only on the numbers $u_1,\ldots,u_l$.
For a set $X$, $\Card X$ denotes a cardinal number of $X$.

Let $(M,g)$ be a closed Riemannian manifold.
For any $p\geq 1$, we use the normalized $L^p$-norm:
\begin{equation*}
\|f\|_p^p:=\frac{1}{\Vol(M)}\int_M |f|^p\,d\mu_g,
\end{equation*}
and $\|f\|_{\infty}:=\mathop{\mathrm{sup~ess}}\limits_{x\in M}|f(x)|$ for a measurable function on $M$. We also use this notation for tensors.
We have $\|f\|_p\leq \|f\|_q$ for any $p\leq q \leq \infty$.
Let $\nabla$ denotes the Levi-Civita connection.
Throughout in this paper, 
 $0=\lambda_0(g)< \lambda_1(g) \leq \lambda_2(g) \leq\cdots \to \infty$
denotes the eigenvalues of the Laplacian $\Delta=-\sum_{i,j}g^{ij}\nabla_i \nabla_j$ acting on functions, and $\{\phi_i\}$ denotes the complete orthonormal system of eigenfunctions in $L^2(M)$:
\begin{equation*}
\Delta \phi_i=\lambda_i \phi_i,\quad \|\phi_i\|_2=1.
\end{equation*}
We sometimes identify $TM$ and $T^\ast M$ using the metric $g$.

Given points $x,y\in M$, $\gamma_{x,y}$ denotes one of minimal geodesics with unit speed such that $\gamma_{x,y}(0)=x$ and $\gamma_{x,y}(d(x,y))=y$.
For any $x\in M$ and $u\in T_x M$ with $|u|=1$, put
$$
t(u):=\sup\{t\in\mathbb{R}_{>0}: d(x,\exp_x(t u))=t\},
$$
and define the interior set $I_x\subset M$  at $x$ (see also \cite[p.104]{Sa}) by
$$
I_x:=\{\exp_x (t u): u\in T_x M \text{ with $|u|=1$ and } 0\leq t< t(u)\}.
$$
Then, $I_x$ is open and $\Vol(M\setminus I_x)=0$ \cite[III Lemma 4.4]{Sa}.
For any $y\in I_x\setminus \{x\}$, the minimal geodesic $\gamma_{x,y}$ is uniquely determined. The function $d(x,\cdot)\colon M\to \mathbb{R}$ is differentiable in $I_x\setminus\{x\}$ and $\nabla d(x,\cdot)(y)=\dot{\gamma}_{x,y}(d(x,y))$ holds for any $y\in I_x\setminus \{x\}$ \cite[III Proposition 4.8]{Sa}.

We define operators $\nabla^\ast \colon \Gamma(T^\ast M\otimes T^\ast M)\to \Gamma(T^\ast M)$ and $d^\ast \colon \Gamma(\bigwedge^k T^\ast M)\to \Gamma(\bigwedge^{k-1}T^\ast M)$ by
\begin{align*}
\nabla^\ast(\alpha\otimes \beta):&=-\tr_{T^\ast M} \nabla(\alpha\otimes \beta)
=-\sum_{i=1}^n \left(\nabla_{e_i}\alpha\right)(e_i)-\sum_{i=1}^n\alpha(e_i)\nabla_{e_i}\beta.\\
d^\ast \omega:&=-\sum_{i=1}^n\iota(e_i)\nabla_{e_i}\omega
\end{align*}
for all $\alpha\otimes\beta\in \Gamma(T^\ast M\otimes T^\ast M)$ and $\omega\in\Gamma(\bigwedge^k T^\ast M)$, where $\iota$ denotes the inner product, $n=\dim M$ and $\{e_1,\ldots,e_n\}$ is an orthonormal basis of $TM$.
If $M$ is closed, then we have
\begin{align*}
\int_M \langle T,\nabla \alpha\rangle\,d\mu_g&=\int_M \langle \nabla^\ast T, \alpha\rangle\,d\mu_g,\\
\int_M \langle \omega,d\eta \rangle\,d\mu_g&=\int_M \langle d^\ast \omega, \eta \rangle\,d\mu_g
\end{align*}
for all $T\in\Gamma(T^\ast M\otimes T^\ast M)$, $\alpha\in\Gamma(T^\ast M)$, $\omega\in\Gamma(\bigwedge^k T^\ast M)$ and $\eta\in\Gamma(\bigwedge^{k-1} T^\ast M)$ by the divergence formula.
The Hodge Laplacian $\Delta\colon \Gamma(\bigwedge^k T^\ast M)\to\Gamma(\bigwedge^k T^\ast M)$ is defined by
$$
\Delta:=d d^\ast +d^\ast d.
$$
We frequently use the following formula:
\begin{Thm}[Bochner formula]
Let $(M,g)$ be a Riemannian manifold.
Then, we have
\begin{equation*}
\Delta \omega =\nabla^\ast \nabla \omega + \Ric(\omega^\sharp,\cdot)
\end{equation*}
for all $\omega\in\Gamma(T^\ast M)$,
where $\omega^\sharp$ denotes a vector field such that $g(\omega^\sharp,Y)=\omega(Y)$ for any vector field $Y$.
\end{Thm}
By the Bochner formula, we immediately have
$$
\int_M(\Delta f)^2 \,d\mu_g=\int_M |\nabla^2 f|^2\,d\mu_g-\int_M\Ric(\nabla f,\nabla f) \,d\mu_g
$$
for all $f\in C^\infty(M)$.

Finally, we list some important notation.
Some of them will be defined in the later sections.
\begin{itemize}
\item $d_H$ denotes the Hausdorff distance (see Definition \ref{Dhau}).
\item $d_{GH}$ denotes the Gromov-Hausdorff distance (see Definition \ref{DGH}).
\item $B_r(x)$ denotes the open ball centered at $x\in X$ with radius $r>0$ for a metric space $(X,d)$.
\item $S^n(r)$ denotes the $n$-dimensional standard sphere of radius $r$.
\item $S^n:=S^n(1)$.
\item For any closed Riemannian manifold $(M,g)$:
\begin{itemize}
\item $d$ denotes the Riemannian distance function.
\item $\Ric$ denotes the Ricci curvature.
\item $\Scal$ denotes the scalar curvature.
\item $\diam$ denotes the diameter.
\item $\Vol$ or $\mu_g$ denotes the Riemannian volume measure.
\item$\|\cdot\|_p$ denotes the normalized $L^p$-norm for each $p\geq 1$, which is defined by
\begin{equation*}
\|f\|_p^p:=\frac{1}{\Vol(M)}\int_M |f|^p\,d\mu_g
\end{equation*}
for any measurable function $f$ on $M$.
\item $\|f\|_{\infty}$ denotes the essential sup of $|f|$ for any measurable function $f$ on $M$.
\item $\nabla$ denotes the Levi-Civita connection.
\item $\nabla^2$ denotes the Hessian for functions.
\item $\Delta\colon \Gamma(\bigwedge^k T^\ast M)\to\Gamma(\bigwedge^k T^\ast M)$ denotes the Hodge Laplacian defined by $\Delta:=d d^\ast +d^\ast d$.
We frequently use the Laplacian acting on functions.
Note that $\Delta=-\tr_g \nabla^2$ holds for functions under our sign convention. 
\item $0=\lambda_0(g)< \lambda_1(g) \leq \lambda_2(g) \leq\cdots \to \infty$
denotes the eigenvalues of the Laplacian acting on functions.
\item $\{\phi_i\}$ denotes the complete orthonormal system of eigenfunctions in $L^2(M)$.
\item $\gamma_{x,y}\colon [0,d(x,y)]\to M$ denotes one of minimal geodesics with unit speed such that $\gamma_{x,y}(0)=x$ and $\gamma_{x,y}(d(x,y))=y$ for any $x,y\in M$.
\item $I_x$ denotes the interior set at $x\in M$. We have $\Vol(M\setminus I_x)=0$. We have that $\gamma_{x,y}$ is uniquely determined and $\nabla d(x,\cdot)=\dot{\gamma}_{x,y}(d(x,y))$ holds for any $y\in I_x\setminus\{x\}$.
\item $T_{(n,\delta)}\colon L^2(M)\to L^2(M)$ denotes the projection onto a subspace $\Span\{\phi_j:j\in \mathbb{N}\text{ with }n-\sqrt{\delta}\leq\lambda_j\leq n+\sqrt{\delta}\}\subset C^\infty(M)$ for each $n\in \mathbb{Z}_{\geq 2}$ and $\delta>0$. We call $T_{(n,\delta)}\colon L^2(M)\to L^2(M)$ the $(n,\delta)$-projection (see Definition \ref{Def1} below).
\item We say that a subspace $V$ in $C^\infty(M)$ satisfies the $\delta$-pinching condition for some $\delta>0$ if the pinching condition $\|\nabla^2 f +f g\|_2\leq\delta\|f\|_2$ holds for all $f\in V$ (see Definition \ref{dpin} below).
\item $E$ denotes the vector bundle $TM\oplus \mathbb{R}e$, where $\mathbb{R}e$ denotes the rank $1$ trivial bundle.
\item $\nabla^E$ denotes the connection on $E$ defined by $\nabla^E_Y (X+f e):=\nabla_Y X+ f Y+(Y f-g(Y,X))e$ for any $X,Y\in \Gamma(TM)$ and $f\in C^\infty(M)$.
\item $\bar{\Delta}^E$ denotes the connection Laplacian acting on $\Gamma(E)$ (see Definition \ref{def1} below).
\item $0\leq \lambda_1(\bar{\Delta}^E) \leq \lambda_2(\bar{\Delta}^E) \leq\cdots \to \infty$ denotes the eigenvalues of the connection Laplacian $\bar{\Delta}^E$ acting on $\Gamma(E)$.
\end{itemize}
\end{itemize}
Note that the lowest eigenvalue of the Laplacian $\Delta$ acting on function is always equal to $0$, and so we start counting the eigenvalues of it from $i=0$.
This is not the case with the connection Laplacian $\bar{\Delta}^E$ acting on $\Gamma(E)$, and so we start counting the eigenvalues of it from $i=1$.

\subsection{Modification of a pinched function}
In this subsection, we modify a function that satisfies our pinching condition to a  function that is easy to handle from the point of view of spectral geometry.
More precisely, we show that if a function $f\in C^\infty(M)$ satisfies our pinching condition $\|\nabla^2 f + f g\|_2\leq \delta\|f\|_2$, then
the projection of $f$ onto a subspace $\Span\{\phi_j: j\in \mathbb{N}\text{ with }n-\sqrt{\delta}\leq\lambda_j\leq n+\sqrt{\delta}\}\subset C^\infty(M)$ also satisfies our pinching condition.

We first consider the following unnormalized condition.
\begin{Lem}\label{p2a}
Let $(M,g)$ be an $n$-dimensional closed Riemannian manifold.
Suppose that there exist real numbers $\delta>0$, $c\in \mathbb{R}$ and a non-constant smooth function $f\in C^\infty(M)$ with $\|\nabla^2 f+ c f g\|_2\leq \delta\|\Delta f\|_2$.
Then, we have the following properties.
\begin{itemize}
\item[(i)] We have $\|\Delta f- n c f\|_2\leq \sqrt{n}\delta \|\Delta f\|_2$.
\item[(ii)] Set $f_0:=f-\int_M f \,d\mu_g/\Vol(M)$. Then, we have $f_0\neq 0$ and $\|\nabla^2 f_0+ c f_0 g\|_2\leq \delta\|\Delta f_0\|_2$, and so $\|\Delta f_0- n c f_0\|_2\leq \sqrt{n}\delta \|\Delta f_0\|_2$.
\item[(iii)] If $\delta\leq \frac{1}{2\sqrt{n}}$, then $c>0$ and $\lambda_1(g)\leq n c(1+2\sqrt{n}\delta)$ hold.
\item[(iv)]  For positive real numbers $K>0$ and $D>0$, there exists a positive constant $C_{\Ca}(n,K,D)>0$ such that the following property holds.
If $\Ric\geq -Kg$, $\diam(M)\leq D$ and $\delta\leq  \frac{1}{2\sqrt{n}}$, then $c\geq C_{\Ca}(n,K,D)$ holds.
\item[(v)] For a positive real number $K>0$, there exists a positive constant $C_{\Cb}(n,K)>0$ such that if $\delta\leq \frac{1}{2\sqrt{n}}$ and $\Ric\leq Kg$, then $c\leq C_{\Cb}(n,K)$ holds.
\end{itemize}
\end{Lem}
\begin{proof}
For any smooth function $F\in C^\infty(M)$, we have
\begin{equation}\label{1a}
\begin{split}
&|\nabla^2 F+ c F g|^2=|\nabla^2 F|^2- 2c F\Delta F+nc^2 |F|^2\\
=& \frac{1}{n} |\Delta F|^2 - 2c F \Delta F +n c^2 |F|^2+\left|\nabla^2 F+\frac{\Delta F}{n}g\right|^2\\
= &\frac{1}{n} |\Delta F- n c F|^2+\left|\nabla^2 F+\frac{\Delta F}{n}g\right|^2.
\end{split}
\end{equation}
This implies (i).

Since $f$ is non-constant, we have $f_0\neq 0$.
By (\ref{1a}), we get
\begin{equation*}
\begin{split}
\|\nabla^2 f_0 +c f_0 g\|_2^2
=&\|\nabla^2 f +c f g\|_2^2-n c^2\left(\frac{1}{\Vol(M)}\int_M f\,d\mu_g\right)^2\\
\leq&\delta^2 \|\Delta f\|_2^2=\delta^2 \|\Delta f_0\|_2^2.
\end{split}
\end{equation*}
This implies (ii).

By (\ref{1a}) and (ii), we get $\delta^2\|\Delta f_0\|_2^2\geq \frac{1}{n}\|\Delta f_0\|_2^2-2c\|\nabla f_0\|_2^2$.
Thus, we get $c>0$ if $\delta\leq \frac{1}{2\sqrt{n}}$.
By (ii), we have
\begin{equation}\label{1b}
(1-\sqrt{n}\delta)\|\Delta f_0\|_2\leq n c \|f_0\|_2.
\end{equation}
Thus, we get $\lambda_1\|f_0\|_2\leq \|\Delta f_0\|_2\leq n c (1+2\sqrt{n}\delta)\|f_0\|_2$. This implies (iii).

We next prove (iv).
Assume that  $\Ric\geq -Kg$, $\diam(M)\leq D$ and $\delta\leq \frac{1}{2\sqrt{n}}$.
By the Li-Yau estimate \cite[p.116]{SY}, we have
$\lambda_1\geq C(n,K,D)>0$.
Thus, by (iii) we get $c\geq \frac{\lambda_1}{2n}\geq C_1(n,K,D)$.
This implies (iv).

Finally, we prove (v). 
Assume that $\Ric\leq Kg$ and $\delta\leq \frac{1}{2\sqrt{n}}$ (we do not assume $\Ric\geq -K g$).
By (ii), we have $n c \|f_0\|_2\leq 2\|\Delta f_0\|_2$.
By the Bochner formula and (\ref{1a}), we get
\begin{equation*}
\begin{split}
\delta^2\|\Delta f_0\|_2^2\geq\|\nabla^2 f_0+\frac{\Delta f_0}{n}g\|_2^2
=&\|\nabla^2 f_0\|_2^2 -\frac{1}{n}\|\Delta f_0\|_2^2\\
=&\frac{n-1}{n}\|\Delta f_0\|_2^2-\frac{1}{\Vol(M)}\int_M \Ric(\nabla f_0,\nabla f_0)\,d\mu_g\\
\geq&\frac{n-1}{n}\|\Delta f_0\|_2^2-K\|\nabla f_0\|_2^2\\
\geq&\frac{n-1}{n}\|\Delta f_0\|_2^2-K\|\Delta f_0\|_2 \|f_0\|_2\\
\geq&\left(\frac{n-1}{n}-\frac{2K}{n c}\right)\|\Delta f_0\|_2^2.
\end{split}
\end{equation*}
Thus, we get $c\leq C_{\Cb}(n,K)$.
\end{proof}

We immediately have the following corollary.
\begin{Cor}\label{p2b}
Take an integer $n\geq 2$ and positive real numbers $K>0$, $D>0$, $\mu>0$ and $0<\delta\leq \frac{1}{2\sqrt{n}}$.
Let $(M,g)$ be an $n$-dimensional closed Riemannian manifold with $\Ric \geq-K g$ and $\diam(M)\leq D$, and let $c$ be a real constant with $c\leq \mu$.
Suppose that there exists a non-constant function $f\in C^\infty(M)$ with $\|\nabla^2 f+c f g\|_2\leq\delta \|\Delta f\|_2$.
Then, $c\geq C_{\Ca}>0$ and the metric $\tilde{g}=cg$ satisfies $\Ric_{\tilde{g}} \geq - K/C_1 \tilde{g}$, $\diam(M,\tilde{g})\leq \mu^\frac{1}{2} D$ and $\|(\nabla^{\tilde{g}})^2 f+f\tilde{g}\|_{L^2(M,\tilde{g})}\leq \delta \|\Delta^{\tilde{g}} f\|_{L^2(M,\tilde{g})}$.
\end{Cor}

By Corollary \ref{p2b}, we only consider the case when $c=1$. 
By Lemma \ref{p2a} (v), we do not need to assume $c\leq \mu$ if we assume $\Ric \leq Kg$.
Note that if a non-zero function $f$ satisfies $\|\nabla^2 f+ f g\|_2\leq \delta\|\Delta f\|_2$ for some $\delta>0$, then $f$ is non-constant.

In the following, we sometimes normalize the function $f_1$ defined in Lemma \ref{p2c} so that $\|f_1\|_2^2=1/(n+1)$.
Thus, the condition $\|\nabla^2 f + fg\|_2\leq\delta \|f\|_2$ seems more natural.
Since $\sqrt{n}\|\nabla^2 f + fg\|_2\geq \|\Delta f -n f\|_2$ (see (\ref{1a})), we get the following lemma immediately.
\begin{Lem}\label{add}
Let $(M,g)$ be an $n$-dimensional closed Riemannian manifold.
Take arbitrary $0<\delta\leq \frac{1}{2\sqrt{n}}$.
Then, we have the following properties.
\begin{itemize}
\item[(i)] If a smooth function $f\in C^\infty (M)$ satisfies $\|\nabla^2 f +f g\|_2\leq \delta \|\Delta f\|_2$, then $\|\nabla^2 f +f g\|_2\leq 2n \delta \|f\|_2$.
\item[(ii)]  If a smooth function $f\in C^\infty (M)$ satisfies $\|\nabla^2 f +f g\|_2\leq \delta \|f\|_2$, then $\|\nabla^2 f +f g\|_2\leq \delta \|\Delta f\|_2$.
\end{itemize}
\end{Lem}

Let us consider the projection of our function onto a subspace $\Span\{\phi_j:j\in \mathbb{N}\text{ with }n-\sqrt{\delta}\leq\lambda_j\leq n+\sqrt{\delta}\}\subset C^\infty(M)$.

\begin{Lem}\label{p2c}
Given an integer $n\geq 2$ and positive real numbers $K>0$ and $D>0$, there exist positive constants $C(n,K,D)>0$ and $N_1(n,K,D)>0$ such that the following properties hold.
Take a positive real number $0<\delta\leq\frac{1}{8n^3}$.
Let $(M,g)$ be an $n$-dimensional closed Riemannian manifold with $\Ric\geq -Kg$ and $\diam(M)\leq D$.
Suppose that there exists a non-zero smooth function $f\in C^\infty(M)$ with $\|\nabla^2 f+ f g\|_2\leq \delta\|f\|_2$.
Then, we have the following properties.
\begin{itemize}
\item[(i)] Decompose $f$ as $f=\sum_{j=0}^\infty \alpha_j\phi_j$ $(\alpha_j\in \mathbb{R})$. Set
\begin{equation*}
f_1:=\sum_{n-\sqrt{\delta}\leq \lambda_j\leq n+\sqrt{\delta}} \alpha_j\phi_j.
\end{equation*}
Then, we have $f_1\neq0$ and
\begin{equation*}
\|\nabla^2 f_1+ f_1g\|_2\leq C\sqrt{\delta}\|f_1\|_2,\,
\|\Delta f_1 - n  f_1\|_2\leq C \delta\|f_1\|_2. 
\end{equation*}
\item[(ii)]  We normalize $f_1$ so that $\|f_1\|_2^2=\frac{1}{n+1}$.
Then, we have
\begin{equation*}
\|\nabla f_1\|_{\infty}\leq C,\,\|f_1\|_{\infty} \leq C,
\end{equation*}
\begin{equation*}
\|\nabla \phi_j\|_{\infty}\leq C, \,\|\phi_j\|_{\infty} \leq C
\end{equation*}
for all $j$ with $n-\sqrt{\delta}\leq \lambda_j\leq n+\sqrt{\delta}$, and
\begin{equation*}
\Card\{j\in \mathbb{Z}_{>0}: n-\sqrt{\delta}\leq \lambda_j\leq n+\sqrt{\delta}\}\leq N_1.
\end{equation*}
\end{itemize}
\end{Lem}
\begin{proof}
By Lemma \ref{add}, we have $\|\nabla^2 f+ f g\|_2\leq \delta\|\Delta f\|_2$.
Since $\frac{1}{8n^3} \leq \frac{1}{2\sqrt{n}}$,
we can use the result of Lemma \ref{p2a} (i), (ii), (iii) and (iv).
Set
\begin{align*}
f_1:=&\sum_{n-\sqrt{\delta}\leq \lambda_j\leq n+\sqrt{\delta}} \alpha_j\phi_j,\quad
&&f_2:=f_0-f_1=\sum_{0<\lambda_j\leq n-\sqrt{\delta}} \alpha_j\phi_j+\sum_{\lambda_j\geq n+\sqrt{\delta}} \alpha_j\phi_j, \\
f_3:=& \sum_{\lambda_j>2 n } \alpha_j\phi_j,\quad
&&f_4:=f_2-f_3=\sum_{0<\lambda_j\leq n-\sqrt{\delta}} \alpha_j\phi_j+\sum_{n+\sqrt{\delta}\leq\lambda_j\leq 2n} \alpha_j\phi_j.
\end{align*}
Then, by the definition of $f_2$, Lemma \ref{p2a} (ii) and (\ref{1b}) , we have
\begin{equation*}
\begin{split}
\sqrt{\delta} \|f_2\|_2\leq \|\Delta f_2-n f_2\|_2\leq \|\Delta f_0-n f_0\|_2\leq &\sqrt{n}\delta\|\Delta f_0\|_2\\
\leq &2n \sqrt{n} \delta \|f_0\|_2.
\end{split}
\end{equation*}
Thus, we get 
\begin{equation}\label{1c}
\|f_2\|_2\leq \sqrt{n \delta}\|\Delta f_0\|_2\leq 2n\sqrt{n\delta}\|f_0\|_2.
\end{equation}
Therefore, we have $\|f_0\|_2^2=\| f_1\|_2^2+\|f_2\|_2^2\leq \|f_1\|_2^2+4n^3\delta \| f_0\|_2^2$, and so we get
\begin{equation}\label{1e}
\| f_0\|_2^2\leq (1+8n^3 \delta)\| f_1\|_2^2\leq 2\| f_1\|_2^2.
\end{equation}
In particular, we get $f_1\neq0$.
Since we have
\begin{equation*}
\|\Delta f_3 -n f_3\|_2^2
=\sum_{\lambda_j>2 n}\alpha_j^2 (\lambda_j-n )^2
\geq \sum_{\lambda_j>2 n }\alpha_j^2 \frac{1}{4}\lambda_j^2=\frac{1}{4}\|\Delta f_3\|_2^2,
\end{equation*}
and (\ref{1c}), we get
\begin{equation}\label{1d1}
\begin{split}
\|\Delta f_2\|_2^2=\|\Delta f_3\|_2^2+\|\Delta f_4\|_2^2\leq& 4\|\Delta f_3- n f_3\|_2^2+ (2 n)^2 \|f_4\|_2^2\\
\leq &4\|\Delta f_0- n f_0\|_2^2+ (2 n)^2 \|f_2\|_2^2\\
\leq &C \delta \| f_0\|_2^2.
\end{split}
\end{equation}
By (\ref{1c}) and (\ref{1d1}), we have 
\begin{equation}\label{1f}
\|\nabla f_2\|_2^2\leq \|\Delta f_2\|_2 \|f_2\|_2\leq C\delta \|f_0\|_2^2.
\end{equation}
By the Bochner formula, (\ref{1d1}) and (\ref{1f}), we get
\begin{equation}\label{1g}
\begin{split}
\|\nabla^2 f_2\|_2^2=&\|\Delta f_2\|_2^2 -\frac{1}{\Vol(M)} \int_M \Ric(\nabla f_2,\nabla f_2)\,d\mu_g\\
\leq  &\|\Delta f_2\|_2^2+K\|\nabla f_2\|_2^2\leq C\delta \|f_0\|_2^2.
\end{split}
\end{equation}
By (\ref{1c}), (\ref{1e}) and (\ref{1g}), we have
\begin{equation*}
\begin{split}
\|\nabla^2 f_1+ f_1 g\|_2\leq &
\|\nabla^2 f_0+f_0 g\|_2+ \|\nabla^2 f_2\|_2+\sqrt{n}  \|f_2\|_2\\
\leq &C \sqrt{\delta} \|f_0\|_2
\leq C\sqrt{\delta}\|f_1\|_2,
\end{split}
\end{equation*}
and
\begin{equation*}
\begin{split}
\|\Delta f_1-n f_1\|_2\leq &
\|\Delta f_0- n f_0\|_2\\
\leq &\sqrt{n}\delta \|\Delta f_0\|_2
\leq C\delta\| f_1\|_2.
\end{split}
\end{equation*}
Thus, we get (i).

We next prove (ii).
By the gradient estimate for eigenfunctions  \cite[Theorem 7.3]{Pe1},
we have $\|\nabla \phi_j\|_{\infty}\leq  C(n,K,D)$ for all $j$ with $\lambda_j\leq n+\sqrt{\delta}(\leq n+1)$. Since $\phi_j$ takes $0$ at some point, we have $\|\phi_j\|_{\infty}\leq C$.
By the Gromov estimate of the eigenvalues \cite[Appendix C]{Gr}, for each $j$, we have 
\begin{equation*}
\lambda_j\geq  C(n,K,D) j^\frac{2}{n}>0.
\end{equation*}
Thus, for $j$ with $\lambda_j\leq n+\sqrt{\delta}$, we have
$j\leq \left(\frac{n+1}{C}\right)^\frac{n}{2}=:N_1(n,K,D)$.
Therefore, if we normalize $f_1$ so that $\|f_1\|_2^2=\frac{1}{n+1}$, we get
\begin{equation*}
\begin{split}
\|\nabla f_1\|_{\infty}\leq &\sum_{n-\sqrt{\delta}\leq \lambda_j\leq n+\sqrt{\delta}} |\alpha_j|\|\nabla \phi_j\|_{\infty}\\
\leq & C \left(\sum_{n-\sqrt{\delta}\leq \lambda_j\leq n+\sqrt{\delta}} \alpha_j^2\right)^\frac{1}{2}N_1^\frac{1}{2}\leq C
\end{split}
\end{equation*}
and $\|f_1\|_{\infty}\leq C$.
Thus, we get (ii).
\end{proof}

We give the following definition based on Lemma \ref{p2c}.
\begin{Def}\label{Def1}
Take an integer $n\geq 2$.
For an $n$-dimensional closed Riemannian manifold $(M,g)$ and a $\delta>0$, we define the $(n,\delta)$-projection $T_{(n,\delta)}\colon L^2(M)\to L^2(M)$ as follows.
For any function $f\in L^2(M)$ such that $f=\sum_{j=0}^\infty \alpha_j\phi_j$ ($\alpha_j\in\mathbb{R}$), we define 
\begin{equation*}T_{(n,\delta)}(f)= \sum_{n-\sqrt{\delta}\leq \lambda_j\leq n+\sqrt{\delta}} \alpha_j\phi_j.
\end{equation*}
\end{Def}

\subsection{Gradient estimate}
In this subsection, we modify the Cheng-Yau estimate to apply it to our function $f_1$ in Lemma \ref{p2c}.
Although we do not have $L^\infty$ estimate for the Hessian $\nabla^2 f_1$,
Lemma \ref{p2d} below enable us to give an upper bound of the norm of gradient $|\nabla f_1|$ around the maximum point of $f_1$.
The proof of the following lemma is similar to that of \cite[Theorem 7.1]{Ch}.
\begin{Lem}\label{p2d}
Take an integer $n\geq 2$ and positive real numbers $K>0$, $D>0$ and $0<\epsilon_1 \leq1$. Then, there exists a positive constant $ C(n,K,D)>0$ such that the following property holds.
Let $(M,g)$ be an $n$-dimensional closed Riemannian manifold with $\Ric\geq-Kg$ and $\diam(M)\leq D$. Take an open subset $A\subset M$ and a function $f\in C^\infty(M)$ with $\|f\|_2^2=\frac{1}{n+1}$ of the form
\begin{equation*}
f=\sum_{\lambda_j\leq n+1} \alpha_j\phi_j.
\end{equation*}
Suppose that at some point $p\in A$, $f(p)=\sup_{x\in A} f(x)$,
and $\sup_{x\in A} (f(p)-f(x)+\epsilon_1)^{-2}|\nabla f|^2(x)$ is attained at some point in $A$.
Then, for all $x\in A$, we have 
\begin{equation*}
|\nabla f|^2(x)\leq \frac{C}{\epsilon_1}\left(f(p)-f(x)+\epsilon_1\right)^2.
\end{equation*}
\end{Lem}

\begin{proof}
Note that there exist  positive constants $N_2(n,K,D)>0$ and $ C(n,K,D)>0$ such that
\begin{equation*}
\Card \{j: \lambda_j \leq n+1\}\leq N_2
\end{equation*}
and $\|\nabla\phi_j\|_{\infty}\leq  C$
for all $j$ with $\lambda_j \leq n+1$ by the gradient estimate for eigenfunctions  \cite[Theorem 7.3]{Pe1} and Gromov's eigenvalue estimate \cite[Appendix C]{Gr}.

Set
$h(x):=f(p)-f(x)+\epsilon_1(\geq \epsilon_1)$ and $v(x):=\log h(x)$ for any $x\in A$.
Then, we have
\begin{equation*}
\nabla v=\frac{\nabla h}{h}=\frac{-\nabla f}{h},\quad \Delta v=\frac{\Delta h}{h}+|\nabla v|^2, \quad |\nabla v|^2=\frac{|\nabla f|^2}{(f(p)-f+\epsilon_1)^2}.
\end{equation*}
By the Bochner formula, we get
\begin{equation}\label{1i}
\begin{split}
\frac{1}{2}\Delta |\nabla v|^2&=\langle\nabla \Delta v,\nabla v\rangle-\Ric(\nabla v,\nabla v)-|\nabla^2 v+\frac{\Delta v}{n} g|^2-\frac{1}{n}|\Delta v|^2\\
&\leq \langle\nabla \Delta v,\nabla v\rangle+K |\nabla v|^2-\frac{1}{n}\left(\frac{\Delta h}{h}+|\nabla v|^2\right)^2.
\end{split}
\end{equation}
We have
\begin{equation}\label{1j}
\begin{split}
 \langle\nabla \Delta v,\nabla v\rangle=& \langle\nabla \frac{\Delta h}{h},\nabla v\rangle+ \langle\nabla |\nabla v|^2,\nabla v\rangle\\
=& \langle\frac{\nabla \Delta h}{h},\nabla v\rangle-\frac{\Delta h}{h}|\nabla v|^2+ \langle\nabla |\nabla v|^2,\nabla v\rangle.
\end{split}
\end{equation}
We estimate the first and second terms of (\ref{1j}).
Since we have 
\begin{equation*}
\frac{\Delta h}{h}=-\frac{\Delta f}{h}=-\frac{1}{h}\sum_{\lambda_j\leq n+1}\lambda_j \alpha_j\phi_j,
\end{equation*}
we get
\begin{equation}\label{1k}
\begin{split}
\left|\frac{\Delta h}{h}\right|
\leq& \frac{n+1}{\epsilon_1}\sum_{\lambda_j\leq n+1}|\alpha_j||\phi_j|
\leq \frac{C}{\epsilon_1}.
\end{split}
\end{equation}
Since we have 
\begin{equation*}
\begin{split}
\left\langle\frac{\nabla \Delta h}{h},\nabla v\right\rangle=-\frac{1}{h}\sum_{\lambda_j\leq n+1}\lambda_j \alpha_j \langle \nabla \phi_j,\nabla v\rangle,
\end{split}
\end{equation*}
we get 
\begin{equation}\label{1l}
\begin{split}
\left|\left\langle\frac{\nabla \Delta h}{h},\nabla v\right\rangle\right|
\leq \frac{C}{\epsilon_1}|\nabla v|.
\end{split}
\end{equation}
By (\ref{1k}) and (\ref{1l}), we get
\begin{equation}\label{1m}
\left|\left\langle\frac{\nabla \Delta h}{h},\nabla v\right\rangle-\frac{\Delta h}{h}|\nabla v|^2\right|
\leq \frac{ C}{\epsilon_1}|\nabla v|+ \frac{ C}{\epsilon_1}|\nabla v|^2.
\end{equation}

By the assumption, $|\nabla v|^2$ takes its maximum value at some point $q\in A$.
We have $\langle\nabla |\nabla v|^2,\nabla v\rangle(q)=0$.
By (\ref{1j}) and (\ref{1m}), we get
\begin{equation*}
|\langle \nabla \Delta v,\nabla v\rangle|(q)\leq \frac{ C}{\epsilon_1}|\nabla v|(q)+ \frac{ C}{\epsilon_1}|\nabla v|^2(q),
\end{equation*}
and so, by (\ref{1i}) and the maximum principle,
\begin{equation*}
0\leq \frac{1}{2}\Delta |\nabla v|^2(q)\leq \frac{ C}{\epsilon_1} |\nabla v|^2(q) +\frac{ C}{\epsilon_1}|\nabla v|(q)-\frac{1}{n}\left(\frac{\Delta h}{h}+|\nabla v|^2\right)^2(q).
\end{equation*}

We consider the following cases.

\begin{itemize}
\item[ (a)] $|\nabla v|^2(q)\geq -2\frac{\Delta h}{h}(q)$.
\item[ (b)] $|\nabla v|^2(q)< -2\frac{\Delta h}{h}(q)$.
\end{itemize}

We first consider the case (a).
Since we have $\frac{\Delta h}{h}(q)+|\nabla v|^2(q)\geq \frac{1}{2}|\nabla v|^2(q)$, we get
\begin{equation*}
|\nabla v|^4(q) \leq  \frac{ C}{\epsilon_1} \left(|\nabla v|^2(q) +|\nabla v|(q)\right).
\end{equation*}
We further consider the following two cases.

\begin{itemize}
\item[(a-1)] $ |\nabla v|^2(q) \leq |\nabla v|(q)$.
\item[(a-2)] $ |\nabla v|^2(q) >|\nabla v|(q)$.
\end{itemize}
For both cases, we get
$|\nabla v|^2(q) \leq\frac{C}{\epsilon_1}$.

We next consider the case (b).
By (\ref{1k}), we get  
\begin{equation*}
|\nabla v|^2(q)\leq -2\frac{\Delta h}{h}(q)\leq \frac{C}{\epsilon_1}.
\end{equation*}
Thus, we get
$|\nabla v|^2(q)\leq \frac{ C}{\epsilon_1}$ for every cases.
Since we have $|\nabla v|^2(x)\leq |\nabla v|^2(q)$ for all $x\in A$, we get the lemma.
\end{proof}


\subsection{Segment inequality}

In this subsection, we recall the segment inequality \cite[Theorem 2.1]{CC2} and prove its easy consequence.
Recall that $\gamma_{y_1,y_2}$ denotes one of minimal geodesics with unite speed such that $\gamma_{y_1,y_2}(0)=y_1$ and $\gamma_{y_1,y_2}(d(y_1,y_2))=y_2$. The minimal geodesic $\gamma_{y_1,y_2}$ is uniquely determined if $y_2\in I_{y_1}\setminus \{y_1\}$, and we have $\Vol(M\setminus I_{y_1})=0$ (see subsection 2.1 for details).

\begin{Thm}[segment inequality]\label{p2f}
Given an integer $n\geq 2$ and positive real numbers $K>0$ and $D>0$, there exists a positive constant $ C(n,K,D)>0$ such that the following property holds.
Let $(M,g)$ be an $n$-dimensional closed Riemannian manifold with $\Ric\geq -Kg$ and $\diam(M)\leq D$.
For any non-negative measurable function $h\colon M\to \mathbb{R}_{\geq 0}$, we have
\begin{equation*}
\frac{1}{\Vol(M)^2}\int_{M\times M} \int_0^{d(y_1,y_2)} h\circ \gamma_{y_1,y_2}(s) \,dsdy_1dy_2\leq  \frac{C}{\Vol(M)}\int_M h\,d\mu_g.
\end{equation*}
\end{Thm}
We use the segment inequality to prove the following lemma.
The following lemma gives us one dimensional integral pinching conditions on most of the geodesics under the integral pinching condition for a function on $M$.
\begin{Lem}\label{p2e}
Given an integer $n\geq2$ and positive real numbers $K>0$, $D>0$ and $0<\delta\leq \frac{1}{8n^3}$, there exists a positive constant $ C_{\Ck}(n,K,D)>0$ such that the following property holds.
Let $(M,g)$ be an $n$-dimensional closed Riemannian manifold with $\Ric\geq -Kg$ and $\diam(M)\leq D$.
Suppose that a non-zero function $f\in C^\infty(M)$ satisfies $\|\nabla^2 f+f g\|_2\leq \delta\|f\|_2$ and $\|f_1\|_2^2=\frac{1}{n+1}$, where we put $f_1=T_{(n,\delta)}(f)$.
For any $y_1\in M$, we define
\begin{equation*}
\begin{split}
D(y_1):=\Big\{y_2\in I_{y_1}\setminus\{y_1\}:  \int_0^{d(y_1,y_2)} |\nabla^2 f_1+f_1 g|\circ\gamma_{y_1,y_2}(s)\, d s\leq\delta^\frac{1}{6}
\Big\},
\end{split}
\end{equation*}
and
\begin{equation*}
Q:=\{y_1\in M: \Vol(D(y_1))\geq(1-\delta^\frac{1}{6})\Vol(M)\}.
\end{equation*}
Then, we have
\begin{equation*}
\Vol(Q)\geq(1- C_{\Ck}\delta^\frac{1}{6})\Vol(M).
\end{equation*}
\end{Lem}
\begin{proof}
For any $y_1\in M\setminus Q$, we have $\Vol(M\setminus D(y_1))>\delta^\frac{1}{6}\Vol(M)$, and so we get
\begin{equation}\label{1n}
\begin{split}
&\frac{1}{\Vol(M)}\int_{M} \int_0^{d(y_1,y_2)} |\nabla^2 f_1+f_1 g|\circ \gamma_{y_1,y_2}(s) \,dsdy_2\\
\geq &\frac{1}{\Vol(M)}\int_{M \setminus D(y_1)} \int_0^{d(y_1,y_2)} |\nabla^2 f_1+f_1 g|\circ \gamma_{y_1,y_2}(s) \,dsdy_2\\
\geq &\frac{\Vol(M\setminus D(y_1))}{\Vol(M)} \delta^\frac{1}{6} \geq \delta^\frac{1}{3}.
\end{split}
\end{equation}
By Theorem \ref{p2f}, (\ref{1n}) and $\|\nabla^2 f_1+ f_1g\|_1\leq \|\nabla^2 f_1 +  f_1 g\|_2\leq C \sqrt{\delta}$, we get
\begin{equation*}
\begin{split}
&\frac{\Vol(M\setminus Q)}{\Vol(M)} \delta^\frac{1}{3}\\
\leq& \frac{1}{\Vol(M)^2} \int_{M\setminus Q} dy_1\int_M dy_2\int_0^{d(y_1,y_2)} |\nabla^2 f_1+f_1 g|\circ \gamma_{y_1,y_2}(s) \,ds\\
\leq &\frac{1}{\Vol(M)^2}\int_{M\times M} \int_0^{d(y_1,y_2)} |\nabla^2 f_1+f_1 g|\circ \gamma_{y_1,y_2}(s) \,dsdy_1dy_2\\
\leq&C\sqrt{\delta}.
\end{split}
\end{equation*}
Therefore, we get $\Vol(M\setminus Q)\leq C \delta^\frac{1}{6}\Vol(M)$.
\end{proof}
\subsection{Almost cosine function}
In this subsection, we show that our function $f_1$ in Lemma \ref{p2c} is an almost cosine function in $L^\infty$ and $W^{1,2}$ sense.

We need the following easy consequence of the Bishop-Gromov inequality.
\begin{Thm}\label{p2h}
Given a positive integer $n\geq2$ and positive real numbers $K>0$ and $D>0$, there exists a positive constant $ C_{\Cm}(n,K,D)>0$ such that the following property holds.
Let $(M,g)$ be an $n$-dimensional closed Riemannian manifold with $\Ric\geq -Kg$ and $\diam(M)\leq D$.
Then, for any $p\in M$ and $0<r\leq D+1$, we have $\Vol(B_r(p))\geq  C_{\Cm} r^n \Vol(M)$.
\end{Thm}

The following lemma is standard (c.f. \cite[Lemma 2.41]{CC2}).
\begin{Lem}\label{p2g}
Take positive real numbers $l,\epsilon_i>0$ $(i=0,1,2)$ and real numbers $a,b\in \mathbb{R}$.
Put
\begin{equation*}
\widetilde{u}(t):=a \cos t+ b \sin t,
\end{equation*}
for any $t\in \mathbb{R}$.
Suppose that a smooth function $u\colon [0,l]\to \mathbb{R}$ satisfies
\begin{itemize}
\item[$\cdot$] $|u(0)-a|\leq \epsilon_0$,
\item[$\cdot$] $|u'(0)-b|\leq \epsilon_1$,
\item[$\cdot$] $\int_0^l |u''(t)+u(t)| \,dt\leq\epsilon_2$.
\end{itemize}
Then, we have
\begin{equation*}
\begin{split}
|u(t)-\widetilde{u}(t)|&\leq \epsilon_0 \cosh t+(\epsilon_1+\epsilon_2)\sinh t,\\
|u'(t)-\widetilde{u}'(t)|&\leq \epsilon_1+\epsilon_2+ \int_0^t|u(s)-\widetilde{u}(s)|\,ds,
\end{split}
\end{equation*}
for all $t\in [0,l]$.
\end{Lem}


Recall that 
we defined the $(n,\delta)$-projection $T_{(n,\delta)}\colon L^2(M)\to L^2(M)$ to be the orthogonal projection onto a subspace $\Span\{\phi_j:j\in \mathbb{N}\text{ with }n-\sqrt{\delta}\leq\lambda_j\leq n+\sqrt{\delta}\}$.
In Lemma \ref{p2c}, we showed that if a function $f$ satisfies our pinching condition $\|\nabla^2 f + f g\|_2\leq \delta\|f\|_2$, then $T_{(n,\delta)}(f)$ also satisfies our pinching condition.

The following proposition is the goal of this subsection.
\begin{Prop}\label{p2i}
Given an integer $n\geq 2$ and positive real numbers $K>0$ and $D>0$, there exist positive constants $ C(n,K,D)>0$ and $0<\delta_{\Da}(n,K,D)\leq\frac{1}{8n^3}$ such that the following properties hold.
Take a positive real number $0<\delta\leq \delta_{\Da}$.
Let $(M,g)$ be an $n$-dimensional closed Riemannian manifold with $\Ric\geq -Kg$ and $\diam(M)\leq D$.
Suppose that a non-zero function $f\in C^\infty(M)$ satisfies $\|\nabla^2 f+f g\|_2\leq \delta\| f\|_2$ and $\|f_1\|_2^2=\frac{1}{n+1}$, where we put $f_1:=T_{(n,\delta)}(f)$.
Then, there exists a point $p\in M$ such that $f_1(p)>0$, $|f_1(p)-1|\leq  C\delta^\frac{1}{48n}$, and putting $h\colon M\to \mathbb{R}$ as $h:=\cos d(p,\cdot)$, we have
\begin{equation*}
\begin{split}
\|f_1-h\|_{\infty}&\leq C \delta^\frac{1}{48n},\\
\|\nabla f_1-\nabla h\|_2&\leq  C \delta^\frac{1}{48n}.
\end{split}
\end{equation*}
\end{Prop}
\begin{proof}
Take a maximum point $\tilde{p}\in M$ of $f_1$.
By Lemma \ref{p2d},
\begin{equation}\label{2a}
|\nabla f_1|^2\leq \frac{ C}{\delta^\frac{1}{6n}}(f_1(\tilde{p})-f_1+\delta^\frac{1}{6n})^2.
\end{equation}
Put $\epsilon:=\left(2 C_{\Ck}\delta^\frac{1}{6}/ C_{\Cm}\right)^\frac{1}{n}$. We can assume that $\epsilon\leq 1$ by taking $\delta_{\Da}$ sufficiently small.
By Lemma \ref{p2e} and the Bishop-Gromov inequality (Theorem \ref{p2h}), we get
\begin{equation*}
\Vol\left(B_{\epsilon}(\tilde{p})\right)\geq 2  C_{\Ck}\delta^\frac{1}{6}\Vol(M)\geq 2\Vol(M\setminus Q).
\end{equation*}
Thus, we have $Q\cap B_{\epsilon}(\tilde{p})\neq \emptyset$.
Take a point $p \in Q\cap B_{\epsilon}(\tilde{p})$.
Since $\|\nabla f_1\|_{\infty}\leq C$ by Lemma \ref{p2c} (ii), we have $|f(\tilde{p})-f(p)|\leq  C \delta^{\frac{1}{6n}}$.
Thus, by (\ref{2a}), we get
\begin{equation}\label{2b}
|\nabla f_1(p)|\leq  C \delta^\frac{1}{12n}.
\end{equation}

Let us show that $f_1$ is close to $f_1(p)\cos d(p,\cdot)$ in $L^\infty$ sense.
Take arbitrary $x\in D(p)$.
Since
\begin{equation*}
(\nabla^2f_1+f_1 g)(\dot{\gamma}_{p,x}(s),\dot{\gamma}_{p,x}(s))=\frac{d^2}{ds^2} f_1\circ\gamma_{p,x}(s)+f_1\circ\gamma_{p,x}(s),
\end{equation*}
we have
\begin{equation*}
\left|\frac{d^2}{ds^2} f_1\circ\gamma_{p,x}+f_1\circ\gamma_{p,x}\right|\leq |\nabla^2f_1+
f_1 g|\circ\gamma_{p,x}.
\end{equation*}
Therefore, we get
\begin{equation*}
\int_0^{d(p,x)}\left|\frac{d^2}{ds^2} f_1\circ\gamma_{p,x}+f_1\circ\gamma_{p,x}\right|\, d s\leq\delta^\frac{1}{6}.
\end{equation*}
Moreover, by (\ref{2b}), the smooth map $f_1\circ\gamma_{p,x}\colon [0,d(p,x)]\to\mathbb{R}$ satisfies
\begin{itemize}
\item[$\cdot$] $|f_1\circ\gamma_{p,x}(0)-f_1(p)|=0$,
\item[$\cdot$] $|\frac{d}{ds}f_1\circ\gamma_{p,x}(0)-0| \leq  C\delta^\frac{1}{12n}$.
\end{itemize}
Thus, by Lemma \ref{p2g}, for all $t\in [0,d(p,x)]$, we get
\begin{equation*}
\begin{split}
|f_1(\gamma_{p,x}(t))-f_1(p)\cos t|\leq&C\delta^\frac{1}{12n},\\
|\langle\nabla f_1,\dot{\gamma}_{p,x}(t)\rangle+f_1(p)\sin t|
\leq & C\delta^\frac{1}{12n}.
\end{split}
\end{equation*}
In particular, for all $x\in D(p)$, we get
\begin{equation}\label{2c}
\begin{split}
|f_1(x)-f_1(p)\cos d(p,x)|&\leq  C\delta^\frac{1}{12n},\\
|\langle\nabla f_1(x),\dot{\gamma}_{p,x}\rangle+f_1(p)\sin d(p,x)|&\leq  C\delta^\frac{1}{12n}.
\end{split}
\end{equation}

Take arbitrary $x\in M$.
Similarly to $p$, we can take a point $\tilde{x}\in D(p)$ with $d(x,\tilde{x})\leq (2\delta^\frac{1}{6}/ C_{\Cm})^\frac{1}{n}$ by Lemma \ref{p2e} and the Bishop-Gromov inequality (Theorem \ref{p2h}).
By (\ref{2c}) and  $\|\nabla f_1\|_{\infty}\leq C$, we get
\begin{equation*}
\begin{split}
&|f_1(x)-f_1(p)\cos d(p,x)|\\
\leq& |f_1(x)-f_1(\tilde{x})|+|f_1(\tilde{x})-f_1(p)\cos d(p,\tilde{x})|\\
&\qquad+|f_1(p)\cos d(p,\tilde{x})-f_1(p)\cos d(p,x)|\\
\leq & C\delta^\frac{1}{12n}.
\end{split}
\end{equation*}
Define $\tilde{h}:=f_1(p)\cos d(p,\cdot)$. Then, we get
\begin{equation}\label{2d}
\|f_1-\tilde{h}\|_{\infty}\leq  C\delta^\frac{1}{12n},
\end{equation}
and so
\begin{equation}\label{2e}
\frac{1}{\sqrt{n+1}}- C\delta^\frac{1}{12n}\leq
\|\tilde{h}\|_2\leq \frac{1}{\sqrt{n+1}}+ C \delta^\frac{1}{12n}.
\end{equation}
Since $\|f_1\|_\infty\geq \|f_1\|_2=\frac{1}{\sqrt{n+1}}$ and $f_1(p)\geq f_1(\tilde{p})-C\delta^{\frac{1}{6n}}\geq -C\delta^{\frac{1}{6n}}$, we get
\begin{equation}\label{rv0}
f_1(p)\geq \frac{1}{\sqrt{n+1}}-C \delta^{\frac{1}{12n}}>0
\end{equation}
by the definition of $\tilde{h}$ and (\ref{2d}).

Let us estimate $\|\nabla f_1- \nabla \tilde{h}\|_2$.
Decompose $\tilde{h}$ as $\tilde{h}=\sum_{j=0}^\infty \beta_j\phi_j$ ($\beta_j\in\mathbb{R}$). Define $h_1:=\sum_{\lambda_j\geq n-\sqrt{\delta}} \beta_j\phi_j$ and $h_2:=\tilde{h}-h_1$.
Then, we have $\int_M f_1 h_2\,d\mu_g=0$, and so, by (\ref{2d}) and (\ref{2e}),
\begin{equation*}
\begin{split}
\|h_2\|_2^2=&\frac{1}{\Vol(M)}\int_M \tilde{h}h_2\,d\mu_g\\
=&\frac{1}{\Vol(M)}\int_M (\tilde{h}-f_1)h_2\,d\mu_g
\leq  C\delta^\frac{1}{12n}\|\tilde{h}\|_2
\leq C \delta^\frac{1}{12n}\|\tilde{h}\|_2^2.
\end{split}
\end{equation*}
Thus, we get
\begin{equation*}
\|\nabla \tilde{h}\|_2^2\geq (n-\sqrt{\delta})\|h_1\|_2^2\geq (n-\sqrt{\delta}) (1-C\delta^\frac{1}{12n})\|\tilde{h}\|_2^2.
\end{equation*}
Moreover, by (\ref{2e}) we have
\begin{equation*}
\|\nabla f_1\|_2^2\leq (n+\sqrt{\delta})\|f_1\|_2^2\leq (n+\sqrt{\delta})(1+C \delta^\frac{1}{12n})\|\tilde{h}\|_2^2.
\end{equation*}
Therefore, we get
\begin{equation}\label{2f}
\begin{split}
\|\nabla f_1\|_2^2
\leq  (1+ C\delta^\frac{1}{12n})\|\nabla \tilde{h}\|_2^2.
\end{split}
\end{equation}

Note that for any  $x\in I_p\setminus\{p\}$, we have $\nabla \tilde{h}(x)=-(f_1(p)\sin d(p,x) )\dot{\gamma}_{p,x}(d(p,x))$.
Thus, by $\|f_1\|_{\infty}\leq C$, for almost all point $x\in M$, we have $|\nabla \tilde{h}|(x)\leq C$.

For any $x\in D(p)$, by (\ref{2c}) we have
\begin{equation*}
\begin{split}
 C \delta^\frac{1}{12n}
\geq&\left|\langle\nabla f_1(x),\dot{\gamma}_{p,x}\rangle+f_1(p)\sin d(p,x)\right|\\
\geq&|f_1(p)\sin d(p,x)|-|\langle\nabla f_1(x),\dot{\gamma}_{p,x}\rangle|\\
=&|\nabla \tilde{h}|(x)-|\langle\nabla f_1(x),\dot{\gamma}_{p,x}\rangle|.
\end{split}
\end{equation*}
Thus, for any $x\in D(p)$, using $\|\nabla f_1\|_{\infty}\leq C$, we get
\begin{equation}\label{2g}
\begin{split}
|\nabla \tilde{h}|^2(x)
\leq &\left(|\langle\nabla f_1(x),\dot{\gamma}_{p,x}\rangle|+ C\delta^\frac{1}{12n}\right)^2\\
\leq&|\langle\nabla f_1(x),\dot{\gamma}_{p,x}\rangle|^2+ C \delta^\frac{1}{12n}
\leq |\nabla f_1|^2(x)+ C\delta^\frac{1}{12n}.
\end{split}
\end{equation}
To estimate $\|\nabla f_1-\nabla\tilde{h}\|_2$, we show that the value $|\nabla f_1-\nabla\tilde{h}|(x)$ is small for most of the points $x\in M$ in the following two claims.
\begin{Clm}\label{C1}
Define $S:=\left\{x\in D(p):|\nabla \tilde{h}|^2(x)+\delta^\frac{1}{24n}\leq|\nabla f_1|^2(x)\right\}$.
Then, we have $\Vol(S)\leq  C \delta^\frac{1}{24n}\Vol(M)$.
\end{Clm}
\begin{proof}[Proof of Claim \ref{C1}]
By the definition of $S$, (\ref{2f}) and (\ref{2g}), we get
\begin{equation*}
\begin{split}
\|\nabla f_1\|_2^2\geq& \frac{1}{\Vol(M)}\int_{D(p)}|\nabla f_1|^2\,d\mu_g\\
\geq &\frac{1}{\Vol(M)}\int_{D(p)\setminus S}|\nabla f_1|^2\,d\mu_g+ \frac{1}{\Vol(M)}\int_{S} |\nabla \tilde{h}|^2\,d\mu_g+\frac{\Vol(S)}{\Vol(M)}\delta^\frac{1}{24n}\\
\geq & \frac{1}{\Vol(M)}\int_{D(p)}|\nabla \tilde{h}|^2\,d\mu_g- C\delta^\frac{1}{12n}+\frac{\Vol(S)}{\Vol(M)}\delta^\frac{1}{24n}\\
\geq &\|\nabla \tilde{h}\|_2^2- C \frac{\Vol(M\setminus D(p))}{\Vol(M)}
- C \delta^\frac{1}{12n}+\frac{\Vol(S)}{\Vol(M)}\delta^\frac{1}{24n}\\
\geq & \frac{1}{1+ C\delta^\frac{1}{12n}}\|\nabla f_1 \|_2^2-C\delta^\frac{1}{12n}+\frac{\Vol(S)}{\Vol(M)}\delta^\frac{1}{24n}.
\end{split}
\end{equation*}
Thus,
\begin{equation*}
\begin{split}
\frac{\Vol(S)}{\Vol(M)}\delta^\frac{1}{24n}\leq&  C\delta^\frac{1}{12n}\|\nabla f_1\|_2^2+C\delta^\frac{1}{12n}\leq  C \delta^\frac{1}{12n},
\end{split}
\end{equation*}
and so we get the claim.
\end{proof}

\begin{Clm}\label{C2}
For any $x\in D(p)\setminus S$, we have $|\nabla f_1-\nabla \tilde{h}|^2(x)\leq  C\delta^\frac{1}{24n}$.
\end{Clm}
\begin{proof}[Proof of Claim \ref{C2}]
For any $x\in D(p)\setminus S$, we have
\begin{equation}\label{2i}
|\nabla \tilde{h}|^2(x)+\delta^\frac{1}{24n}>|\nabla f_1|^2(x).
\end{equation}
Since we have $\left\langle\nabla\tilde{h}(x),\dot{\gamma}_{p,x}\right\rangle\dot{\gamma}_{p,x}=\nabla \tilde{h}(x)$, we get
\begin{equation}\label{2j}
\begin{split}
&|\nabla f_1-\nabla \tilde{h}|^2(x)\\
\leq&\left|\left\langle\nabla f_1(x)-\nabla \tilde{h}(x), \nabla f_1(x)-\nabla \tilde{h}(x)-\langle\nabla f_1(x)-\nabla\tilde{h}(x),\dot{\gamma}_{p,x}\rangle\dot{\gamma}_{p,x}\right\rangle\right|\\
&\qquad\qquad\qquad\qquad + \left|\left\langle\nabla f_1(x)-\nabla \tilde{h}(x),\langle\nabla f_1(x)-\nabla\tilde{h}(x),\dot{\gamma}_{p,x}\rangle\dot{\gamma}_{p,x}\right\rangle\right|\\
=&\left|\left\langle\nabla f_1(x),\nabla f_1(x)-\langle\nabla f_1(x),\dot{\gamma}_{p,x}\rangle\dot{\gamma}_{p,x}\right\rangle\right|+ \left|\left\langle\nabla f_1(x)-\nabla\tilde{h}(x),\dot{\gamma}_{p,x}\right\rangle\right|^2,
\end{split}
\end{equation}
By (\ref{2g}) and (\ref{2i}),
\begin{equation*}
\begin{split}
&\left|\left\langle\nabla f_1(x),\nabla f_1(x)-\langle\nabla f_1(x),\dot{\gamma}_{p,x}\rangle\dot{\gamma}_{p,x}\right\rangle\right|\\
=&|\nabla f_1|^2(x)-\left|\left\langle\nabla f_1(x),\dot{\gamma}_{p,x}\right\rangle\right|^2\\
\leq& |\nabla f_1|^2(x)-|\nabla \tilde{h}|^2(x)+ C\delta^\frac{1}{12n}
\leq C\delta^\frac{1}{24n},
\end{split}
\end{equation*}
and by (\ref{2c}),
\begin{equation*}
\begin{split}
\left|\left\langle\nabla f_1(x)-\nabla\tilde{h}(x),\dot{\gamma}_{p,x}\right\rangle\right|^2=&|\langle\nabla f_1(x),\dot{\gamma}_{p,x}\rangle+f_1(p)\sin d(p,x)|^2 \leq  C\delta^\frac{1}{6n}.
\end{split}
\end{equation*}
Thus, by (\ref{2j}), we get
$|\nabla f_1-\nabla \tilde{h}|^2(x)\leq C\delta^\frac{1}{24n}$.
\end{proof}
Since
\begin{equation*}
\|\nabla f_1-\nabla \tilde{h}\|_{\infty}=\|\nabla f_1+f_1(p)\sin d(p,\cdot)\nabla (d(p,\cdot))\|_{\infty}\leq C,
\end{equation*}
we get
\begin{equation*}
\begin{split}
&\|\nabla f_1-\nabla \tilde{h}\|_2^2\\
=&\frac{1}{\Vol(M)}\int_{M\setminus D(p)}|\nabla f_1-\nabla \tilde{h}|^2\,d\mu_g
+\frac{1}{\Vol(M)}\int_{D(p)\setminus S}|\nabla f_1-\nabla \tilde{h}|^2\,d\mu_g\\
&\qquad\qquad\qquad+\frac{1}{\Vol(M)}\int_{S}|\nabla f_1-\nabla \tilde{h}|^2\,d\mu_g\\
\leq & C \delta^\frac{1}{24n}
\end{split}
\end{equation*}
by Claim \ref{C1} and Claim \ref{C2}.
Thus, we get
\begin{equation}\label{2k}
\|\nabla f_1-\nabla \tilde{h}\|_2\leq C\delta^\frac{1}{48n}.
\end{equation}

Let us show that $f_1(p)$ is close to $1$, and complete the proof.
For any $x\in I_p\setminus\{p\}$, we have
\begin{equation*}
f_1(p)^2=\tilde{h}(x)^2+|\nabla \tilde{h}|^2(x),
\end{equation*}
and so
\begin{equation*}
\begin{split}
&\left|f_1(p)^2-(f_1(x)^2+|\nabla f_1|^2(x))\right|\\
\leq& |f_1(x)^2-\tilde{h}(x)^2|+\left||\nabla f_1|^2(x)-|\nabla \tilde{h}|^2(x)\right|\\
\leq & C|f_1(x)-\tilde{h}(x)|+ C\left|\nabla f_1(x)-\nabla \tilde{h}(x)\right|.
\end{split}
\end{equation*}
Thus, by (\ref{2d}), (\ref{rv0}) and (\ref{2k}), we get
\begin{equation*}
\begin{split}
&\left|f_1(p)-1\right|\\
\leq&\left|f_1(p)-1\right| \left|f_1(p)+1\right|\\
=&\left|f_1(p)^2-(n+1)\|f_1\|_2^2\right|\\
\leq &\left|f_1(p)^2-(\|f_1\|_2^2+\|\nabla f_1\|_2^2)\right|+\left|n\|f_1\|_2^2-\|\nabla f_1\|_2^2\right|\\
\leq&\left|\frac{1}{\Vol(M)}\int_M f_1(p)^2-(f_1^2+|\nabla f_1|^2)\,d\mu_g\right|+\sqrt{\delta}\|f_1\|_2^2\\
\leq &C\|f_1-\tilde{h}\|_1+C \|\nabla f_1-\nabla \tilde{h}\|_1+\frac{\sqrt{\delta}}{n+1}
\leq  C\delta^\frac{1}{48n}.
\end{split}
\end{equation*}
This gives
\begin{equation*}
\begin{split}
\|h-\tilde{h}\|_{\infty}&\leq C\delta^\frac{1}{48n},\\
\|\nabla h-\nabla \tilde{h}\|_{\infty}&\leq C\delta^\frac{1}{48n}.
\end{split}
\end{equation*}
Thus, we have
\begin{equation*}
\begin{split}
\|f_1-h\|_{\infty}&\leq C\delta^\frac{1}{48n},\\
\|\nabla f_1-\nabla h\|_2&\leq C\delta^\frac{1}{48n}
\end{split}
\end{equation*}
by (\ref{2d}) and (\ref{2k}).
Therefore, we get Proposition \ref{p2i}.
\end{proof}

\subsection{Small excess}
In this subsection, we show the small excess like property under our pinching condition.
To explain this, we recall the small excess lemma for Riemannian manifolds with $\Ric\geq(n-1)g$ \cite[Lemma 1]{GP}.
The small excess lemma asserts that if a closed Riemannian manifold $(M,g)$ with $\Ric \geq (n-1)g$ has two points $p,q\in M$ such that $d(p,q)$ is close to $\pi$, then the excess
$$
d(p,x)+d(x,q)-d(p,q)
$$
is small for all $x\in M$.
Note that the classical Myers theorem and Cheng's maximal diameter theorem assert that
we have $\diam(M)\leq \pi$ under the assumption $\Ric \geq (n-1)g$, and $\diam(M)=\pi$ holds if and only if $(M,g)$ is isometric to the $n$-dimensional standard sphere with radius $1$.
The proof of the small excess lemma relies on the Bishop-Gromov inequality:
$$
\frac{\Vol(B_r(x))}{\Vol (M)}\geq \frac{\Vol(B_r^{S^n})}{\Vol(S^n)}
$$
holds for all $x\in M$ and $0\leq r\leq \pi$, where $B_r^{S^n}$ denotes the metric ball in $S^n$ with radius $r$.
Since the proof uses the positive Ricci curvature assumption deeply,
it cannot be applied to our situation immediately.
To show the small excess like property under our pinching condition, we divide the manifold to pieces such that the small excess property holds in each pieces.
More precisely, under our pinching condition, we show that we can take finite points $x_0,x_1,\ldots,x_N$ such that if we define $B_{ij}$ for each $i,j$ such that $d(x_i,x_j)$ is close to $\pi$, to be the subset of $M$ consisting of the points $x$ that satisfy the small excess property:
$$
d(x_i,x)+d(x,x_j)-d(x_i,x_j)
$$
is small,
then we have $\bigcup B_{ij}=M$ (see Proposition \ref{p2k} for details).

The following lemma is standard.
\begin{Lem}\label{p2j}
For all $t\in \mathbb{R}$, we have
\begin{equation*}
1-\frac{1}{2}t^2\leq \cos t\leq 1-\frac{1}{2}t^2+\frac{1}{24}t^4.
\end{equation*}
For any $t\in [-\pi,\pi]$, we have $\cos t\leq 1-\frac{1}{9}t^2$, and so $|t|\leq3(1-\cos t)^\frac{1}{2}$.
For any $t_1,t_2 \in [0,\pi]$, we have $|t_1-t_2|\leq3|\cos t_1-\cos t_2|^\frac{1}{2}$.
\end{Lem}

Let us show the small excess like property under our pinching condition.
\begin{Prop}\label{p2k}
Given an integer $n\geq 2$ and positive real numbers $K>0$ and $D>0$, there exist positive constants $N_3(n,K,D)>0$ and $0<\delta_{\Db}(n,K,D)\leq\delta_{\Da}$ such that the following properties hold.
Take a positive real number $0<\delta\leq \delta_{\Db}$.
Let $(M,g)$ be an $n$-dimensional closed Riemannian manifold with $\Ric\geq -Kg$ and $\diam(M)\leq D$.
Suppose that a non-zero function $f\in C^\infty(M)$ satisfies $\|\nabla^2 f+f g\|_2\leq \delta\|f\|_2$.
Take a point $p\in M$ of Proposition \ref{p2i}.
Then, we have following properties.
\begin{itemize}
\item[(i)] For each $m\in \mathbb{Z}_{>0}$, we define
\begin{equation*}
A_m:=\{x\in M: m\pi-\delta^\frac{1}{100n}\leq d(p,x) \leq m\pi+\delta^\frac{1}{100n}\}.
\end{equation*}
Then, $A_1\neq \emptyset$.
Define a relation $\sim$ on $A_m$ by
\begin{equation*}
x\sim y :\iff d(x,y)<\delta^\frac{1}{200n^2},
\end{equation*}
for $x,y\in A_m$.
Then, $\sim$ is an equivalence relation on $A_m$.
Moreover, we can choose finite points $x_1,\ldots,x_N\in M$ $(N\leq N_3)$ such that
each $x_i$ $(i=1,\ldots,N)$ is an element of $A_{m_i}$ for some $m_i\in \mathbb{Z}_{>0}$ and
\begin{equation*}
\bigcup_{m=1}^\infty A_{m}/\sim =\{[x_1],\ldots,[x_N]\}
\end{equation*}
holds. 
\item[(ii)] For all $i=1,\ldots, N$, we have
$$
\|f_1-(-1)^{m_i}\sqrt{n+1}\|f_1\|_2 \cos d(x_i,\cdot)\|_{\infty}\leq C\delta^\frac{1}{75n^2}\|f_1\|_2,
$$
where we put $f_1:=T_{(n,\delta)}(f)$.
\item[(iii)] We set $x_0:=p$. For each $i,j=0,\ldots, N$ $(i\neq j)$, there exists a positive integer $k_{ij}\in\mathbb{Z}_{>0}$ such that
\begin{equation*}
\left|d(x_i,x_j)-\pi k_{ij}\right|\leq \delta^\frac{1}{200n^2}.
\end{equation*}
Moreover, if $(-1)^{m_i}=(-1)^{m_j}$, then $k_{ij}$ is even, and if $(-1)^{m_i}=-(-1)^{m_j}$, then $k_{ij}$ is odd.
\item[(iv)] For each $i,j=0,\ldots, N$ such that $\left|d(x_i,x_j)-\pi\right|\leq \delta^\frac{1}{200n^2}$, we define
\begin{equation*}
\begin{split}
B_{ij}&:=\{x\in M: d(x_i,x)+d(x,x_j)\leq d(x_i,x_j)+\delta^\frac{1}{250n^2}\},\\
\widetilde{B}_{ij}&:=\{x\in B_{ij}: d(x,x_i)>3\delta^\frac{1}{250n^2}\text{ and } d(x,x_j)> 3\delta^\frac{1}{250n^2}\}.
\end{split}
\end{equation*}
Then, $M=\bigcup_{ij} B_{ij}$ holds, $\widetilde{B}_{ij}$ is an open subset of $M$ for each $i,j$, and
$\widetilde{B}_{ij}\cap B_{kl}=\emptyset$ if $\{k,l\}\neq \{i,j\}$.
\item[(v)] Suppose that $i,j,k=0,\ldots, N$ satisfy $j\neq k$, $\left|d(x_i,x_j)-\pi\right|\leq \delta^\frac{1}{200n^2}$ and $\left|d(x_i,x_k)-\pi\right|\leq \delta^\frac{1}{200n^2}$.
Then, for any $x\in B_{ij}$ and $y\in B_{ik}$, we have $d(x,y)\geq d(x,x_i)+d(x_i,y)-5\delta^\frac{1}{250n^2}$.
\end{itemize}
\end{Prop}
\begin{proof}
We first assume $\delta\leq \delta_{\Da}$.
Set $f_1:=T_{(n,\delta)}(f)$.
We can assume $\|f_1\|_2^2=\frac{1}{n+1}$.
Take a point $p\in M$ of Proposition \ref{p2i} and set $h:=\cos d(p,\cdot)$.
Considering $-f$, we can take a point $q\in M$ with $|f_1(q)+1|\leq C\delta^\frac{1}{48n}$. 
\begin{Clm}\label{c1}
There exists a positive constant $\eta(n,K,D)>0$ such that the following property holds.
If $\delta\leq\eta$, then we have $d(p,q)\geq \pi-\delta^\frac{1}{100n}$.
In particular, we have $A_1\neq\emptyset$.
\end{Clm}
\begin{proof}[Proof of Claim \ref{c1}]
We have
\begin{equation*}
|1-\cos(d(p,q)-\pi)|=|1+\cos d(p,q)|\leq|1+f_1(q)|+|f_1(q)-h(q)|\leq C\delta^\frac{1}{48n}.
\end{equation*}
Take an integer $m_q\geq 0$ with $d(p,q)-\pi-2m_q\pi \in [-\pi,\pi]$.
By Lemma \ref{p2j}, we get
\begin{equation*}
|d(p,q)-\pi-2m_q\pi|\leq C\delta^\frac{1}{96n},
\end{equation*}
and so 
\begin{equation*}
d(p,q)\geq \pi-C \delta^\frac{1}{96n}.
\end{equation*}
Thus, we get the claim.
\end{proof}
Let us show that, for all point $x\in M$, there exists a nice point close to $x$ such that the value of $|\nabla f_1- \nabla h|$ is small at the point.
\begin{Clm}\label{c2}
There exists a positive constant $\eta(n,K,D)>0$ such that the following property holds.
Suppose that $\delta\leq \eta$.
For all $x\in M$, there exists a point $\tilde{x}\in B_{\delta^\frac{1}{75n^2}}(x)\cap (I_p\setminus\{p\})\cap Q$ with $|\nabla f_1 -\nabla h|(\tilde{x})\leq \delta^\frac{1}{75n^2}$.
\end{Clm}
\begin{proof}[Proof of Claim \ref{c2}]
Take arbitrary $x\in M$.
Suppose that, for all $\tilde{x}\in B_{\delta^\frac{1}{75n^2}}(x)\cap (I_p\setminus\{p\})$, we have either $\tilde{x}\notin Q$ or $|\nabla f_1-\nabla h|(\tilde{x})\geq \delta^\frac{1}{75n^2}$.
By the Bishop-Gromov inequality (Theorem \ref{p2h}), we get
$\Vol\left(B_{\delta^\frac{1}{75n^2}}(x)\right)\geq C_{\Cm}\delta^\frac{1}{75n}\Vol(M)$, and so we have at least one of the following:

\begin{itemize}
\item[(a)] $\Vol(M\setminus Q)\geq \frac{1}{2}C_{\Cm}\delta^\frac{1}{75n}\Vol(M)$.
\item[(b)] $\Vol\left(\{\tilde{x}\in I_p\setminus\{p\}:|\nabla f_1-\nabla h|(\tilde{x})\geq \delta^\frac{1}{75n^2}\}\right)\geq \frac{1}{2}C_{\Cm}\delta^\frac{1}{75n}\Vol(M)$.
\end{itemize}
If (a) holds, then by Lemma \ref{p2e}, $C_{\Ck}\delta^\frac{1}{6}\geq \frac{1}{2}C_{\Cm}\delta^\frac{1}{75n}$, and so $\delta^\frac{1}{4n}\geq C$.
If (b) holds, then by Proposition \ref{p2i},
\begin{equation*}
\frac{1}{2}C_{\Cm}\delta^\frac{n+1}{75n^2}\leq \|\nabla f_1-\nabla h\|_1\leq C \delta^\frac{1}{48n}
\leq C\delta^\frac{n+1}{72n^2},
\end{equation*}
and so $\delta^\frac{1}{600n^2}\geq C$.
Taking $\eta$ sufficiently small, the above situation cannot occur. Thus, we get the claim. 
\end{proof}
Let us show that our function $f_1$ is almost cosine function of the distance from the point in $A_m$ similarly to Proposition \ref{p2i} using Claim \ref{c2}.
\begin{Clm}\label{c3}
For any point $x\in A_m$ $(m\in \mathbb{Z}_{>0})$, we have
\begin{equation*}
\|f_1-(-1)^m \cos d(x,\cdot)\|_{\infty}\leq C\delta^\frac{1}{75n^2}.
\end{equation*}
\end{Clm}
\begin{proof}[Proof of Claim \ref{c3}]
Take arbitrary $x\in A_m$.
By Claim \ref{c2}, we can take a point $\tilde{x}\in B_{\delta^\frac{1}{75n^2}}(x)\cap (I_p\setminus\{p\}) \cap Q$ with $|\nabla f_1 -\nabla h|(\tilde{x})\leq \delta^\frac{1}{75n^2}$.
Then, we have $|d(p,\tilde{x})-m\pi|\leq 2\delta^\frac{1}{75n^2}$, and so by Proposition \ref{p2i},
\begin{equation*}
\begin{split}
|f_1(\tilde{x})-(-1)^m |&\leq |h(\tilde{x})-\cos m\pi|+C\delta^\frac{1}{48n}\leq C\delta^\frac{1}{75n^2},\\
|\nabla f_1(\tilde{x})|&\leq |\nabla h(\tilde{x})|+\delta^\frac{1}{75n^2}\leq 3\delta^\frac{1}{75n^2}.
\end{split}
\end{equation*}
Thus, similarly to (\ref{2d}) in Proposition \ref{p2i}, we get the claim.
\end{proof}


Note that we also have
\begin{equation*}
\|f_1-\cos \d(p,\cdot)\|_{\infty}\leq C\delta^\frac{1}{75n^2}.
\end{equation*}

Let us show that the distance between the points in $A_m$ and $A_{m'}$ is almost integer multiple of $\pi$.
\begin{Clm}\label{c4}
There exists a positive constant $\eta(n,K,D)>0$ such that the following property holds.
Suppose that $\delta\leq \eta$.
Define $A_0:=\{p\}$.
For each $m,m'\in \mathbb{Z}_{\geq 0}$, $x\in A_m$ and $y\in A_{m'}$, there exists an integer $k\in\mathbb{Z}_{\geq0}$ such that
\begin{equation*}
\left|d(x,y)-\pi k\right|\leq \delta^\frac{1}{200n^2}.
\end{equation*}
Moreover, if $(-1)^{m}=(-1)^{m'}$, then $k$ is even, and if $(-1)^{m}=-(-1)^{m'}$, then $k_{ij}$ is odd.
\end{Clm}
\begin{proof}[Proof of Claim \ref{c4}]
By Proposition \ref{p2i}, we have
\begin{equation*}
\begin{split}
\left|f_1(y)-(-1)^{m'} \right|
\leq &C\delta^\frac{1}{48n}+(1-\cos (d(p,y)-m'\pi))
\leq C \delta^\frac{1}{100n}.
\end{split}
\end{equation*}
Thus, by Claim \ref{c3}, we get
\begin{equation*}
\begin{split}
&\left|1-(-1)^{m+m'}\cos d(x,y)\right|\\
\leq&\left|f_1(y)-(-1)^{m'}\right|+\left|f_1(y)-(-1)^{m}\cos d(x,y)\right|
\leq C\delta^\frac{1}{75n^2}.
\end{split}
\end{equation*}
Take an integer $\tilde{k}\in \mathbb{Z}$ such that $d(x,y)+(m+m')\pi-2\tilde{k}\pi \in [-\pi,\pi]$.
Then, by Lemma \ref{p2j}, we get 
\begin{equation*}
\begin{split}
\left|d(x,y)+(m+m')\pi-2\tilde{k}\pi\right|
\leq C \delta^\frac{1}{150n^2}.
\end{split}
\end{equation*}
Putting $k=2\tilde{k}-(m+m')$, we get the claim.
\end{proof}

The purpose of the following two claims is to prove (i).
\begin{Clm}\label{c5}
The relationship $\sim$ is an equivalence relation on $A_m$ for all $m\in \mathbb{Z}_{>0}$.
\end{Clm}
\begin{proof}[Proof of Claim \ref{c5}]
Suppose that, $x\in A_m$ and $y\in A_m$ satisfy $d(x,y)>\delta^\frac{1}{200n^2}$.
Then, by Claim \ref{c4}, there exists a positive integer $k\in \mathbb{Z}_{>0}$ such that
\begin{equation*}
\left|d(x,y)-2\pi k\right|\leq \delta^\frac{1}{200n^2}.
\end{equation*}
In particular,
$d(x,y)\geq 2\pi-\delta^\frac{1}{200n^2}$ holds.

If $x\sim y$ and $y\sim z$ ($x,y,z\in A_m$), then $d(x,z)\leq 2\delta^\frac{1}{200n^2}$, and so $d(x,z)\leq \delta^\frac{1}{200n^2}$.
\end{proof}

\begin{Clm}\label{c6}
There exists a positive constant $N_3(n,K,D)>0$ such that
\begin{equation*}
\Card \left(\bigcup_{m=1}^\infty A_{m}/\sim \right)\leq N_3.
\end{equation*}
\end{Clm}
\begin{proof}[Proof of Claim \ref{c6}]
Take equivalent classes $[x_1],\ldots, [x_l]\in \bigcup_{m=1}^\infty A_{m}/\sim$ ($[x_i]\neq [x_j]$ when $i\neq j$).
Then, by Claim \ref{c4}, we have $d(x_i,x_j)\geq \frac{\pi}{2}$ when $i\neq j$.
Thus, $B_{ \frac{\pi}{4}}(x_i)\cap B_{\frac{\pi}{4}}(x_j)=\emptyset$ when $i\neq j$.
Therefore, by the Bishop-Gromov inequality (Theorem \ref{p2h}),
\begin{equation*}
\Vol(M)\geq \sum_{i=1}^l \Vol(B_{ \frac{\pi}{4}}(x_i))\geq l C_{\Cm} \left(\frac{\pi}{4}\right)^n\Vol(M).
\end{equation*}
Putting $N_3(n,K,D):=\left(4/\pi\right)^n/C_{\Cm}$, we get the claim.
\end{proof}

Define $N:=\Card \left(\bigcup_{m=1}^\infty A_{m}/\sim \right)$.
Take points $x_1,\ldots,x_N\in M$ such that each $x_i$ ($i=1,\ldots,N$) is an element of $A_{m_i}$ for some $m_i\in \mathbb{Z}_{>0}$, and
\begin{equation*}
\bigcup_{m=1}^\infty A_{m}/\sim =\{[x_1],\ldots,[x_N]\}
\end{equation*}
holds.
We get (ii) by Claim \ref{c3}.
We set $x_0:=p$.
Using Claim \ref{c4}, for each $i,j=0,\ldots,N$ ($i\neq j$), we define a positive integer $k_{ij}\in\mathbb{Z}_{>0}$ to be
\begin{equation*}
\left|d(x_i,x_j)-\pi k_{ij}\right|\leq \delta^\frac{1}{200n^2}.
\end{equation*}
If $(-1)^{m_i}=(-1)^{m_j}$, then $k_{ij}$ is even, and if $(-1)^{m_i}=-(-1)^{m_j}$, then $k_{ij}$ is odd.
This gives (iii).
Note that if $\left|d(x_i,x_j)-\pi\right|\leq \delta^\frac{1}{200n^2}$,
then $m_i=m_j+1$ or $m_i=m_j-1$.
For each $i,j=0,\ldots, N$ with $\left|d(x_i,x_j)-\pi\right|\leq \delta^\frac{1}{200n^2}$, we define
\begin{equation*}
\begin{split}
B_{ij}&:=\{x\in M: d(x_i,x)+d(x,x_j)\leq d(x_i,x_j)+\delta^\frac{1}{250n^2}\},\\
\widetilde{B}_{ij}&:=\{x\in B_{ij}: d(x,x_i)>3\delta^\frac{1}{250n^2}\text{ and } d(x,x_j)> 3\delta^\frac{1}{250n^2}\}.
\end{split}
\end{equation*}
The purpose of the following three claims is to show that $M=\bigcup_{ij}B_{ij}$.
To do this, if $M\neq \bigcup_{ij}B_{ij}$, we apply Lemma \ref{p2d} to the subset $\text{interior} (F)\subset M$, where the subset $F$ is defined in the proof of Claim \ref{c9} and $d(x,x_i)\geq\delta^{\frac{1}{210n^2}}$ holds for all $i=0,\ldots, N$ and $x\in F$.
Then, we show that the gradient $|\nabla f_1|$ is small at the almost minimum point of $f_1$ in $\text{interior} (F)$, and $f_1$ is an almost cosine function of the distance from the point.
This implies that the point is close to $x_i$ for some $i$, and contradicts to the definition of $F$.
The following claim, which asserts that the points in $M\setminus \bigcup_{ij}B_{ij}$ are not so close to $x_i$ for any $i$, is the first step to justify the above argument.
\begin{Clm}\label{c7}
There exists a positive constant $\eta(n,K,D)>0$ such that the following property holds.
Suppose that $\delta\leq \eta$.
We have $d(x,x_i)> \delta^\frac{1}{225n^2}$ for all $x\notin \bigcup_{jk} B_{jk}$ and $i=0,\ldots,N$.
In particular, $A_m\setminus \bigcup_{ij} B_{ij}=\emptyset$ for all $m\in\mathbb{Z}_{\geq0}$.
\end{Clm}
\begin{proof}[Proof of Claim \ref{c7}]
We can assume that $2\delta^\frac{1}{225n^2}\leq \delta^\frac{1}{250n^2}$.

Take arbitrary $x\in M$ such that $d(x,x_i)\leq \delta^\frac{1}{225n^2}$ holds for some $i=0,\ldots,N$.
For all $z\in M$, we get
\begin{equation}\label{dd}
d(x_i,x)+d(x,z)\leq 2 d(x_i,x)+ d(x_i,z)\leq \delta^\frac{1}{250n^2}+ d(x_i,z).
\end{equation}
If $i=0$, we take $j=1,\ldots,N$ with $x_j\in A_1$.
Then, we have $|d(x_i,x_j)-\pi|\leq \delta^\frac{1}{100n}$.
If $i\geq 1$, we take $j=0,\ldots,N$ as follows.
There exists a point $t\in [0,d(p,x_i)]$ with $\gamma_{p,x_i}(t)\in A_{m_i-1}$.
We take $j=0,\ldots,N$ with $[x_j]=[\gamma_{p,x_i}(t)]$.
Then, we have
\begin{equation*}
d(p,\gamma_{p,x_i}(t))+d(\gamma_{p,x_i}(t),x_i)=d(p,x_i).
\end{equation*}
Thus, by $x_j\sim\gamma_{p,x_i}(t)$, we get
\begin{equation*}
\begin{split}
d(x_i,x_j)\leq d(x_i,\gamma_{p,x_i}(t))+d(\gamma_{p,x_i}(t),x_j)=&d(p,x_i)-d(p,\gamma_{p,x_i}(t))+d(\gamma_{p,x_i}(t),x_j)\\
\leq &\pi+ 2\delta^\frac{1}{100n}+\delta^\frac{1}{200n^2}<\pi+3\delta^\frac{1}{200n^2}.
\end{split}
\end{equation*}
By Claim \ref{c4}, we get $\left|d(x_i,x_j)-\pi\right|\leq \delta^\frac{1}{200n^2}$.
For both cases, $\left|d(x_i,x_j)-\pi\right|\leq \delta^\frac{1}{200n^2}$ holds.
Putting $z=x_j$ into (\ref{dd}), we get $x\in B_{ij}$.
Thus, we get the claim.
\end{proof}

Let us estimate the distance $d(x,x_j)$ for $x\notin\bigcup_{kl} B_{kl}$.
\begin{Clm}\label{c8}
There exists a positive constant $\eta(n,K,D)>0$ such that the following property holds.
Suppose that $\delta\leq \eta$. For all $x\notin \bigcup_{kl} B_{kl}$, there exists $i=0,\ldots,N$ such that
\begin{equation*}
\begin{split}
\delta^\frac{1}{225n^2}\leq d(x,x_i)&\leq \pi+\delta^\frac{1}{200n^2},\\
d(x,x_j)&\geq \pi+\delta^\frac{1}{220n^2}\, \text{ for all $j$ with $(-1)^{m_i}=-(-1)^{m_j}$}.
\end{split}
\end{equation*}
\end{Clm}
\begin{proof}[Proof of Claim \ref{c8}]
Take arbitrary $x\notin \bigcup_{kl} B_{kl}$.
We first prove that there exists $i=0,\ldots,N$ with $d(x_i,x)\leq \pi+\delta^\frac{1}{200n^2}$.
Take an integer $m\in\mathbb{Z}_{\geq0}$ with $d(p,x)-m\pi \in[0,\pi]$.
By Claim \ref{c7}, we have $x\notin A_m$ and $x\notin A_{m+1}$. Thus, $ d(p,x)-m\pi \in[\delta^\frac{1}{100n},\pi-\delta^\frac{1}{100n}]$.
There exists a point $t\in [0,d(p,x)]$ with $\gamma_{p,x}(t)\in A_m$.
Take $i=0,\ldots,N$ with $[x_i]=[\gamma_{p,x}(t)]$.
Since we have $d(p,x)=d(p,\gamma_{p,x}(t))+d(\gamma_{p,x}(t),x)$,
we get
\begin{equation*}
\begin{split}
d(x_i,x)\leq &d(\gamma_{p,x}(t),x_i)+d(\gamma_{p,x}(t),x)\\
\leq &d(p,x)-d(p,\gamma_{p,x}(t))+\delta^\frac{1}{200n^2}\leq \pi+\delta^\frac{1}{200n^2}.
\end{split}
\end{equation*}

We fix an $i=0,\ldots,N$ with $d(x_i,x)\leq\pi+\delta^\frac{1}{200n^2}$.
Take arbitrary $j$ with $(-1)^{m_i}=-(-1)^{m_j}$.

We first consider the case $k_{ij}\geq 3$.
Then, we get
\begin{equation}\label{d0}
d(x,x_j)\geq d(x_i,x_j)-d(x,x_i)\geq 2\pi-2\delta^\frac{1}{200n^2}\geq\pi+\delta^\frac{1}{220n^2}.
\end{equation}

We next consider the case $k_{ij}=1$.
By Claim \ref{c3} and $m_i=m_j\pm 1$,
\begin{equation*}
\begin{split}
\left|f_1(x)-(-1)^{m_i} \cos d(x_i,x)\right|\leq &C\delta^\frac{1}{75n^2},\\
\left|f_1(x)+(-1)^{m_i} \cos d(x_j,x)\right|\leq &C\delta^\frac{1}{75n^2}.
\end{split}
\end{equation*}
Thus, we get
\begin{equation*}
\left|\cos d(x_i,x)+\cos d(x_j,x)\right|\leq C\delta^\frac{1}{75n^2}.
\end{equation*}
We consider the following five cases:
\begin{itemize}
\item[(a)] $\delta^\frac{1}{225n^2}\leq d(x_i,x)\leq \pi$ and $\delta^\frac{1}{225n^2}\leq d(x_j,x)\leq \pi$ hold.
\item[(b)] $\pi<d(x_i,x)\leq\pi+ \delta^\frac{1}{200n^2}$ and $\delta^\frac{1}{225n^2}\leq d(x_j,x)\leq \pi$ hold.
\item[(c)] $\delta^\frac{1}{225n^2}\leq d(x_i,x)\leq \pi$ and $\pi< d(x_j,x)\leq 2\pi$ hold.
\item[(d)] $\pi<d(x_i,x)\leq\pi+ \delta^\frac{1}{200n^2}$ and $\pi< d(x_j,x)\leq 2\pi$ hold.
\item[(e)] $d(x_j,x)>2\pi$ holds.
\end{itemize}
We can show that the cases (a) and (b) cannot occur, and $d(x,x_j)\geq \pi+\delta^\frac{1}{220n^2}$ holds for the cases (c), (d) and (e), and so we get the claim.
We only prove the assertion for the case (c).
Similarly to the case (c), we can show that the case (a) contradicts to the assumption $x\notin B_{ij}$,
the case (b) contradicts to Claim \ref{c7}, and the assertion holds for the case (d).
The assertion is trivial for the case (e).



Suppose that the case (c) holds.
Since $\left|\cos d(x_i,x)-\cos(d(x_j,x)-\pi)\right|\leq C\delta^\frac{1}{75n^2}$,
we get 
\begin{equation*}
\left|
\pi+d(x_i,x)-d(x_j,x)
\right|
\leq C\delta^\frac{1}{150n^2},
\end{equation*}
by Lemma \ref{p2j}.
Thus, we get
\begin{equation*}
\begin{split}
d(x_j,x)\geq d(x_i,x)+\pi-C\delta^\frac{1}{150n^2}
\geq \pi+\delta^\frac{1}{225n^2}-C\delta^\frac{1}{150n^2}\geq \pi+\delta^\frac{1}{220n^2}.
\end{split}
\end{equation*}
Thus, we get the assertion for the case (c).
\end{proof}

Let us complete the proof of the assertion $\bigcup_{kl} B_{kl}=M$.
\begin{Clm}\label{c9}
There exists a positive constant $\eta(n,K,D)>0$ such that the following property holds.
If $\delta\leq \eta$, then $\bigcup_{kl} B_{kl}=M$.
\end{Clm}
\begin{proof}[Proof of Claim \ref{c9}]
Suppose that $\bigcup_{kl} B_{kl}\neq M$.
We fix a point $y_0\notin \bigcup_{kl} B_{kl}$ and $i=0,\ldots,N$ of Claim \ref{c8}.
By Claim \ref{c3} and Lemma \ref{p2j}, we have
\begin{equation}\label{sa1}
\begin{split}
(-1)^{m_i}f_1(y_0)\leq \cos d(x_i,y_0)+ C\delta^\frac{1}{75n^2}
\leq& \cos\delta^\frac{1}{225n^2}+ C\delta^\frac{1}{75n^2}\\
\leq&1-\frac{1}{9}\delta^\frac{2}{225n^2} + C\delta^\frac{1}{75n^2}.
\end{split}
\end{equation}
Define
\begin{equation*}
\begin{split}
F:=\Big\{y\in M: d(y,x_j)&\geq \delta^\frac{1}{210n^2} \text{ for $j$ such that $(-1)^{m_j}=(-1)^{m_i}$, and } \\
d(y,x_j)&\geq \pi+\delta^\frac{1}{210n^2} \text{ for $j$ such that $(-1)^{m_j}=-(-1)^{m_i}$}
\Big\}.\\
\widetilde{F}:=\Big\{y\in M: d(y,x_j)&\geq \delta^\frac{1}{215n^2} \text{ for $j$ such that $(-1)^{m_j}=(-1)^{m_i}$, and } \\
d(y,x_j)&\geq \pi+\delta^\frac{1}{215n^2} \text{ for $j$ such that $(-1)^{m_j}=-(-1)^{m_i}$}
\Big\}.
\end{split}
\end{equation*}
Then, we have $\widetilde{F}\subset \text{interior} (F)$ and $y_0\in \widetilde{F}$.

In the following, we take points $y_1,y_2,y_3\in M$.
For the reader's convenience, we list how we take them and give a brief sketch of the proof of the claim here.
\begin{itemize}
\item We take $y_1\in \widetilde{F}$ so that $(-1)^{m_i}f_1(y_1)=\min_{y\in\widetilde{F}}(-1)^{m_i}f_1(y)$.
\item We take $y_2\in F\setminus \text{interior}(F)$ so that $((-1)^{m_i}f_1-(-1)^{m_i}f_1(y_1)+\delta^\frac{1}{6n})^{-2}|\nabla f_1|^2$ takes its maximum value at $y_2$ in $F$ when the function does not take its maximum in $\text{interior} (F)$.
\item We take $y_3\in Q\cap \text{interior}(F)$ so that $d(y_1,y_3)\leq  \left(2C_{\Ck}\delta^\frac{1}{6}/C_{\Cm}\right)^\frac{1}{n}$.
\end{itemize}
We first show that $\min_{y\in F}(-1)^{m_i}f_1(y)=\min_{y\in \widetilde{F}}(-1)^{m_i}f_1(y)$, and so $y_1$ is the local minimum point of $f_1$.
Then, we show that $|\nabla f|(y_3)$ is small by using Lemma \ref{p2d}.
Similarly to (\ref{2d}) in Proposition \ref{p2i}, we get that
$f_1$ is close to $f_1(y_3)\cos d(y_3,\cdot)$ in $L^\infty$ sense.
This gives that $|f_1(y_3)|$ is close to $1$, and $y_3$ is close to $x_j$ for some $j=0,\ldots,N$.
This contradicts to $y_3\in F$.

We return to the proof.
For any $y\in F\setminus \widetilde{F}$, we have either
\begin{equation*}
\begin{split}
 \delta^\frac{1}{210n^2} \leq d(y,x_j)&< \delta^\frac{1}{215n^2}  \text{ for some $j$ with $(-1)^{m_j}=(-1)^{m_i}$, or } \\
 \pi+\delta^\frac{1}{210n^2} \leq d(y,x_j)&<\pi+\delta^\frac{1}{215n^2}\text{ for some $j$ with $(-1)^{m_j}=-(-1)^{m_i}$},
\end{split}
\end{equation*}
and so, by Claim \ref{c3} and Lemma \ref{p2j},
\begin{equation}\label{sa2}
\begin{split}
(-1)^{m_i}f_1(y)\geq& \cos\delta^\frac{1}{215n^2} -C\delta^\frac{1}{75n^2}
\geq1 -C\delta^\frac{2}{215n^2}.
\end{split}
\end{equation}
We can assume that $\frac{1}{9}\delta^\frac{2}{225n^2}
\geq C\delta^\frac{2}{215n^2}  + C\delta^\frac{1}{75n^2}+\delta^\frac{1}{75n^2}$.
Then, by (\ref{sa1}) and (\ref{sa2}),
\begin{equation}\label{d5}
(-1)^{m_i}f_1(y)-(-1)^{m_i}f_1(y_0)\geq \delta^\frac{1}{75n^2}
\end{equation} for all $y\in F\setminus \widetilde{F}$.
In particular, we get 
\begin{equation*}
\min_{y\in F}(-1)^{m_i}f_1(y)=\min_{y\in \widetilde{F}}(-1)^{m_i}f_1(y).
\end{equation*}
Take a point $y_1\in \widetilde{F}$ such that $(-1)^{m_i}f_1(y_1)=\min_{y\in\widetilde{F}}(-1)^{m_i}f_1(y).$
We first suppose that the function on $F$, $((-1)^{m_i}f_1-(-1)^{m_i}f_1(y_1)+\delta^\frac{1}{6n})^{-2}|\nabla f_1|^2$ takes its maximum value at some point in $\text{interior}(F)$.
By Lemma \ref{p2d}, we have
\begin{equation*}
|\nabla f_1|^2\leq \frac{C}{\delta^\frac{1}{6n}}\left((-1)^{m_i}f_1-(-1)^{m_i}f_1(y_1)+\delta^\frac{1}{6n}\right)^2
\end{equation*}
on $\text{interior}(F)$.
We next suppose that $((-1)^{m_i}f_1-(-1)^{m_i}f_1(y_1)+\delta^\frac{1}{6n})^{-2}|\nabla f_1|^2$ takes its maximum value at some point $y_2\in F\setminus \text{interior}(F)$.
By (\ref{d5}), we have
\begin{equation*}
(-1)^{m_i}f_1(y_2)-(-1)^{m_i}f_1(y_1)+\delta^\frac{1}{6n}\geq (-1)^{m_i}f_1(y_2)-(-1)^{m_i}f_1(y_0)\geq\delta^\frac{1}{75n^2}
\end{equation*}
and so, we get
\begin{equation*}
|\nabla f_1|^2\leq \frac{C}{\delta^\frac{2}{75n^2}}\left((-1)^{m_i}f_1-(-1)^{m_i}f_1(y_1)+\delta^\frac{1}{6n}\right)^2
\end{equation*}
on $F$ by $\|\nabla f\|_\infty\leq C$.
Thus, we get 
\begin{equation*}
|\nabla f_1|^2\leq \frac{C}{\delta^\frac{1}{6n}}\left((-1)^{m_i}f_1-(-1)^{m_i}f_1(y_1)+\delta^\frac{1}{6n}\right)^2
\end{equation*}
on $\text{interior}(F)$ for both cases.
There exists a point $y_3\in Q$ with $d(y_1,y_3)\leq  \left(2C_{\Ck}\delta^\frac{1}{6}/C_{\Cm}\right)^\frac{1}{n}$.
We can assume that 
$\delta^\frac{1}{215n^2}-\left(2C_{\Ck}\delta^\frac{1}{6}/C_{\Cm}\right)^\frac{1}{n}> \delta^\frac{1}{210n^2}$.
Then, we have $y_3\in \text{interior}(F)$.
Thus, we get 
\begin{equation*}
|\nabla f_1|(y_3)\leq C \delta^\frac{1}{12n}.
\end{equation*}
Therefore, similarly to (\ref{2d}) in Proposition \ref{p2i}, for all $x\in M$,
\begin{equation*}
\left|f_1(x)-f_1(y_3)\cos d(y_3,x)\right|\leq C\delta^\frac{1}{12n},
\end{equation*}
and so
\begin{equation*}
\left| \cos d(p,x)-f_1(y_3)\cos d(y_3,x)\right|\leq C\delta^\frac{1}{48n}.
\end{equation*}
Putting 
$x=p$, we get
\begin{equation}\label{d6}
\left|1-f_1(y_3)\cos d(p,y_3)\right|\leq C\delta^\frac{1}{48n},
\end{equation}
and so
\begin{equation*}
1 \leq|f_1(y_3)|+ C\delta^\frac{1}{48n}.
\end{equation*}
Putting $x=y_3$, we get
\begin{equation*}
\left|\cos d(p,y_3)-f_1(y_3)\right|\leq C\delta^\frac{1}{48n},
\end{equation*}
and so
\begin{equation*}
|f_1(y_3)|\leq1+C\delta^\frac{1}{48n}.
\end{equation*}
Thus, we get 
\begin{equation}\label{d7}
\left|1-|f_1(y_3)|\right|\leq C\delta^\frac{1}{48n}.
\end{equation}
Take $\tau\in\{0,1\}$ such that $|f_1(y_3)|=(-1)^\tau f_1(y_3)$.
By (\ref{d6}) and (\ref{d7}), we get
\begin{equation*}
\left|1-\cos(d(p,y_3)+\tau\pi)\right|\leq C\delta^\frac{1}{48n}.
\end{equation*}
Take an integer $k\in \mathbb{Z}$ such that $d(p,y_3)+\tau\pi-2k\pi\in [-\pi,\pi]$.
Then, $k\geq \tau$.
By Lemma \ref{p2j},
\begin{equation*}
\left|d(p,y_3)-(2k-\tau)\pi\right|\leq C\delta^\frac{1}{96n}.
\end{equation*}
We can assume that $C\delta^\frac{1}{96n}\leq \delta^\frac{1}{100n}$.
If $2k=\tau$, then $\tau=0$ and $d(p,y_3)\leq \delta^\frac{1}{100n}$.
This contradicts to $y_3\in F$.
If $2k>\tau$, then $y_3\in A_{2k-\tau}$, and so $d(y_3,x_j)\leq \delta^\frac{1}{200n^2}$ for some $j=1,\ldots, N$.
This also contradicts to $y_3\in F$.
Therefore, we get $\bigcup_{kl} B_{kl}=M$.
\end{proof}
The purpose of the following two claims is to prove the remaining part of (iv).
\begin{Clm}\label{c10}
There exists a positive constant $\delta_{\Db}(n,K,D)>0$ such that the following property holds.
Suppose that $\delta\leq \delta_{\Db}$.
For any $x\in M$ with $d(x,x_i)\leq \pi -\delta^\frac{1}{250n^2}$ for some $i=0,\ldots, N$, we have $B_{\delta^\frac{1}{200n^2}}(x)\cap B_{kl}=\emptyset$ for any $k,l=0,\ldots, N$ with $i\neq k$ and $i\neq l$.
\end{Clm}
\begin{proof}[Proof of Claim \ref{c10}]
We can assume that $5\delta^\frac{1}{200^2}\leq \delta^\frac{1}{250n^2}$.

Take arbitrary $y\in B_{\delta^\frac{1}{200n^2}}(x)$ and $k,l=0,\ldots, N$ with $i\neq k$, $i\neq l$ and $|d(x_k,x_l)-\pi|\leq \delta^\frac{1}{200n^2}$.
Then,
\begin{equation*}
\begin{split}
d(x_k,y)+d(y,x_l)\geq &d(x_k,x_i)+d(x_i,x_l)-2d(y,x_i)\\
> &\pi+2\delta^\frac{1}{250n^2}-4\delta^\frac{1}{200n^2}\geq d(x_k,x_l)+\delta^\frac{1}{250n^2}
\end{split}
\end{equation*}
by (iii),
and so $y\notin  B_{kl}$.
Thus, we get the claim.
\end{proof}
\begin{Clm}\label{c11}
We have that $\widetilde{B}_{ij}$ is an open subset of $M$ for each $i,j$, and $\widetilde{B}_{ij}\cap B_{kl}=\emptyset$ holds if $\{k,l\}\neq \{i,j\}$.
\end{Clm}
\begin{proof}[Proof of Claim \ref{c11}]
Take arbitrary
$x\in\widetilde{B}_{ij}$.
Take $0<r<\delta^\frac{1}{200n^2}$ such that
$d(x,x_i)-r>3\delta^\frac{1}{250n^2}$ and  $d(x,x_j)-r> 3\delta^\frac{1}{250n^2}$.
Then, 
\begin{equation*}
d(x,x_i)\leq d(x_i,x_j)-d(x_j,x)+\delta^\frac{1}{250n^2}\leq\pi -\delta^\frac{1}{250n^2}.
\end{equation*}
Similarly, $d(x,x_j)\leq\pi -\delta^\frac{1}{250n^2}$.
Thus, by Claim \ref{c10}, 
\begin{equation}\label{d8}
B_r(x)\cap B_{kl}=\emptyset
\end{equation}
for any $k,l=0,\dots,N$ with $\{k,l\}\neq \{i,j\}$, and so $B_r(x)\subset B_{ij}$ by Claim \ref{c9}.
Since $d(y,x_i)>3\delta^\frac{1}{250n^2}$ and $d(y,x_j)>3\delta^\frac{1}{250n^2}$ for all $y\in B_r(x)$, we have $B_r(x)\subset \widetilde{B}_{ij}$. Thus, $\widetilde{B}_{ij}$ is open.
Moreover, (\ref{d8}) implies $\widetilde{B}_{ij}\cap B_{kl}=\emptyset$ for any $k,l=0,\dots,N$ with $\{k,l\}\neq \{i,j\}$.
\end{proof}
Finally, we show (v).
\begin{Clm}\label{c12}
Suppose that numbers $i,j,k=0,\ldots, N$ satisfy $j\neq k$, $\left|d(x_i,x_j)-\pi\right|\leq \delta^\frac{1}{200n^2}$ and $\left|d(x_i,x_k)-\pi\right|\leq \delta^\frac{1}{200n^2}$.
Then, for any $x\in B_{ij}$ and $y\in B_{ik}$, we have $d(x,y)\geq d(x,x_i)+d(x_i,y)-5\delta^\frac{1}{250n^2}$.
\end{Clm}
\begin{proof}[Proof of Claim \ref{c12}]
We have
\begin{equation*}
\begin{split}
d(x_j,x_k)\geq &2\pi-\delta^\frac{1}{200n^2},\\
d(x_j,x)\leq &d(x_i,x_j)-d(x_i,x)+\delta^\frac{1}{250n^2}\leq \pi-d(x_i,x)+2\delta^\frac{1}{250n^2},\\
d(x_k,y)\leq &d(x_i,x_k)-d(x_i,y)+\delta^\frac{1}{250n^2}\leq \pi-d(x_i,y)+2\delta^\frac{1}{250n^2}.
\end{split}
\end{equation*}
Thus, we get
\begin{equation*}
d(x,y)\geq d(x_j,x_k)-d(x_j,x)-d(x_k,y)\geq d(x,x_i)+d(x_i,y)-5\delta^\frac{1}{250n^2}.
\end{equation*}
This gives the claim.
\end{proof}
These claims imply Proposition \ref{p2k}.
\end{proof}
We consider the case $N=1$ in Proposition \ref{p2k} and give the diameter estimate.
\begin{Prop}\label{p4a}
Take an integer $n\geq 2$ and positive real numbers $K>0$, $D>0$ and $0<\delta\leq \delta_{\Db}$.
Let $(M,g)$ be an $n$-dimensional closed Riemannian manifold with $\Ric\geq -Kg$ and $\diam(M)\leq D$.
Suppose that a non-zero function $f\in C^\infty(M)$ satisfies $\|\nabla^2 f+f g\|_2\leq \delta\|f\|_2$.
Take a point $p\in M$ of Proposition \ref{p2i} and suppose that $N=1$ in Proposition \ref{p2k}.
Then, 
$\left|\diam(M)-\pi\right|\leq 2\delta^\frac{1}{250n^2}$ and
there exists a point $q\in M$ such that
\begin{equation*}
\begin{split}
\left|d(p,q)-\pi\right|\leq &\delta^\frac{1}{200n^2},\\
d(p,x)+d(x,q)\leq &d(p,q)+\delta^\frac{1}{250n^2} \text{ for all $x\in M$}.
\end{split}
\end{equation*}
\end{Prop}
\begin{proof}
Putting $q=x_1$, we have $|d(p,q)-\pi|\leq \delta^\frac{1}{200n^2}$ and $d(p,x)+d(x,q)\leq d(p,q)+\delta^\frac{1}{250n^2}$ for all $x\in M$ by Proposition \ref{p2k}.
Take arbitrary $x,y\in M$.
Then,
\begin{equation*}
\begin{split}
d(x,y)\leq d(x,p)+d(p,y)\leq &2d(p,q)-d(x,q)-d(y,q)+2\delta^\frac{1}{250n^2}\\
\leq& 2\pi-d(x,q)-d(y,q)+4\delta^\frac{1}{250n^2}.
\end{split}
\end{equation*}
If $d(x,q)+d(y,q)\geq \pi+2\delta^\frac{1}{250n^2}$, then
$d(x,y)\leq \pi+2\delta^\frac{1}{250n^2}$.
If $d(x,q)+d(y,q)< \pi+2\delta^\frac{1}{250n^2}$, then
$d(x,y)< \pi+2\delta^\frac{1}{250n^2}$.
For both cases, we get $d(x,y)\leq \pi+2\delta^\frac{1}{250n^2}$, and so
$\diam(M)\leq  \pi+2\delta^\frac{1}{250n^2}$.
Since $d(p,q)\geq \pi-\delta^\frac{1}{200n^2}$, we get
$\left|\diam(M)-\pi\right|\leq 2\delta^\frac{1}{250n^2}$.
\end{proof}

We can apply the Cheeger-Colding result \cite{CC2} to each $B_{ij}$ of Proposition \ref{p2k}.
Then, we get that each $B_{ij}$ is close to the spherical suspension.
Moreover, we can show that we always have $N=1$.
See Appendix A for details.
Note that our proofs of the main theorems does not rely on that result, and we give easier proof of $N=1$ under the pinching condition of Main Theorem 1' (see the proof of Lemma \ref{p32c}). 

\section{Pinching condition on a subspace of $C^\infty(M)$}
The goal of this section is to show that if an $n$-dimensional closed Riemannian manifold $(M,g)$ with $\Ric\geq -Kg$ and $\diam(M)\leq D$ ($K,D>0$) admits an $(n+1)$-dimensional subspace $V\subset C^\infty(M)$ such that $\|\nabla^2 f+f g\|_2\leq \delta\|f\|_2$ holds for all $f\in V$,
then the map $\Psi\colon M\to S^n\subset \mathbb{R}^{n+1}$ defined by
$$
\Psi:=\frac{1}{\left(\sum_{s=1}^{n+1}f_s^2\right)^{1/2}}(f_1,\ldots,f_{n+1})
$$
is an approximation map, where $f_1,\ldots,f_{n+1}\in T_{(n,\delta)}(V) $ satisfy
\begin{equation*}
\|f_s\|_2^2=\frac{1}{n+1},\quad \int_M f_s f_t \,d\mu_g=0\quad \text{for $s,t=1,\ldots,n+1$ with $s\neq t$}.
\end{equation*}
Our assumption in this section is stronger than that of Main Theorem 1', and we weaken the assumption ``an $(n+1)$-dimensional subspace $V\subset C^\infty(M)$'' to ``an $n$-dimensional subspace $V\subset C^\infty(M)$'' in section 4.

\subsection{The $\delta$-pinching condition}
Let us give the definition of the pinching condition that we deal with in this section.
\begin{Def}\label{dpin}
For a closed Riemannian manifold $(M,g)$, a positive real number $\delta>0$ and a subspace $V \subset C^\infty(M)$, we say that $V$ satisfies the $\delta$-pinching condition if $\|\nabla^2 f+f g\|_2\leq \delta\|f\|_2$ holds for all $f\in V$.
\end{Def}
\begin{Rem}
Let $(M,g)$ be an $n$-dimensional closed Riemannian manifold.
If a subspace $V \subset C^\infty(M)$ satisfies the $\delta$-pinching condition and $0<\delta\leq \frac{1}{8n^3}$, then $T_{(n,\delta)}$ is injective on $V$ by Lemma \ref{p2c}.
If a subspace $V\subset C^\infty(M)$ satisfies the $\delta$-pinching condition, for all $f_1\in T_{(n,\delta)}(V)$, there exists $f\in V$ such that $f_1=T_{(n,\delta)}(f)$ holds, and we can apply Proposition \ref{p2i} and Proposition \ref{p2k} for $f$.
To express this procedure, we simply say ``Proposition \ref{p2i} (or Proposition \ref{p2k}) for $f_1$''.
\end{Rem}

\begin{Lem}\label{p31c}
Take an integer $n\geq 2$ and positive real numbers $K>0$, $D>0$ and $0<\delta\leq \delta_{\Db}$.
Let $(M,g)$ be an $n$-dimensional closed Riemannian manifold with $\Ric\geq -Kg$ and $\diam(M)\leq D$.
Suppose that a 2-dimensional subspace $V\subset C^\infty(M)$ satisfies the $\delta$-pinching condition.
Take non-zero elements  $f_1,f_2\in V$.
For $f_s$ $(s=1,2)$, we use the notation $N^s$ of Proposition \ref{p2k}.
Then, if $N^1=1$, we have $N^2=1$.
\end{Lem}
\begin{proof}
If $N^1=1$, then $\diam(M)\leq \pi+2\delta^\frac{1}{250n^2}$ by Proposition \ref{p4a}.
If $N^2\geq 2$, then $\diam(M)\geq 2\pi- \delta^\frac{1}{200n^2}$.
Thus, we get $N^2=1$.
\end{proof}
\begin{Def}
Take an integer $n\geq 2$ and positive real numbers $K>0$, $D>0$ and $0<\delta\leq \delta_{\Db}$.
Let $(M,g)$ be an $n$-dimensional closed Riemannian manifold with $\Ric\geq -Kg$ and $\diam(M)\leq D$.
Suppose that  a subspace $V\subset C^\infty(M)$ satisfies the $\delta$-pinching condition.
We say that $V$ satisfies $N(V)=1$ if $N=1$ holds for some (and so, for all) non-zero element of $V$.
\end{Def}


\subsection{The approximation map}
In this subsection, we construct an approximation map from the manifold to the standard sphere under our pinching condition.
\begin{notation}
Let $d_{S^n}$ denotes the intrinsic distance function on $S^n$ with radius $1$.
We consider the standard embedding $S^n\subset \mathbb{R}^{n+1}$.
Let $d_{\mathbb{R}^{n+1}}$ denotes the standard distance function on $\mathbb{R}^{n+1}$.
\end{notation}
The following lemma is standard.
\begin{Lem}\label{32a}
For all $x,y\in S^n\subset\mathbb{R}^{n+1}$, we have $d_{\mathbb{R}^{n+1}}(x,y)\leq d_{S^n}(x,y)\leq 3 d_{\mathbb{R}^{n+1}}(x,y)$.
\end{Lem}

Let us show that if an $(n+1)$-dimensional subspace $V\subset C^\infty(M)$ satisfies the $\delta$-pinching condition, then the image of the map $\widetilde{\Psi}\colon M\to \mathbb{R}^{n+1}$
 defined by $\widetilde{\Psi}:=(f_1,\ldots, f_{n+1})$
is close to $S^n$, where $\{f_1,\ldots,f_{n+1}\}$ is a orthonormal basis of $T_{(n,\delta)}(V)$ in $L^2$ sense.
The following lemma corresponds to \cite[Lemma 5.2]{Pe1}.
\begin{Lem}\label{32b}
Given an integer $n\geq 2$ and positive real numbers $K>0$ and $D>0$, there exists a positive constant $C(n,K,D)>0$ such that the following property holds.
Take a positive real number $0<\delta\leq \delta_{\Db}$.
Let $(M,g)$ be an $n$-dimensional closed Riemannian manifold with $\Ric\geq -Kg$ and $\diam(M)\leq D$.
Suppose that an $(n+1)$-subspace $V\subset C^\infty(M)$ satisfies the $\delta$-pinching condition.
Take elements $f_s\in T_{(n,\delta)}(V) $ $(s=1,\ldots,n+1)$ such that
\begin{equation*}
\|f_s\|_2^2=\frac{1}{n+1},\quad \int_M f_s f_t \,d\mu_g=0\quad \text{for $s,t=1,\ldots,n+1$ with $s\neq t$}.
\end{equation*}
Define $\widetilde{\Psi}\colon M\to \mathbb{R}^{n+1}$ by $\widetilde{\Psi}(x):=(f_1(x),\ldots,f_{n+1}(x))$.
Then, we have $\left|1-|\widetilde{\Psi}(x)|\right|\leq C \delta^\frac{1}{48n^2}$ for all $x\in M$.
\end{Lem}
\begin{proof}
By Proposition \ref{p2i}, for each $s=1,\ldots,n+1$, there exists $p_s\in M$ such that
\begin{equation*}
\begin{split}
\|f_s-h_s\|_{\infty}&\leq C \delta^\frac{1}{48n},\\
\|\nabla f_s-\nabla h_s\|_2&\leq  C \delta^\frac{1}{48n},
\end{split}
\end{equation*}
where we put $h_s:=\cos d(p_s,\cdot)$.

We first show the following claim.
\begin{Clm}\label{c32a}
For all $a\in M$, we have $|\widetilde{\Psi}(a)|\leq 1+C\delta^\frac{1}{48n}$.
\end{Clm}
\begin{proof}[Proof of Claim \ref{c32a}]
Take arbitrary $a\in M$.
Put $F:=\sum_{s=1}^{n+1}f_s(a) f_s$.
Then, we have $\|F\|_2^2=\frac{|\widetilde{\Psi}(a)|^2}{n+1}$ and $F\in T_{(n,\delta)}(V)$.
Thus, by Proposition \ref{p2i}, there exists $p\in M$ such that
\begin{equation*}
\|F-|\widetilde{\Psi}(a)|\cos d(p,\cdot)\|_{\infty}\leq |\widetilde{\Psi}(a)|C\delta^\frac{1}{48n}.
\end{equation*}
Since $F(a)=|\widetilde{\Psi}(a)|^2$, we get  $|\widetilde{\Psi}(a)|\leq 1+C\delta^\frac{1}{48n}$.
\end{proof}
We need the following claim \cite[Theorem 7.1]{Pe1}. Note that our sign convention of the Laplacian is different from \cite{Pe3}.
\begin{Clm}\label{c32b}
For a smooth functions $u\in C^\infty(M)$ and a non-negative continuous function $f$ with $\Delta u\leq f$, we have
\begin{equation*}
u\leq C(n,K,D)\|f\|_n+\frac{1}{\Vol M}\int_M u\,d\mu_g.
\end{equation*}
\end{Clm}
We use this claim for $u=|\widetilde{\Psi}|^2$ and $u=-|\widetilde{\Psi}|^2$.
We compute $\Delta|\widetilde{\Psi}|^2$.
\begin{equation}\label{a3a}
\begin{split}
\Delta|\widetilde{\Psi}|^2=&\Delta \sum_{s=1}^{n+1} f_s^2\\
=&2\sum_{s=1}^{n+1} f_s \Delta f_s-2\sum_{s=1}^{n+1} |\nabla f_s|^2\\
=&2\sum_{s=1}^{n+1}  (\Delta f_s-n f_s)f_s+2(n+1)(|\widetilde{\Psi}|^2-1)-2\sum_{s=1}^{n+1} (|\nabla f_s|^2+f_s^2 -1).
\end{split}
\end{equation}
We estimate each component.

We first estimate $\|(\Delta f_s-n f_s)f_s\|_{\infty}$.
Put $f_s=\sum_{t=1}^\infty \alpha_{st}\phi_t$.
Then, for each $s$, we have
\begin{equation*}
\begin{split}
\|\Delta f_s-n f_s\|_{\infty}\leq& \sqrt{\delta}\sum_{t=1}^{\infty} |\alpha_{st}|\|\phi_t\|_{\infty}\leq C \sqrt{\delta}.
\end{split}
\end{equation*}
Thus, we get
\begin{equation}\label{a3b}
 \|(\Delta f_s-n f_s)f_s\|_{\infty}\leq C\sqrt{\delta}.
\end{equation}

We next estimate
$\||\widetilde{\Psi}|^2-1\|_n$.
For $x\in M$ with $|\widetilde{\Psi}(x)|^2-1< 0$, we have $||\widetilde{\Psi}(x)|^2-1|=1-|\widetilde{\Psi}(x)|^2$.
For $x\in M$ with $|\widetilde{\Psi}(x)|^2-1\geq 0$, by Claim \ref{c32a}, we have $||\widetilde{\Psi}(x)|^2-1|=|\widetilde{\Psi}(x)|^2-1 \leq C\delta^\frac{1}{48n}\leq 1-|\widetilde{\Psi}(x)|^2+C\delta^\frac{1}{48n}$.
For both cases, we have $||\widetilde{\Psi}(x)|^2-1|\leq 1-|\widetilde{\Psi}(x)|^2+C\delta^\frac{1}{48n}$. Combining this and $\|\widetilde{\Psi}\|_2=1$, we get
\begin{equation*}
\||\widetilde{\Psi}|^2-1\|_1 \leq C \delta^\frac{1}{48n}.
\end{equation*}
Since $\||\widetilde{\Psi}|^2-1\|_\infty \leq C$, we get
\begin{equation*}
\||\widetilde{\Psi}|^2-1\|_n^n\leq C\||\widetilde{\Psi}|^2-1\|_1\leq C \delta^\frac{1}{48n}.
\end{equation*}
Thus, we get 
\begin{equation}\label{a3c}
\||\widetilde{\Psi}|^2-1\|_n\leq C\delta^\frac{1}{48n^2}.
\end{equation}

Finally, we estimate $\||\nabla f_s|^2+f_s^2 -1\|_n$.
For almost all point in $M$, we have
\begin{equation*}
\begin{split}
||\nabla f_s|^2+f_s^2 -1|
\leq &||\nabla f_s|^2 - |\nabla h_s|^2| +|f_s^2-h_s^2|\\
\leq& C|\nabla f_s-\nabla h_s| +C|f_s-h_s|
\leq C.
\end{split}
\end{equation*}
Thus, we get 
\begin{equation*}
\begin{split}
\||\nabla f_s|^2+f_s^2 -1\|_n^n
\leq &C \||\nabla f_s|^2+f_s^2 -1\|_1
\leq C\delta^\frac{1}{48n}.
\end{split}
\end{equation*}
Therefore, we get
\begin{equation}\label{a3d}
\||\nabla f_s|^2+f_s^2 -1\|_n\leq C\delta^\frac{1}{48n^2}.
\end{equation}

By (\ref{a3a}), (\ref{a3b}), (\ref{a3c}) and (\ref{a3d}), we get
\begin{equation*}
\|\Delta|\widetilde{\Psi}|^2\|_n\leq C \delta^\frac{1}{48n^2}.
\end{equation*}
Thus, by Claim \ref{c32b}, we get $||\widetilde{\Psi}(x)|-1|\leq ||\widetilde{\Psi}(x)|^2-1|\leq C\delta^\frac{1}{48n^2}$.
\end{proof}

Let us show the normalized map $\Psi:=\widetilde{\Psi}/|\widetilde{\Psi}|\colon M\to S^n$ is a Hausdorff approximation map. 
We first get an approximation of the cosine of the distance, i.e.,
$|\cos d_{S^n}(\Psi(a_1),\Psi(a_2))-\cos d(a_1,a_2)|$ is small for all $a_1,a_2\in M$.
By showing that $N=1$ holds in Proposition \ref{p2k} under our pinching condition, we get $\diam (M)\leq \pi +2\delta^{\frac{1}{250n^2}}$ and an approximation of the distance.
The following lemma corresponds to \cite[Lemma 6.1]{Pe1}.

\begin{Lem}\label{p32c}
Given an integer $n\geq 2$ and positive real numbers $K>0$ and $D>0$, there exist  positive constants $0<\delta_{\Dp}(n,K,D)\leq \delta_{\Db}$ and $C(n,K,D)>0$ such that the following properties hold.
Take a positive real number $0<\delta\leq\delta_{\Dp}$.
Let $(M,g)$ be an $n$-dimensional closed Riemannian manifold with $\Ric\geq -Kg$ and $\diam(M)\leq D$.
Suppose that an $(n+1)$-subspace $V\subset C^\infty(M)$ satisfies the $\delta$-pinching condition.
Take elements $f_s\in T_{(n,\delta)}(V) $ $(s=1,\ldots,n+1)$ such that
\begin{equation*}
\|f_s\|_2^2=\frac{1}{n+1},\quad \int_M f_s f_t \,d\mu_g=0\quad \text{for $s,t=1,\ldots,n+1$ with $s\neq t$}.
\end{equation*}
Define $\widetilde{\Psi}\colon M\to \mathbb{R}^{n+1}$ by $\widetilde{\Psi}(x):=(f_1(x),\ldots,f_{n+1}(x))$.
\begin{itemize}
\item[(i)] For all $x\in M$, we have $|\widetilde{\Psi}(x)|>0$.
\item[(ii)] Define $\Psi \colon M\to S^n$ by $\Psi:=\widetilde{\Psi}/|\widetilde{\Psi}|$. Then, $\Psi$ is a $C\delta^{\frac{1}{250n^2}}$-Hausdorff approximation map. In particular, we get that $d_{GH}(M,S^n)\leq C\delta^{\frac{1}{250n^2}}$.
\end{itemize}
\end{Lem}
\begin{proof}
By Lemma \ref{32b}, we get (i) immediately.

We prove (ii).
We first show that $\Psi$ is almost surjective.
\begin{Clm}\label{eas1}
For any $u\in S^n$, there exists $x\in M$ such that $d_{S^n}(\Psi(x),u)\leq C\delta^{\frac{1}{96n^2}}$.
\end{Clm}
\begin{proof}[Proof of Claim \ref{eas1}]
Take arbitrary $u\in S^n$, and set $f=\sum_{s=1}^{n+1}u_s f_s$.
Then, we have $f\in T_{(n,\delta)}(V)$ and $\|f\|_2^2=\frac{1}{n+1}$.
Take $p$ of Proposition \ref{p2i} for $f$.
Since $f=u\cdot \widetilde{\Psi}$, we get
$$
|u\cdot \widetilde{\Psi}(x)-\cos d(p,x)|\leq C\delta^{\frac{1}{48n}}
$$
for all $x\in M$
by Proposition \ref{p2i}.
Thus, by Lemma \ref{32b}, we get
$$
|\cos d_{S^n}(u,\Psi(x))-\cos d(p,x)|\leq C\delta^{\frac{1}{48n^2}}.
$$
Putting $x=p$, we get
$$
|\cos d_{S^n}(u,\Psi(p))-1|\leq C\delta^{\frac{1}{48n^2}}.
$$
This gives
$d_{S^n}(u,\Psi(p))\leq C\delta^{\frac{1}{96n^2}}$ by Lemma \ref{p2j}, and we get the claim.
\end{proof}
\begin{Clm}\label{eas2}
We have the following:
\begin{itemize}
\item[(a)] For any $a\in M$, there exists $q_a\in M$ such that
$\|\cos d_{S^n}(\Psi(a),\Psi(\cdot))-\cos d(q_a, \cdot)\|_\infty \leq C\delta^{\frac{1}{75n^2}}$ and $d(q_a,a)\leq C\delta^{\frac{1}{150n^2}}$ hold.
\item[(b)] If we have $N(V)=1$, for any $a\in M$, there exists $q_a\in M$ such that
$\|\cos d_{S^n}(\Psi(a),\Psi(\cdot))-\cos d(q_a, \cdot)\|_\infty\leq C\delta^{\frac{1}{48n^2}}$ and $d(q_a,a)\leq C\delta^{\frac{1}{96n^2}}$ hold.
\end{itemize}
\end{Clm}
\begin{proof}[Proof of Claim \ref{eas2}]
Take arbitrary $a\in M$, and set $f:=\sum_{s=1}^{n+1}\Psi_s(a) f_s$.
Then, $f\in T_{(n,\delta)}(V)$ and $\|f\|_2^2=\frac{1}{n+1}$.
Take $x_0(:=p),x_1,\ldots,x_N\in M$ of Proposition \ref{p2k} for $f$.
Then, there exist $i,j$ such that
$|d(x_i,x_j)-\pi|\leq C\delta^{\frac{1}{200n^2}}$ and $a\in B_{ij}$ by Proposition \ref{p2k} (iv).
By Proposition \ref{p2k} (iii), we have $(-1)^{m_i}=-(-1)^{m_j}$.
Thus, we can assume that
$$
\|f-\cos d(x_i,\cdot)\|_\infty\leq C\delta^{\frac{1}{75n^2}}
$$
by Proposition \ref{p2k} (ii).
Since we have $f=\Psi(a)\cdot \widetilde{\Psi}$ by the definition of $f$, we get
$$
|\cos d_{S^n}(\Psi(a),\Psi(x))-\cos d(x_i,x)|\leq C\delta^{\frac{1}{75n^2}}.
$$
by Lemma \ref{32b}.
This gives
$|\cos d(x_i,a)|\leq C\delta^{\frac{1}{75n^2}}$.
Since we have $d(x_i,a)\leq \pi +2 \delta^{\frac{1}{250n^2}}$ by the definition of $B_{ij}$,
we get $d(x_i,a)\leq C\delta^{\frac{1}{150n^2}}$ by Lemma \ref{p2j}.
Putting $q_a:=x_i$, we get (a).

If $N=1$, then we can take $q_a:=p$, and so we get (b) by Proposition \ref{p2i} similarly to (a).
\end{proof}
\begin{Clm}\label{eas3}
We have the following:
\begin{itemize}
\item[(a)] We have $|\cos d_{S^n}(\Psi(a_1),\Psi(a_2))-\cos d(a_1,a_2)|\leq C\delta^{\frac{1}{150n^2}}$ for all $a_1,a_2\in M$.
\item[(b)] If $N(V)=1$ holds, we have $|\cos d_{S^n}(\Psi(a_1),\Psi(a_2))-\cos d(a_1,a_2)|\leq C\delta^{\frac{1}{96n^2}}$ for all $a_1,a_2\in M$.
\end{itemize}
\end{Clm}
\begin{proof}[Proof of Claim \ref{eas3}]
By Claim \ref{eas2} (a), we get
\begin{align*}
&|\cos d_{S^n}(\Psi(a_1),\Psi(a_2))-\cos d(a_1,a_2)|\\
\leq& |\cos d_{S^n}(\Psi(a_1),\Psi(a_2))-\cos d(q_{a_1},a_2)|+d(a_1, q_{a_1})
\leq C\delta^{\frac{1}{150n^2}}.
\end{align*}
Thus, we get (a).

Similarly, we get (b) by Claim \ref{eas2} (b).
\end{proof}

Let us assume that $N(V)=1$ holds, and complete the proof in this case.
Take arbitrary $a_1,a_2\in M$.
Then, we have $d(a_1,a_2)\leq \pi +2\delta^\frac{1}{250n^2}$ by Proposition \ref{p4a}.
If $d(a_1,a_2)\leq \pi$, we get $| d_{S^n}(\Psi(a_1),\Psi(a_2))-d(a_1,a_2)|\leq C \delta^\frac{1}{192n^2}$
by Claim \ref{eas3} (b) and Lemma \ref{p2j}.
If $\pi\leq d(a_1,a_2)\leq \pi +2\delta^\frac{1}{250n^2}$, we get
$|\cos d_{S^n}(\Psi(a_1),\Psi(a_2))+1|\leq C \delta^\frac{1}{125n^2}$.
Thus, by Lemma \ref{p2j}, we get
$|d_{S^n}(\Psi(a_1),\Psi(a_2))- d(a_1,a_2)|\leq |d_{S^n}(\Psi(a_1),\Psi(a_2))- \pi|+2\delta^\frac{1}{250n^2}\leq C\delta^\frac{1}{250n^2}$.
Therefore, we get $|d_{S^n}(\Psi(a_1),\Psi(a_2))- d(a_1,a_2)|\leq C\delta^\frac{1}{250n^2}$ for all $a_1,a_2\in M$.
Combining this and Claim \ref{eas1}, we get that $\Psi$ is a $C\delta^{\frac{1}{250n^2}}$-Hausdorff approximation map, and so we get (ii).

Let us show that we have $N(V)=1$.
In the following, we use the notation of Proposition \ref{p2k} for $f_1$.
We suppose that $N\geq 2$ holds and show a contradiction.
Note that $N\leq N_3(n,K,D)$.
Set 
\begin{equation*}
N_4(n,K,D):=\begin{pmatrix}
N_3+1\\
2
\end{pmatrix}+1
=\frac{N_3(N_3+1)}{2}+1.
\end{equation*}
We can take points 
$$X_1,\ldots,X_{N_4}\in \{(0,u_2.\ldots,u_{n+1}): u_2^2+\ldots +u_{n+1}^2=1\}\subset S^n$$
such that
$d_{S^n}(X_i,X_j)\geq 2\pi/N_4$ for all $i,j=1,\ldots,N_4$ with $i\neq j$.
By Claim \ref{eas1}, there exists a point $y_i\in M$ with
$d_{S^n}(X_i,\Psi(y_i))\leq C\delta^{\frac{1}{96n^2}}$ for each $i$.
By Proposition \ref{p2k} (iv) and the definition of $N_4$,
there exist $i,j=1,\ldots, N$ such that $|d(x_i,x_j)-\pi|\leq C\delta^{\frac{1}{200n^2}}$ and
$\Card \{k\in\{1,\ldots,N_4\}: y_k\in B_{i j}\}\geq 2$.
Fix such $i$ and $j$.
Without loss of generality, we can assume $y_1,y_2\in B_{i j}$.
We define a unit vector $e_1$ in $\mathbb{R}^{n+1}$ by $e_1:=(1,0,\ldots,0)$.
Since $f_1=e_1\cdot \widetilde{\Psi}$,
we have
$$
|\cos d(e_1,\Psi(x_i))-(-1)^{m_i}|\leq C\delta^{\frac{1}{75n^2}}
$$
by Proposition \ref{p2k} (ii) and Lemma \ref{32b}.
Thus, we get
$d_{S^n}(\Psi(x_i),(-1)^{m_i}e_1)\leq C\delta^{\frac{1}{150n^2}}$
by Lemma \ref{p2j}.
Therefore, for each $k=1,2$, we have
\begin{align*}
&|\cos d_{S^n}(\Psi(x_i),\Psi(y_k))|\\
\leq& |\cos d_{S^n}((-1)^{m_i}e_1,X_k)|+d_{S^n}(\Psi(x_i),(-1)^{m_i}e_1)+d_{S^n}(X_k,\Psi(y_k))\leq C\delta^{\frac{1}{150n^2}}.
\end{align*}
By Claim \ref{eas3}, we get
$|\cos d(x_i,y_k)|\leq C\delta^{\frac{1}{150n^2}}$.
Since we have $d(x_i,y_k)\leq \pi+2\delta^{\frac{1}{250n^2}}$ by the definition of $B_{ij}$, we get
$|d(x_i,y_k)-\frac{\pi}{2}|\leq C\delta^{\frac{1}{300n^2}}$ by Lemma \ref{p2j}.
Similarly, we have
$|d(x_j,y_k)-\frac{\pi}{2}|\leq C\delta^{\frac{1}{300n^2}}$.

The following claim asserts that $B_{ij}$ is not isolated from the others.
\begin{Clm}\label{3c1}
There exists $k=0,\ldots,N$ such that we have $k\neq i,j$ and either $\left|d(x_i,x_k)-\pi\right|\leq \delta^\frac{1}{200n^2}$ or $\left|d(x_j,x_k)-\pi\right|\leq \delta^\frac{1}{200n^2}$.
\end{Clm}
\begin{proof}[Proof of Claim \ref{3c1}]
Suppose that for all $k=0,\ldots,N$ with $k\neq i,j$, we have $\left|d(x_i,x_k)-\pi\right|> \delta^\frac{1}{200n^2}$ and $\left|d(x_j,x_k)-\pi\right|>\delta^\frac{1}{200n^2}$.
For any $x\in M$ with $d(x,x_i)<4 \delta^\frac{1}{250n^2}$, we have $x\in B_{ij}$ by Claim \ref{c10} and the assumption.
Similarly, for $x\in M$ with $d(x,x_j)<4 \delta^\frac{1}{250n^2}$, we have $x\in B_{ij}$.
Thus, for any $x\in B_{ij}\setminus \widetilde{B}_{ij}$, we get $B_{\delta^\frac{1}{250n^2}}(x)\subset B_{ij}$.
Since $\widetilde{B}_{ij}$ is an open subset of $M$, we get that $B_{ij}$ is an open subset of $M$.
On the other hand,  $B_{ij}$ is a closed subset of $M$.
Thus, $B_{ij}=M$.
Since $x_k\notin B_{ij}$  for all $k=0,\ldots,N$ with $k\neq i,j$ by Proposition \ref{p2k} (ii) and the definition of $B_{ij}$, this contradicts to $N\geq 2$.
\end{proof}
Without loss of generality, we can assume that there exists $k=0,\ldots,N$ such that we have $k\neq i,j$ and $\left|d(x_j,x_k)-\pi\right|\leq \delta^\frac{1}{200n^2}$.
Put
$y:=\gamma_{x_j,x_k}\left(\frac{\pi}{2}\right)$.
Then, we have
$|d(y_k,y)-\pi|\leq C\delta^{\frac{1}{300n^2}}$ by Proposition \ref{p2k} (v).
Thus, by Claim \ref{eas3}, we get
$|\cos d_{S^n}(\Psi(y_k),\Psi(y))+1|\leq C\delta^{\frac{1}{300n^2}}$,
and so $d_{S^n}(\Psi(y_k),-\Psi(y))\leq C\delta^{\frac{1}{600n^2}}$ for each $k=1,2$ by Lemma \ref{p2j}.
Therefore, we get $d_{S^n}(\Psi(y_1),\Psi(y_2)))\leq C\delta^{\frac{1}{600n^2}}$.
Thus,
$$d_{S^n}(X_1,X_2)\leq d_{S^n}(\Psi(y_1),\Psi(y_2)))+C\delta^{\frac{1}{96n^2}}\leq C\delta^{\frac{1}{600n^2}}.
$$
If we take $\delta$ sufficiently small, this contradicts to $d_{S^n}(X_1,X_2)\geq 2\pi/N_4$.
Thus, we get $N=1$ and the proposition.
\end{proof}

\section{Proof of main theorems}

The purpose of this section is to complete the proofs of main theorems except for Main Theorem 3.
In subsection 4.1, we show the equivalence of our pinching condition in the previous sections and the pinching condition on the eigenvalues of a certain Laplacian $\bar{\Delta}^E$ on the vector bundle $E$ (see Definition \ref{def1} below).
As a consequence, in subsection 4.2, for the orientable case, we show that if there exists an $n$-dimensional subspace that satisfies our pinching condition, then there exists such an $(n+1)$-dimensional subspace, where $n$ denotes the dimension of the manifold.
Combining this with the result in the previous section, we show Main Theorem 1 in subsection 4.3.
In subsection 4.4, we see that the weaker condition ``$\|\nabla^2 f+\frac{\Delta f}{n} g\|_2$ is small'' implies our pinching condition ``$\|\nabla^2 f+ f g\|_2$ is small'' under the Ricci curvature bound and the assumption that the scalar curvature is equal to the constant $n (n-1)$, and give the proof of Main Theorem 2.
\subsection{Fiber argument}
In this subsection, we explain why Main Theorem 1 is equivalent to Main Theorem 1'.
More precisely, we show that for any closed Riemannian manifold $(M,g)$ with $\Ric\geq -Kg$ and $\diam(M)\leq D$ ($K,D>0$),
the following two conditions are equivalent:
\begin{itemize}
\item $\lambda_k (\bar{\Delta}^E)$ (see Definition \ref{def1} below) is small,
\item There exists a $k$-dimensional subspace $V\subset C^\infty(M)$ such that $\|\nabla^2 f+fg\|_2$ is small for all $f\in V$.
\end{itemize}
\begin{Def}\label{def1}
Take a Riemannian manifold $(M,g)$.
Consider the rank $1$ trivial bundle $M\times \mathbb{R}$ and its section $e\colon M\to M\times \mathbb{R}$ with $e(x)=(x,1)$.
We write $\mathbb{R} e=M\times \mathbb{R}$ and define $E=E_M:=TM \oplus \mathbb{R}e$.
We consider the product metric $\langle \cdot,\cdot \rangle_E$ on $E$:
\begin{equation*}
\langle X+f e, Y+h e\rangle_E:=g(X,Y)+f h,
\end{equation*}
for any $X,Y\in \Gamma(TM)$ and $f,h\in C^\infty(M)$.
We consider the following connection $\nabla^E$ on $E$:
\begin{equation*}
\nabla^E_Z (X+f e):=\nabla_Z X+ f Z+(Zf-g(Z,X))e,
\end{equation*}
for any  $X,Z \in \Gamma(TM)$ and $f\in C^\infty(M)$.
We define $(\nabla^E)^\ast\colon \Gamma(T^\ast M\otimes E)\to \Gamma(E)$ by
\begin{equation*}
(\nabla^E)^\ast (\alpha\otimes S):=-\tr_{T^\ast M} \nabla^{T^\ast M\otimes E} (\alpha\otimes S)=-\sum_{i=1}^n \left(\nabla_{e_i}\alpha\right)(e_i)S-\sum_{i=1}^n \alpha(e_i)\nabla^E_{e_i}S,
\end{equation*}
for any $\alpha\otimes S\in \Gamma(T^\ast M\otimes E)$, where $n=\dim M$ and $\{e_1,\ldots,e_n\}$ is an orthonormal basis of $TM$.
Define $\bar{\Delta}^E \colon \Gamma(E)\to \Gamma(E)$ by $\bar{\Delta}^E:=(\nabla^E)^\ast \nabla^E$.
For a function $f\in C^\infty(M)$, we define $S_f:=\nabla f+fe\in \Gamma(E)$.
When $M$ is closed, we consider the eigenvalues of $\bar{\Delta}^E$:
\begin{equation*}
0\leq \lambda_1(\bar{\Delta}^E)\leq\lambda_2(\bar{\Delta}^E)\leq \cdots \to \infty.
\end{equation*}
\end{Def}
Note that if $(M,g)$ is a closed Riemannian manifold,
then we have
$$
\int_M \langle\nabla^E S,T\rangle_{T^\ast M\otimes E}\,d\mu_g=\int_M \langle S,(\nabla^E)^\ast T\rangle_E\,d\mu_g
$$
for any $S\in\Gamma(E)$ and $T\in\Gamma(T^\ast M\otimes E)$ by the divergence theorem.

By a straightforward calculus, we have the following lemma.
\begin{Lem}\label{p41a}
For any $n$-dimensional Riemannian manifold $(M,g)$ and any section $X+f e\in \Gamma(E)$, we have
\begin{equation*}
\bar{\Delta}^E (X+fe)=\nabla^\ast \nabla X-2\nabla f +X+(\Delta f -2\nabla^\ast X +nf)e.
\end{equation*}
In particular, for any function $f\in C^\infty(M)$, we have
\begin{equation*}
\bar{\Delta} S_f=\nabla^\ast \nabla \nabla f-\nabla f +(n f-\Delta f)e.
\end{equation*}
\end{Lem}
We show that if there exists a $k$-dimensional subspace $V\subset C^\infty(M)$ such that $\|\nabla^2 f+fg\|_2$ is small for all $f\in V$, then $\lambda_k (\bar{\Delta}^E)$ is small.
\begin{Lem}\label{p41b}
Take integers $n\geq 2$ and $k\geq 1$ and a positive real number $0<\delta \leq\frac{1}{2\sqrt{n}}$.
Let $(M,g)$ be an $n$-dimensional closed Riemannian manifold.
If there exists a $k$-dimensional subspace $V\subset C^\infty(M)$ that satisfies the $\delta$-pinching condition, then for all $f\in V$, we have $\|\nabla^E S_f\|_2^2\leq \delta^2\|S_f\|_2^2$.
In particular, we have $\lambda_k(\bar{\Delta}^E)\leq \delta^2$.
\end{Lem}
\begin{proof}
Take arbitrary $f\in V$.
We have
\begin{equation*}
\begin{split}
\|\nabla^E S_f\|_2^2=\frac{1}{\Vol(M)}\int_M \langle \bar{\Delta}^E S_f, S_f\rangle_E\, d\mu_g
=&\frac{1}{\Vol(M)}\int_M |\nabla^2 f|^2-2f\Delta f +n f^2 d\mu_g\\
=&\|\nabla^2 f +f g\|_2^2\leq \delta^2 \|f\|_2^2\leq  \delta^2 \|S_f\|_2^2.
\end{split}
\end{equation*}
By the Rayleigh principle
$$\lambda_k(\bar{\Delta}^E)=\inf\left\{\sup_{S\in W\setminus\{0\}} \frac{\|\nabla^E S\|_2^2}{\|S\|_2^2}: \text{$W$ is a $k$-dimensional subspace of $\Gamma(E)$}\right\},$$
we get the lemma.
\end{proof}
The converse is also true under the assumptions on the Ricci curvature and the diameter.
The following proposition corresponds to Proposition 10 in \cite{Au}, which asserts that if $\lambda_k(\bar{\Delta}^E)$ is close to 0, then $\lambda_k(g)$ is close to $n$ under the assumption $\Ric \geq (n-1)g$ for closed $n$-dimensional Riemannian manifolds.
In that case, it is enough to consider the first derivative of functions to get the estimate of $\lambda_k(g)$ by the Rayleigh principle, and the functions defined by $\langle S,e\rangle_E$ have the required property, where $S$ is one of the eigensections of $\bar{\Delta}^E$.
However, we need to get the information about the Hessian of functions for our pinching condition, and the functions $\langle S,e\rangle_E$ are not enough for our purpose.
To overcome this point, we use the Hodge decomposition.
\begin{Prop}\label{p41c}
Given integers $n\geq 2$ and $k\geq 1$ and positive real numbers $K>0$ and $D>0$, there exist positive constants $ C(n,K,D)>0$ and $\epsilon_1(n,K,D)>0$ such that the following properties hold.
Take a positive real number $0<\epsilon\leq \epsilon_1$.
Let $(M,g)$ be an $n$-dimensional closed Riemannian manifold with $\Ric\geq -Kg$ and $\diam(M)\leq D$.
If $\lambda_k(\bar{\Delta}^E)\leq \epsilon$, then there exists a $k$-dimensional subspace $V\subset C^\infty(M)$ that satisfies the $C\epsilon^\frac{1}{2}$-pinching condition.
\end{Prop}
\begin{proof}
Let $S_1,\ldots,S_k$ be the eigensections corresponding to $\lambda_1(\bar{\Delta}^E),\ldots,\lambda_k(\bar{\Delta}^E)$.
Define $P\colon \Gamma(E)\to C^\infty(M)$ as $P(S):=\langle S,e\rangle_E$ and $\mathcal{E}:=\Span\{S_1,\ldots,S_k\}\subset \Gamma(E)$.

Note that $\bar{\Delta}^E e=ne$.
If we assume $\epsilon<n$,
we have $\int_M P(S)\,d\mu_g=\int_M\langle S,e\rangle_E\,d\mu_g=0$ for all $S\in \mathcal{E}$.
\begin{Clm}\label{c41a}
If $\epsilon<1$, then $P|_{\mathcal{E}}$ is injective.
\end{Clm}
\begin{proof}[Proof of Claim \ref{c41a}]
Take arbitrary $S\in \ker P\cap \mathcal{E}$.
Then, there exists $X\in \Gamma(TM)$ with $S=X$.
Since $\bar{\Delta}^E S=\nabla^\ast \nabla X+X-2\nabla^\ast X e$,
we get
\begin{equation*}
\epsilon \int_M|X|^2\,d\mu_g \geq \int_M\langle \bar{\Delta}^E S,S\rangle_E\,d\mu_g\geq \int_M |X|^2\,d\mu_g. 
\end{equation*}
Thus, $X=0$, and so $S=0$.
\end{proof}

Define $\widetilde{P}\colon \Gamma(E)\to C^\infty(M)$ as follows.
For arbitrary $S=\alpha + f e \in \Gamma(E)$, there exist a harmonic $1$-form $h$, a function $F\in C^\infty(M)$ with $\int_M F\,d\mu_g=0$ and a closed $2$-form $\beta$ such that $\alpha=h+d F+d^\ast \beta$ by the Hodge decomposition, where we regard $\alpha$ as a $1$-form on $M$.
Then, we define $\widetilde{P}(S):=F$. Note that $F$ is uniquely determined.
Let us show that elements in $\widetilde{P}(F)$ satisfy our pinching condition.
\begin{Clm}\label{c41b}
There exists a positive constant $\epsilon_1(n,K,D)>0$ such that
if $\epsilon\leq \epsilon_1$, then $\widetilde{P}|_\mathcal{E}$ is injective and $\|\nabla^2 F+ Fg\|_2\leq C\epsilon^\frac{1}{2}\|F\|_2$ for all $F\in \widetilde{P}(\mathcal{E})$.
\end{Clm}

\begin{proof}[Proof of Claim \ref{c41b}]

Take arbitrary $S=\alpha +f e\in \mathcal{E}$.
Note that we have 
$$
\int_M f \,d\mu_g=\int_M \langle S,e\rangle_E \,d\mu_g=0.
$$
We decompose $\alpha$ as $\alpha=h+d F+d^\ast \beta$.
Since $|S|^2=|\alpha|^2+f^2$ and $|\nabla^E S|^2=|\nabla \alpha +f g|^2+|d f-\alpha|^2$,
we get
\begin{align}
\label{411a} \|\nabla \alpha +f g\|_2^2&\leq \epsilon (\|\alpha\|_2^2+\|f\|_2^2),\\
\label{411b} \|d f-\alpha\|_2^2&\leq \epsilon (\|\alpha\|_2^2+\|f\|_2^2).
\end{align}
By the orthogonal decomposition of $T^\ast M\otimes T^\ast M$, we have 
\begin{equation}\label{411c}
\|\nabla \alpha +f g\|_2^2\geq \frac{1}{2}\|d\alpha\|_2^2+n\|f-\frac{1}{n}d^\ast \alpha \|_2^2.
\end{equation}
Since $h$, $d(f-F)$ and $d^\ast \beta$ are orthogonal to each other in $L^2$ sense, we have
\begin{equation}\label{411d}
\|d f-\alpha\|_2^2= \|h\|_2^2+\|\nabla (f-F)\|_2^2+\|d^\ast \beta\|_2^2.
\end{equation}
By the Bochner formula, (\ref{411a}), (\ref{411b}), (\ref{411c}) and (\ref{411d}), we get
\begin{align}
\label{411e} \|\nabla h\|_2^2=&-\frac{1}{\Vol(M)}\int_M \Ric(h,h)\,d\mu_g\leq K\|h\|_2^2\leq K\epsilon (\|\alpha\|_2^2+\|f\|_2^2),\\
\label{411f} \|\nabla d^\ast \beta\|_2^2\leq&\|d \alpha \|_2^2+K\|d^\ast \beta\|_2^2\leq (K+2)\epsilon(\|\alpha\|_2^2+\|f\|_2^2).
\end{align}
By the Li-Yau estimate $\lambda_1\geq C(n,K,D)>0$ \cite[p.116]{SY}, (\ref{411b}) and (\ref{411d}), we get
\begin{equation}\label{411g}
\|f-F\|_2^2\leq C\epsilon(\|\alpha\|_2^2+\|f\|_2^2).
\end{equation}
Thus, by (\ref{411a}), (\ref{411e}), (\ref{411f}) and (\ref{411g}), we get
\begin{equation}\label{411h}
\begin{split}
\|\nabla^2 F+F g\|_2&\leq \|\nabla^2 F-\nabla \alpha \|_2+\|\nabla \alpha+f g\|_2+\sqrt{n}\|f-F\|_2\\
&\leq \|\nabla h\|_2+\|\nabla d^\ast \beta \|_2+\|\nabla \alpha+f g\|_2+\sqrt{n}\|f-F\|_2\\
&\leq C \epsilon^\frac{1}{2}(\|\alpha\|_2+\|f\|_2).
\end{split}
\end{equation}
By (\ref{411b}), we have
$\|\alpha \|_2\leq \|\nabla f\|_2+\epsilon^\frac{1}{2}(\|\alpha\|_2+\|f\|_2)$. If $\epsilon \leq \frac{1}{4}$, we get
\begin{equation}\label{411i}
\|\alpha\|_2\leq 2(\|\nabla f\|_2+\|f\|_2) \leq C\|\nabla f\|_2
\end{equation}
by the Li-Yau estimate.
By (\ref{411b}), (\ref{411d}) and (\ref{411i}), we get
\begin{equation*}
\|\nabla f\|_2\leq \|\nabla F\|_2+\epsilon^\frac{1}{2}(\|\alpha\|_2+\|f\|_2)
\leq \|\nabla F\|_2+C \epsilon^\frac{1}{2} \|\nabla f\|_2.
\end{equation*}
Thus, taking $\epsilon_1$ sufficiently small, we get $\|\nabla f\|_2\leq 2\|\nabla F\|_2$.
Therefore, by Claim \ref{c41a}, we get $\widetilde{P}|_{\mathcal{E}}$ is injective and, by (\ref{411i}) and the Li-Yau estimate,
\begin{equation*}
\|\alpha \|_2+\|f\|_2\leq C \|\nabla F\|_2\leq C \|\Delta F\|_2.
\end{equation*}
Thus, by (\ref{411h}), we get
$\|\nabla^2 F+F g\|_2\leq C\epsilon^\frac{1}{2}\|\Delta F\|_2$,
and we get the claim by Lemma \ref{add}.
\end{proof}
By Claim \ref{c41b}, the subspace defined by $V:=\widetilde{P}(\mathcal{E})$ satisfies the desired property.
\end{proof}
Lemma \ref{p41b} and Proposition \ref{p41c} show the equivalence of Main Theorem 1 and Main Theorem 1'.


\subsection{Almost parallel section}
In this subsection, we show that if $\lambda_n(\bar{\Delta}^E)$ is close to $0$, then $\lambda_{n+1}(\bar{\Delta}^E)$ is close to $0$ for the orientable case under the assumption $\Ric\geq -Kg$ and $\diam(M)\leq D$.

The following proposition, which asserts that the eigensections $S_1,\ldots,S_k$ are almost orthogonal to each other for most points in $M$ if $\lambda_k(\bar{\Delta}^E)$ is small, corresponds to \cite[Lemma 11]{Au}.
\begin{Lem}\label{p42a}
Given an integer $n\geq 2$ and positive real numbers $K>0$ and $D>0$, there exist positive constants $ C(n,K,D)>0$ and $0<\delta_{\Dq}(n,K,D)\leq\delta_{\Dp}$ such that the following properties hold.
Take a positive real number $0<\delta\leq\delta_{\Dq}$ and an integer $1\leq k\leq n+2$.
Let $(M,g)$ be an $n$-dimensional closed Riemannian manifold with $\Ric\geq -Kg$ and $\diam(M)\leq D$.
Suppose that a $k$-dimensional subspace $V\subset C^\infty(M)$ satisfies the $\delta$-pinching condition.
Take elements $f_i \in T_{(n,\delta)}(V)$ $(i=1,\ldots,k)$ such that
\begin{equation*}
\|f_i\|_2^2=\frac{1}{n+1},\quad \int_M f_i f_j \,d\mu_g=0\quad \text{for $i,j=1,\ldots,k$ with $i\neq j$}.
\end{equation*}
Define $S_i:=S_{f_i}$ and
\begin{equation*}
\begin{split}
G=G(f_1,\ldots,f_k):=\Big\{x\in M: & ||S_i(x)|^2-1|\leq\delta^\frac{1}{96n}\text{ for all $i=1,\ldots,k$, and}\\
&\left|\frac{1}{2}|S_i(x)+S_j(x)|^2-1\right|\leq \delta^\frac{1}{96n},\\
&\left|\frac{1}{2}|S_i(x)-S_j(x)|^2-1\right|\leq \delta^\frac{1}{96n}\text{ for all $i\neq j$}
\Big\}.
\end{split}
\end{equation*}
Then, we have the following properties.
\begin{itemize}
\item[(i)] For all $f\in T_{(n,\delta)}(V)$ with $\|f\|_2^2=\frac{1}{n+1}$, we have $\|S_f\|_{\infty}\leq C$.
\item[(ii)] We have $\Vol(M\setminus G)\leq C\delta^\frac{1}{96n}\Vol(M)$.
\item[(iii)] For all $x\in G$ and $i\neq j$, we have $\left|\langle S_i(x),S_j(x)\rangle_E\right|\leq\delta^\frac{1}{96n}$.
\item[(iv)] For all $x\in G$, $S_1(x),\ldots,S_k(x)$ are linearly independent.
In particular, $k\leq n+1$.
\end{itemize}
\end{Lem}
\begin{proof}
Take arbitrary $f\in T_{(n,\delta)}(V)$ with $\|f\|_2^2=\frac{1}{n+1}$ and take $p\in M$ of Proposition \ref{p2i}.
Then, we get (i) by Lemma \ref{p2c}.
Put $h:=\cos d(p,\cdot)$.
For almost all point in $M$, we have
\begin{equation*}
|f^2+|\nabla f|^2-1|\leq |f+h||f-h|+(|\nabla f|+|\nabla h|)|\nabla f-\nabla h|\leq C\left(|f-h|+|\nabla f-\nabla h|\right).
\end{equation*}
By Proposition \ref{p2i} and $|S_f|^2=f^2+|\nabla f|^2$, we get 
\begin{equation}\label{422a}
\||S_f|^2-1\|_1\leq C\delta^\frac{1}{48n}.
\end{equation}

Since $f_1,\ldots,f_k$ is orthogonal to each other, we get
\begin{equation*}
\begin{split}
\left\|\frac{1}{\sqrt{2}}(f_i+f_j)\right\|_2^2=\frac{1}{n+1},&\quad \left\|\frac{1}{\sqrt{2}}(f_i-f_j)\right\|_2^2=\frac{1}{n+1},\\
\frac{1}{\sqrt{2}}(S_i+S_j)=S_{\frac{1}{\sqrt{2}}(f_i+f_j)},&\quad \frac{1}{\sqrt{2}}(S_i-S_j)=S_{\frac{1}{\sqrt{2}}(f_i-f_j)},
\end{split}
\end{equation*}
for all $i\neq j$.
Thus, by (\ref{422a}), we get
\begin{equation*}
\begin{split}
 \||S_i|^2-1\|_1\leq&C\delta^\frac{1}{48n},\\
\left\|\frac{1}{2}|S_i+S_j|^2-1\right\|_1\leq &C\delta^\frac{1}{48n},\\
\left\|\frac{1}{2}|S_i-S_j|^2-1\right\|_1\leq &C\delta^\frac{1}{48n},
\end{split}
\end{equation*}
for all $i\neq j$.
Therefore, we get
\begin{align*}
&\Vol\left(\left\{x\in M: \left||S_i(x)|^2-1\right|>\delta^{\frac{1}{96n}}\right\}\right)\\
\leq& \delta^{-\frac{1}{96n}}\int_M \left||S_i|^2-1\right|\,d\mu_g
= \delta^{-\frac{1}{96n}}\||S_i|^2-1\|_1\Vol(M)\leq C\delta^{\frac{1}{96n}}\Vol(M)
\end{align*}
for all $i$.
Similarly, we have
\begin{align*}
\Vol\left(\left\{x\in M: \left|\frac{1}{2}|S_i(x)+S_j(x)|^2-1\right|>\delta^{\frac{1}{96n}}\right\}\right)&\leq C\delta^{\frac{1}{96n}}\Vol(M),\\
\Vol\left(\left\{x\in M: \left|\frac{1}{2}|S_i(x)-S_j(x)|^2-1\right|>\delta^{\frac{1}{96n}}\right\}\right)&\leq C\delta^{\frac{1}{96n}}\Vol(M)
\end{align*}
for all $i\neq j$.
Thus, we get (ii).

For all $x\in G$ and $i,j$ with $i\neq j$, we have
\begin{equation*}
\left|\langle S_i(x),S_j(x)\rangle_E\right|
=\frac{1}{4}\left||S_i(x)+S_j(x)|^2-|S_i(x)-S_j(x)|^2\right|\leq\delta^\frac{1}{96n}.
\end{equation*}
Thus, we get (iii).

Finally, we prove (iv).
Take arbitrary $x\in G$ and $a_1,\ldots, a_k\in \mathbb{R}$ with
$a_1 S_1(x)+\cdots+a_k S_k(x)=0$.
Take $i$ with $|a_i|=\max\{|a_1|,\ldots,|a_k|\}$.
Since we have $\langle a_1 S_1(x)+\cdots+a_k S_k(x),S_i(x)\rangle_E=0$, we get
\begin{equation*}
\begin{split}
0\geq |a_i||S_i(x)|^2-\sum_{i\neq j}\left|a_j\langle S_i(x), S_j(x)\rangle_E\right|
\geq&  |a_i|(1-\delta^\frac{1}{96n})-\sum_{i\neq j}|a_i|\delta^\frac{1}{96n}\\
\geq& |a_i|(1-(n+2)\delta^\frac{1}{96n}).
\end{split}
\end{equation*}
Thus, $|a_i|=0$, and so $a_1,\ldots,a_k=0$.
This implies the linearly independence of $S_1(x),\ldots, S_k(x)$.
\end{proof}
By Proposition \ref{p41c} and Lemma \ref{p42a}, we get the following corollary immediately.
\begin{Cor}\label{p42b}
Given an integer $n\geq 2$ and positive real numbers $K>0$ and $D>0$,
there exists $\epsilon_2(n,K,D)>0$ such that 
if $(M,g)$ is an $n$-dimensional closed Riemannian manifold with $\Ric\geq -Kg$ and $\diam(M)\leq D$, then $\lambda_{n+2}(\bar{\Delta}^E)\geq \epsilon_2$.
\end{Cor}
The following corollary asserts that if $\lambda_n(\bar{\Delta}^E)$ is small, then $\lambda_{n+1}(\bar{\Delta}^E)$ is small for the orientable case.
\begin{Cor}\label{p42c}
Given an integer $n\geq 2$ and positive real numbers $K>0$ and $D>0$, there exist positive constants $ C(n,K,D)>0$ and $0<\epsilon_3(n,K,D)\leq \epsilon_1$ such that the following properties hold.
Take a positive real number $0<\epsilon\leq \epsilon_3$.
Let $(M,g)$ be an oriented $n$-dimensional closed Riemannian manifold with $\Ric\geq -Kg$ and $\diam(M)\leq D$.
If $\lambda_n(\bar{\Delta}^E)\leq \epsilon$, then $\lambda_{n+1}(\bar{\Delta}^E)\leq C \epsilon^\frac{1}{2}$.
\end{Cor}
\begin{proof}
By Proposition \ref{p41c}, there exists an $n$-dimensional subspace $V\subset C^\infty(M)$ that satisfies th $C\epsilon^\frac{1}{2}$-pinching condition.
Put $\delta:=C\epsilon^\frac{1}{2}$.
We can assume $\delta\leq\delta_{\Dq}$.
Take elements $f_i \in T_{(n,\delta)}(V)$ $(i=1,\ldots,n)$ such that
\begin{equation*}
\|f_i\|_2^2=\frac{1}{n+1},\quad \int_M f_i f_j \,d\mu_g=0\quad \text{for $i,j=1,\ldots,n$ with $i\neq j$}.
\end{equation*}
Define $S_i:=S_{f_i}$.

By the orientation, we can identify $E$ and $\bigwedge^n E$.
Note that this identification preserves the metric and the connection.
Define $S_{n+1}:=S_1\wedge\cdots\wedge S_n\in \Gamma(E)$.
By Lemma \ref{p2c}, $T_{(n,\delta)}(V)$ satisfies the $C\sqrt{\delta}$-pinching condition, and so, by Lemma \ref{p41b}, we have
\begin{equation}\label{422d}
\begin{split}
\|\nabla^E S_{n+1}\|_2^2
=&\frac{1}{\Vol(M)}\int_M \left|\sum_{i=1}^n S_1\wedge\cdots\wedge \nabla^E S_i\wedge\cdots \wedge S_n \right|^2\,d\mu_g
\leq C\delta.
\end{split}
\end{equation}
By Lemma \ref{p42a}, for all $x\in G=G(f_1,\ldots,f_n)$, we have
\begin{equation*}
\left||S_{n+1}(x)|^2-1\right|
\leq C\delta^\frac{1}{96n}.
\end{equation*}
Thus, we get
\begin{equation}\label{422e}
\|S_{n+1}\|_2^2\geq \frac{1}{\Vol(M)}\int_G (1-C\delta^\frac{1}{96n})\,d\mu_g
\geq 1-C\delta^\frac{1}{96n}\geq \frac{1}{2}.
\end{equation}
By (\ref{422d}) and (\ref{422e}), we get
\begin{equation}\label{422ea}
\|\nabla^E S_{n+1}\|_2^2\leq C\epsilon^\frac{1}{2}\|S_{n+1}\|_2^2.
\end{equation}
Since $S_{n+1}$ is orthogonal to $S_i$ ($i=1,\ldots,n$), we get
$\lambda_{n+1}(\bar{\Delta}^E)\leq C\epsilon^\frac{1}{2}$ by (\ref{422ea}) and Lemma \ref{p41b} for $T_{(n,\delta)}(V)$.
\end{proof}

\subsection{Completion of the proofs}
In this subsection, we show that if $\lambda_n(\bar{\Delta}^E)$ is close to $0$, then the unorientable case cannot occur under the assumption $\Ric\geq -Kg$ and $\diam(M)\leq D$ by a mapping degree argument, and complete the proof of Main Theorem 1.

We first calculate the degree of the approximation map defined in Lemma \ref{p32c}. The following lemma corresponds to the discussion of the page 397--398 in \cite{Au}.
\begin{Lem}\label{p43a}
Given an integer $n\geq 2$ and positive real numbers $K>0$ and $D>0$, there exist positive constants $0<\delta_{\Dr}(n,K,D)\leq \delta_{\Dq}$ and $C(n,K,D)>0$ such that the following properties hold.
Take a positive real number $0<\delta\leq \delta_{\Dr}$.
Let $(M,g)$ be an oriented $n$-dimensional closed Riemannian manifold with $\Ric\geq -Kg$ and $\diam(M)\leq D$.
Suppose that an $(n+1)$-subspace $V\subset C^\infty(M)$ satisfies the $\delta$-pinching condition.
Take elements $f_s\in T_{(n,\delta)}(V)$ $(s=1,\ldots,n+1)$ such that
\begin{equation*}
\|f_s\|_2^2=\frac{1}{n+1},\quad \int_M f_s f_t \,d\mu_g=0\quad \text{for $s,t=1,\ldots,n+1$ with $s\neq t$}.
\end{equation*}
Then, $d_{GH}(M,S^n)\leq C\delta^\frac{1}{250n^2}$, $M$ is diffeomorphic to $S^n$ and the degree of the map $\Psi\colon M\to S^n$ defined by $\Psi=(f_1,\ldots,f_{n+1})/|(f_1,\ldots,f_{n+1})|$ is equal to either $1$ or $-1$.
\end{Lem}
\begin{proof}
We define $\widetilde{\Psi}=(f_1,\ldots,f_{n+1})$.
Put $S_i:=S_{f_i}$ for each $i$.
For each $x\in M$, the differential of $\Psi$ defines the map $d_x \Psi \colon T_x M\to T_{\Psi(x)}S^n$.
Since $T_{\Psi(x)}S^n\subset \mathbb{R}^{n+1}$, we can regard $d_x \Psi$ as a linear map
$T_x M\to \mathbb{R}^{n+1}$.
Then, take a transpose ${}^t d_x \Psi\colon \mathbb{R}^{n+1}\to T_x M$.
The restriction of this map to $T_{\Psi(x)} S^n$ coincides with the transpose of $d_x \Psi$ regarded as a map $T_x M\to T_{\Psi(x)}S^n$.
The determinant of the map ${}^t d_x \Psi\colon T_{\Psi(x)}S^n\to T_x M$ satisfies
$$
(\det {}^t d_x \Psi) \omega_g={}^t d_x \Psi(\bar{e}_1)\wedge\ldots\wedge {}^t d_x \Psi(\bar{e}_n),
$$
where $\omega_g$ denotes the volume form of $(M,g)$ and $\{\bar{e}_1,\ldots,\bar{e}_n\}$ is the orthonormal basis of $T_{\Psi(x)}S^n$ such that $\{\bar{e}_1,\ldots,\bar{e}_n,\Psi(x)\}$ has a positive orientation in $\mathbb{R}^{n+1}$.
The degree of the map $\Psi$ satisfies
$$
\deg \Psi=\frac{1}{\Vol(S^n)}\int_M \det {}^t d_x \Psi\,d\mu_g.
$$

Let us show
the map ${}^t d_x \Psi\colon \mathbb{R}^{n+1}\to T_x M$ satisfies
\begin{equation}\label{rv4}
{}^t d_x \Psi (e_i)=\frac{S_i(x)-\Psi_i(x)\sum_{j=1}^{n+1} \Psi_j(x)S_j(x)}{|\widetilde{\Psi}(x)|}
\end{equation}
for all $i=1,\ldots, n+1$, where $\{e_1,\ldots,e_{n+1}\}$ denotes the standard basis of $\mathbb{R}^{n+1}$ and $\Psi_i(x)=f_i(x)/|\widetilde{\Psi}(x)|$.
Note that the right hand side of (\ref{rv4}) defines an element of $T_x M$ by cancellation
although it is originally an element of the fiber $E_x=T_x M\oplus \mathbb{R}e$.
For all $X\in T_x M$ and $i$, we have
\begin{equation*}
\langle X, {}^t d_x \Psi (e_i) \rangle=
\langle d_x \Psi(X),e_i\rangle
=\left\langle\nabla\left(\frac{f_i}{(\sum_{j=1}^{n+1} f_j^2)^{1/2}}\right)(x),X\right\rangle.
\end{equation*}
Thus, we get
\begin{equation*}
\begin{split}
{}^t d_x \Psi (e_i)
=&\nabla\left(\frac{f_i}{(\sum_{j=1}^{n+1} f_j^2)^{1/2}}\right)(x)\\
=&\frac{\nabla f_i(x)}{|\widetilde{\Psi}(x)|}-\frac{f_i(x)\sum_{j=1}^{n+1}f_j(x)\nabla f_j(x)}{|\widetilde{\Psi}(x)|^3}\\
=&\frac{\nabla f_i(x)-\Psi_i(x)\sum_{j=1}^{n+1}\Psi_j(x)\nabla f_j(x)}{|\widetilde{\Psi}(x)|}.
\end{split}
\end{equation*}
Since $f_i(x)-\Psi_i(x)\sum_{j=1}^{n+1}\Psi_j(x) f_j(x)=0$ and $S_i=\nabla f_i+f_i e$ by the definition, we get (\ref{rv4}).

Define $L_x \colon \mathbb{R}^{n+1}\to E_x$ by $L_x(v):=\sum_{i=1}^{n+1} v_i S_i(x)$ for any $v=\sum_{i=1}^{n+1} v_i e_i\in \mathbb{R}^{n+1}$.
We have 
\begin{equation}\label{rv5}
L_x(v)=|\widetilde{\Psi}(x)| {}^t d_x \Psi (v)\in T_x M
\end{equation}
for all $v\in \mathbb{R}^{n+1}$ with $v\cdot \Psi(x)=0$ by (\ref{rv4}), and we have
\begin{equation}\label{rv6}
\begin{split}
L_x(\Psi(x))=&|\widetilde{\Psi}(x)|^{-1}\sum_{j=1}^{n+1}(\nabla f_j (x)+f_j(x) e) f_j(x)\\
=&|\widetilde{\Psi}(x)|^{-1}\sum_{j=1}^{n+1} f_j(x)\nabla f_j (x)+|\widetilde{\Psi}(x)|e.
\end{split}
\end{equation}
Let $\{\bar{e}_1(x),\cdots,\bar{e}_n(x),\Psi(x)\}$ be a positively oriented orthonormal basis of $\mathbb{R}^{n+1}$.
Define $\omega_E:=\omega_g\wedge e$, where $\omega_g$ denotes the volume form of $(M,g)$.
Then, the determinant of the map $L_x \colon \mathbb{R}^{n+1}\to E_x$ satisfies
\begin{align*}
(\det L_x )\omega_E=&L_x(e_1)\wedge\cdots\wedge L_x(e_{n+1})\\
=&L_x(\bar{e}_1)\wedge\cdots\wedge L_x(\bar{e}_n)\wedge L_x(\Psi(x))
\end{align*}
By the definition of $L$, we have
\begin{equation}\label{rv7}
(\det L_x )\omega_E=S_1(x)\wedge\cdots\wedge S_{n+1}(x).
\end{equation}
Since $L_x(\bar{e}_i(x))\in T_x M$ by (\ref{rv5}),
we have
\begin{equation}\label{rv8}
(\det L_x )\omega_E=|\widetilde{\Psi}(x)|L_x(\bar{e}_1)\wedge\cdots\wedge L_x(\bar{e}_n)\wedge e
\end{equation}
by (\ref{rv6}).
Moreover, we have
\begin{equation}\label{rv9}
(\det {}^t d_x\Psi(x))\omega_g=|\widetilde{\Psi}(x)|^{-n} L_x(\bar{e}_1)\wedge\cdots\wedge L_x(\bar{e}_n)
\end{equation}
by (\ref{rv5}).
By (\ref{rv8}) and (\ref{rv9}), we get
\begin{equation}\label{rv10}
\det L_x =|\widetilde{\Psi}(x)|^{n+1}\det {}^t d_x\Psi(x).
\end{equation}

Put $h:=\det L_x$.
Since $\omega_E$ is parallel and $|\omega_E|=1$, we have $|h|=|S_1\wedge\cdots\wedge S_{n+1}|$ and $|\nabla h|=|\sum_{i=1}^{n+1} S_1\wedge\cdots\wedge \nabla S_i\wedge \cdots\wedge S_{n+1}|$
by (\ref{rv7}).
For $x\in G=G(f_1,\ldots,f_{n+1})$ (see Lemma \ref{p42a}),
we have $||h|^2-1|\leq C\delta^\frac{1}{96n}$.
By Lemma \ref{p42a} (i), we have $\|h\|_{\infty}\leq C$.
Thus, we get $\||h|^2-1\|_1\leq C\delta^\frac{1}{96n}$ by Lemma \ref{p42a} (ii).
This gives 
\begin{equation}\label{hl2b}
|\|h\|_2^2-1|\leq C\delta^\frac{1}{96n}.
\end{equation}
Since $T_{(n,\delta)}(V)$ satisfies the $C\sqrt{\delta}$-pinching condition, by Lemma \ref{p41b} and Lemma \ref{p42a},
we have $\|\nabla h\|_2^2\leq C\delta$.
Thus, we get
\begin{equation*}
0\leq \|h\|_2^2 -\left(\frac{1}{\Vol(M)}\int_M h \,d\mu_g\right)^2\leq C\|\nabla h\|_2^2 \leq C\delta
\end{equation*}
by the Li-Yau estimate \cite[p.116]{SY}, and so
\begin{equation*}
\left|\left|\frac{1}{\Vol(M)}\int_M h\,d\mu_g \right|-1\right|\leq C \delta^\frac{1}{96n}.
\end{equation*}
by (\ref{hl2b}).
Therefore, by (\ref{rv10}) and Lemma \ref{32b},
\begin{equation}\label{deg}
\begin{split}
\left||\deg \Psi |\frac{\Vol(S^n)}{\Vol(M)}-1\right|
=&\left|\left|\frac{1}{\Vol(M)}\int_M \det {}^t d_x\Psi \,d\mu_g\right|-1\right|\\
\leq&\left|\left|\frac{1}{\Vol(M)}\int_M h\,d\mu_g \right|-1\right|+\left|1-\|\widetilde{\Psi}\|_{\infty}^{-(n+1)}\right|\|h\|_{\infty}\\
\leq & C\delta^\frac{1}{96n}.
\end{split}
\end{equation}
In particular $\deg\Psi\neq 0$.
Thus, we get $\Psi$ is surjective.
Therefore, $\Psi$ is a $C\delta^\frac{1}{250n^2}$-Hausdorff approximation map by Lemma \ref{p32c}, and so $d_{GH}(M,S^n)\leq C\delta^\frac{1}{250n^2}$.
Taking $\delta_{\Dr}$ sufficiently small, we get that $M$ is diffeomorphic to $S^n$ by Theorem 1.5.
By the volume convergence theorem \cite[Theorem 0.1]{Co3}, we can assume $\Vol(M)\leq \frac{3}{2}\Vol(S^n)$.
Thus, we get $|\deg \Psi|=1$ by (\ref{deg}).
\end{proof}

Let us show that the unorientable case cannot occur under our pinching condition, and complete the proof of Main Theorem 1.
The proof of the following theorem corresponds to \cite[Proposition 19]{Au}.
\begin{Thm}\label{p43b}
Given an integer $n\geq 2$ and positive real numbers $K>0$ and $D>0$, there exist positive constants $\eta(n,K,D)>0$ and $C(n,K,D)>0$ such that the following property holds.
Take a positive real number $0<\epsilon\leq \eta$.
Let $(M,g)$ be an $n$-dimensional closed Riemannian manifold with $\Ric\geq -Kg$ and $\diam(M)\leq D$.
If $\lambda_n(\bar{\Delta}^E)\leq \epsilon$, then $d_{GH}(M,S^n)\leq C \epsilon^\frac{1}{1000n^2}$.
\end{Thm}
\begin{proof}
For the orientable case, we get the theorem by Proposition \ref{p41c}, Corollary \ref{p42c} and Lemma \ref{p43a}.

Let us show that $M$ is orientable.
Suppose that $M$ is not orientable. Take the orientable Riemannian covering $\pi \colon(\widetilde{M},\tilde{g})\to (M,g)$ with two sheets.
Then, $(\widetilde{M},\tilde{g})$ satisfies $\Ric_{\tilde{g}}\geq -K\tilde{g}$ and $\diam(\widetilde{M})\leq 2D$.
Take orthonormal eigensections $S_1,\ldots,S_n\in \Gamma(E_M)$ corresponding to $\lambda_i(\bar{\Delta}^E,M)$: $\bar{\Delta}^E S_i=\lambda_i(\bar{\Delta}^E,M)S_i$.
Write $S_i=\alpha_i+f_i e$.
Define $\widetilde{S}_i\in \Gamma(E_{\widetilde{M}})$ as $\widetilde{S}_i:=\pi^\ast S_i=\pi^\ast\alpha_i +(f_i\circ\pi) e$.
Then, $\bar{\Delta}^{E}\widetilde{S}_i=\lambda_i(\bar{\Delta}^E,M)\widetilde{S}_i$, and so $\lambda_n(\bar{\Delta}^{E},\widetilde{M})\leq \epsilon$. 
Thus, we obtain $\lambda_{n+1}(\bar{\Delta}^{E},\widetilde{M})\leq C \epsilon^\frac{1}{2}$ by Corollary \ref{p42c}. 
Then, there exists $\tilde{\lambda}\in \mathbb{R}$ with $0\leq \tilde{\lambda}\leq C\epsilon^\frac{1}{2}$ and an eigensection $\widetilde{S}_{n+1}\in \Gamma(E_{\widetilde{M}})$ such that $\bar{\Delta}^{E}\widetilde{S}_{n+1}=\tilde{\lambda}\widetilde{S}_{n+1}$ and $\int_{\widetilde{M}}\langle\widetilde{S}_{n+1},\widetilde{S}_{i}\rangle_{E}\,d\mu_{\tilde{g}}=0$ for all $i=1,\ldots,n$.
Define a subspace $\mathcal{E}\in C^\infty(M)$ by
$$
\mathcal{E}=\Span\{\widetilde{S}_1,\ldots,\widetilde{S}_n,\widetilde{S}_{n+1}\}.
$$
By Claim \ref{c41b}, we have that $\widetilde{P}(\mathcal{E})\subset C^\infty(\widetilde{M})$ satisfies the $C\epsilon^\frac{1}{4}$-pinching condition.
Put $\delta:=C\epsilon^\frac{1}{4}$.

We use the following claim (see \cite[Lemma 17]{Au}).
\begin{Clm}\label{cAu}
Suppose that a map $\Phi\colon S^n\to S^n$ has degree $1$ or $-1$.
Let
$$
\alpha_i\colon \mathbb{Z}/(2\mathbb{Z}) \times S^n\to S^n
$$
be actions $($$i=1,2$$)$ such that
$\Phi(\alpha_1,(a,x))=\alpha_2(a,\Phi(x))$ holds for all $(a,x)\in\mathbb{Z}/(2\mathbb{Z}) \times S^n$.
If the action $\alpha_1$ is free, then $\alpha_2$ is also free.
\end{Clm}
Consider the covering transformation $\Pi\colon\widetilde{M}\to\widetilde{M}$, which satisfies $\pi\circ \Pi=\pi$ and $\Pi^2=\Id$.
We use $\Pi$ for $\alpha_1$ of Claim \ref{cAu}.
Let us find the action $\alpha_2$ and its fixed point to get the contradiction.

The covering transformation $\Pi$ defines an isomorphism $\Pi^\ast\colon E_{\Pi(x)}\to E_x$, $\Pi^\ast(\alpha+f e)=\Pi^\ast \alpha +(f\circ \Pi) e$.
We have 
\begin{equation}\label{rv11}
\Pi^\ast \widetilde{S}_i=\widetilde{S}_i
\end{equation}
for any $i=1,\ldots,n$.
Since $\lambda_{n+2}(\bar{\Delta}^{E},\widetilde{M})\geq \epsilon_2$ by Corollary \ref{p42b}, we obtain
$\Pi^\ast \widetilde{S}_{n+1}\in \mathcal{E}$, and $\mathcal{E}$ is invariant under the action of $\Pi$.
Since $\Pi^\ast \widetilde{S}_{n+1}$ is orthogonal to $\Pi^\ast \widetilde{S}_i=\widetilde{S}_i$ in $L^2$ sense for all $i=1,\ldots,n$, we get 
\begin{equation}\label{rv12}
\Pi^\ast \widetilde{S}_{n+1}=\pm\widetilde{S}_{n+1}.
\end{equation}
Take $f_1,\ldots,f_{n+1}\in T_{(n,\delta)}\widetilde{P}(\mathcal{E})\subset C^\infty(\widetilde{M})$ such that
\begin{equation*}
\begin{split}
\|f_s\|_2^2=\frac{1}{n+1},\quad \int_M f_s f_t \,d\mu_g=&0\quad \text{for $s,t=1,\ldots,n+1$ with $s\neq t$}, \text{ and}\\
f_{n+1}=&\frac{T_{(n,\delta)}\widetilde{P}(\widetilde{S}_{n+1})}{\sqrt{n+1}\|T_{(n,\delta)}\widetilde{P}(\widetilde{S}_{n+1})\|_2}
\end{split}
\end{equation*}
hold.
By the definition of $T_{(n,\delta)}$ and $\widetilde{P}$, we have
\begin{equation}\label{rv13}
T_{(n,\delta)}\widetilde{P}(\Pi^\ast S)=T_{(n,\delta)}\widetilde{P}(S)\circ \Pi.
\end{equation}
for all $S\in \Gamma(E_{\widetilde{M}})$.
Thus, we have
\begin{equation}\label{rv14}
f_{n+1}\circ \Pi=\pm f_{n+1}
\end{equation}
by (\ref{rv12}).
Since the subspace $T_{(n,\delta)}\widetilde{P}(\mathcal{E})\subset C^\infty(\widetilde{M})$ is invariant under the action of $\Pi$ by (\ref{rv13}),
there exists $T\in O(n+1)$ with 
\begin{equation}\label{rv15}
f_i\circ\Pi=\sum_{j=1}^{n+1} T_{ij} f_j
\end{equation}
for all $i=1,\ldots,n+1$ and $T^2=\Id$.
Take  $a_1,\ldots, a_n\in \mathbb{R}$ and
$\tilde{f}_1,\ldots,\tilde{f}_n\in T_{(n,\delta)}\widetilde{P}(\Span\{\widetilde{S}_1,\ldots,\widetilde{S}_n\})$  with
\begin{equation}\label{rv16}
f_i=\tilde{f}_i+a_i f_{n+1}.
\end{equation}
Then, we have 
\begin{equation}\label{rv17}
\tilde{f}_i\circ \Pi=\tilde{f}_i
\end{equation}
for all $i=1,\ldots,n$ by (\ref{rv11}) and (\ref{rv13}).
Define $\widetilde{\Psi}\colon \widetilde{M}\to \mathbb{R}^{n+1}$ and $\Psi\colon \widetilde{M}\to S^n$ by $\widetilde{\Psi}:=(f_1,\ldots, f_{n+1})$ and
 $\Psi:=\widetilde{\Psi}/|\widetilde{\Psi}|$.
Then, $T \widetilde{\Psi}=\widetilde{\Psi}\circ\Pi$  by (\ref{rv15}).
Since $T$ is orthogonal, we have $|\widetilde{\Psi}\circ\Pi|=|T \widetilde{\Psi}|=| \widetilde{\Psi}|$.
Thus, we get 
\begin{equation}\label{rv18}
T \Psi=\Psi\circ\Pi.
\end{equation}
Since $\int_M f_{n+1} \,d\mu_g=0$, we can take a point $x\in \widetilde{M}$ such that $f_{n+1}(x)=0$ holds.
Then, we get $f_i(x)=f_i(\Pi(x))$ for all $i=1,\ldots,n+1$ by (\ref{rv14}), (\ref{rv16}) and (\ref{rv17}).
Thus, we get
\begin{equation}\label{rv19}
T \Psi(x)=\Psi(\Pi(x))=\Psi(x)
\end{equation}
by (\ref{rv18}).
By Lemma \ref{p43a}, $|\deg \Psi|=1$, and $\widetilde{M}$ is diffeomorphic to $S^n$.
Thus, by Claim \ref{cAu} and (\ref{rv18}), the action of $T$ on $S^n$ is free.
This contradicts to (\ref{rv19}).
Therefore, $M$ is orientable, and so we get the theorem.
\end{proof}

Theorem \ref{p43b} immediately implies Main Theorem 1 and Main Theorem 1'.

Finally, we write down the statement of unnormalized version of Main Theorem 1'.
By Corollary \ref{p2b} and Main Theorem 1', we get the following theorem.
\begin{Thm}
Given an integer $n\geq 2$ and positive real numbers $\epsilon>0$, $K>0$, $D>0$ and $\mu>0$, there exists $\delta(n,K,D,\mu,\epsilon)>0$ such that the following property holds.
Let $(M,g)$ be an $n$-dimensional closed Riemannian manifold with $\Ric \geq-K g$ and $\diam(M)\leq D$ and take $c\in\mathbb{R}$ with $c\leq\mu$.
 If there exists an $n$-dimensional subspace $V$ of $C^\infty(M)$ such that  $\|\nabla^2 f+c fg\|_2\leq \delta\|\Delta f\|_2$ holds for all $f\in V$, then we have $ c>0$ and
$d_{GH}(M,S^n(r))\leq \epsilon$, where we put $r=c^{-\frac{1}{2}}$.
\end{Thm}
\subsection{Constant scalar curvature metrics}
In this subsection, we show Main Theorem 2.
We first give an easy calculation.
\begin{Lem}\label{p6a}
Let $(M,g)$ be an $n$-dimensional Riemannian manifold with constant scalar curvature.
For any smooth function $f\in C^\infty(M)$, we have
\begin{equation*}
\nabla^\ast(\Ric(\nabla f))= -g\left(\Ric,\nabla^2 f+\frac{\Delta f}{n}g\right)+\frac{\Scal_g}{n} \Delta f,
\end{equation*}
where $\Scal_g$ denotes the scalar curvature of $(M,g)$.
\end{Lem}
\begin{proof}
By the second Bianchi identity \cite[p.88]{Pe3}, we have
\begin{equation*}
\begin{split}
&\nabla^\ast(\Ric(\nabla f))\\
=&(\nabla^\ast \Ric)(\nabla f)-g\left(\Ric,\nabla^2 f\right)\\
=&-\frac{1}{2}  g(\nabla \Scal_g, \nabla f) -g\left(\Ric,\nabla^2 f+\frac{\Delta f}{n} g\right) +\frac{\Scal_g}{n}\Delta f\\
=&-g\left(\Ric,\nabla^2 f+\frac{\Delta f}{n} g\right) +\frac{\Scal_g}{n}\Delta f
\end{split}
\end{equation*}
for all $f\in C^\infty (M)$.
\end{proof}

By the Bochner formula, we get the following lemma immediately.
\begin{Lem}\label{p6b}
Let $(M,g)$ be an $n$-dimensional closed Riemannian manifold.
For any smooth function $f\in C^\infty(M)$, we have
\begin{equation*}
\|\nabla^2 f +\frac{\Delta f}{n} g\|^2_2
=\frac{n-1}{n}\|\Delta f\|^2_2 -\frac{1}{\Vol(M)}\int_M \Ric(\nabla f,\nabla f)\,d\mu_g.
\end{equation*}
\end{Lem} 

We show that the pinching condition ``$\|\nabla^2 f+\frac{\Delta f}{n}g\|_2$ is small'' implies more strong pinching condition ``$\|\nabla^2 f+\frac{\Scal_g}{n(n-1)} f g\|_2$ is small'' for the Riemannian manifolds with constant scalar curvature.
The following theorem and Main Theorem 1' immediately imply Main Theorem 2.
\begin{Thm}\label{p6c}
Given an integer $n\geq 2$ and positive real numbers $K>0$ and $D>0$, there exist positive constants $\eta(n,K,D)>0$ and $C(n,K,D)>0$ such that the following property holds.
Take a positive real number $0<\delta\leq \eta$.
Let $(M,g)$ be an $n$-dimensional closed Riemannian manifold with constant scalar curvature, $-Kg\leq \Ric\leq Kg$ and $\diam(M)\leq D$.
For all smooth function $f\in C^\infty(M)$ such that $\int_M f\,d\mu_g=0$ and
\begin{equation*}
\|\nabla^2 f+\frac{\Delta f}{n} g\|_2 \leq \delta \|\Delta f\|_2,
\end{equation*}
we have
\begin{equation*}
\|\nabla^2 f+\frac{\Scal_g}{n(n-1)} f g\|_2 \leq C \delta^\frac{1}{2} \|\Delta f\|_2.
\end{equation*}
\end{Thm}
\begin{proof}
Note that we have $|\Ric|\leq n^\frac{1}{2}K$ and the Li-Yau estimate $\lambda_1\geq C(n,K,D)>0$ \cite[p.116]{SY}.

By Lemma \ref{p6b}, we have
\begin{equation}\label{6a}
\frac{1}{\Vol(M)}\int_M \Ric(\nabla f,\nabla f)\,d\mu_g\geq \left(\frac{n-1}{n}-\delta^2\right)\|\Delta f\|^2_2.
\end{equation}
Take a smooth function $h\in C^\infty(M)$ such that $\int_M h\,d\mu_g=0$ and $\Delta h=\Delta f-\frac{\Scal_g}{n-1}f$.
By Lemma \ref{p6b} and (\ref{6a}), for all $t\in \mathbb{R}$, we have
\begin{equation*}
\begin{split}
&\frac{n-1}{n}\left(\|\Delta f\|_2^2+\frac{2t}{\Vol(M)}\int_M \Delta f\Delta h\,d\mu_g +t^2\|\Delta h\|_2^2\right)\\
=&\frac{n-1}{n}\|\Delta( f+th)\|_2^2\\
\geq&\frac{1}{\Vol(M)}\int_M\Ric(\nabla f,\nabla f)\,d\mu_g\\
\qquad \qquad &+\frac{2t}{\Vol(M)}\int_M\Ric(\nabla f,\nabla h)\,d\mu_g+\frac{t^2}{\Vol(M)}\int_M\Ric(\nabla h,\nabla h)\,d\mu_g\\
\geq&\left(\frac{n-1}{n}-\delta^2\right)\|\Delta f\|^2_2 \\
\qquad \qquad &+\frac{2t}{\Vol(M)}\int_M\Ric(\nabla f,\nabla h)\,d\mu_g+\frac{t^2}{\Vol(M)}\int_M\Ric(\nabla h,\nabla h)\,d\mu_g.
\end{split}
\end{equation*}
Putting $t=-\delta$, we get
\begin{equation}\label{6b}
\begin{split}
&\frac{n-1}{n}\frac{1}{\Vol(M)}\int_M \Delta f\Delta h\,d\mu_g-\frac{1}{\Vol(M)}\int_M\Ric(\nabla f,\nabla h)\,d\mu_g\\
\leq& \frac{\delta}{2} \|\Delta f\|_2^2+\frac{\delta}{2}\frac{n-1}{n}\|\Delta h\|_2^2-\frac{\delta}{2}\frac{1}{\Vol(M)}\int_M\Ric(\nabla h,\nabla h)\,d\mu_g\\
\leq& C\delta\|\Delta f\|_2^2+C\delta\|\Delta h\|_2^2.
\end{split}
\end{equation}
By Lemma \ref{p6a}, we have
\begin{equation}\label{6c}
\begin{split}
&\frac{1}{\Vol(M)}\int_M\Ric(\nabla f,\nabla h)\,d\mu_g\\
=&\frac{1}{\Vol(M)}\int_M -g\left(\Ric,\nabla^2 f+\frac{\Delta f}{n}g\right)h+\frac{\Scal_g}{n} h\Delta f\,d\mu_g\\
\leq&\frac{1}{\Vol(M)}\int_M \frac{\Scal_g}{n} h\Delta f \,d\mu_g+n^\frac{1}{2}K\|\nabla^2 f+\frac{\Delta f}{n} g\|_2\|h\|_2\\
\leq&\frac{1}{\Vol(M)}\int_M \frac{\Scal_g}{n} h\Delta f\,d\mu_g+C\delta\|\Delta f\|_2\|\Delta h\|_2\\
\leq&\frac{1}{\Vol(M)}\int_M \frac{\Scal_g}{n} h\Delta f\,d\mu_g+C\delta\|\Delta f\|_2^2+C\delta\|\Delta h\|_2^2 .
\end{split}
\end{equation}
By (\ref{6b}), (\ref{6c}) and
\begin{equation*}
\|\Delta h\|_2^2=
\frac{1}{\Vol(M)}\int_M \Delta f\Delta h\,d\mu_g-\frac{1}{\Vol(M)}\int_M \frac{\Scal_g}{n-1} h\Delta f \,d\mu_g,
\end{equation*}
we get
\begin{equation*}
\left(\frac{n-1}{n}-C \delta\right)\|\Delta h\|_2^2
\leq C\delta \|\Delta f\|_2^2.
\end{equation*}
Thus, we get
\begin{equation*}
\begin{split}\|\nabla^2 f+\frac{\Scal_g}{n(n-1)} f g\|_2 
\leq &\|\nabla^2 f+\frac{\Delta f}{n}g\|_2 + \frac{1}{\sqrt{n}}\|\Delta h\|_2
\leq C \delta^\frac{1}{2}\|\Delta f\|_2.
\end{split}
\end{equation*}
\end{proof}
\begin{Rem}
If in addition $C\delta^\frac{1}{2}\leq \frac{1}{2\sqrt{n}}$ holds in Theorem \ref{p6c}, we have $\Scal_g\geq n(n-1)C_{\Ca}>0$ by Lemma \ref{p2a}.
Thus, we normalize the metric so that $\Scal_g=n(n-1)$.
Theorem \ref{p6c} shows that Main Theorem 1' implies Main Theorem 2.
Note that if $\Scal_g=n(n-1)$ and $\Ric\geq -Kg$, then $\Ric\leq (n+K)(n-1)$.
\end{Rem}

By Lemma \ref{p6b}, Theorem \ref{p6c} and Main Theorem 1, we get Main Theorem 2.
Finally, we write down the statement of unnormalized version of Main Theorem 2.
\begin{Thm}\label{unno}
Given an integer $n\geq 2$ and positive real numbers $\epsilon>0$, $K>0$ and $D>0$, there exists $\delta(n,K,D,\epsilon)>0$ such that if $(M,g)$ is an $n$-dimensional closed Riemannian manifold with constant scalar curvature, $-K g\leq \Ric \leq Kg$, $\diam(M)\leq D$ and $\Omega_{n}(g)\geq \frac{n-1}{n}-\delta^2$, then we have $\Scal_g>0$ and
$d_{GH}(M,S^n(r))\leq \epsilon$, where we put $r=\left(n(n-1)/\Scal_g\right)^\frac{1}{2}$.
\end{Thm}

\section{Almost umbilical manifolds}
In this section, we consider the relationship between the pinching condition on $\lambda_k(\bar{\Delta}^E)$ and almost umbilical manifolds, and show Main Theorem 3, which asserts that the almost umbilical manifold is close to the standard sphere under some geometrical assumption.
\subsection{Estimates by the extrinsic constants}
In this subsection, we investigate the relationship between the Riemannian invariants $\lambda_k(\bar{\Delta}^E)$ or $\Omega_k(g)$ and some extrinsic constants defined using the second fundamental form $B$ or the mean curvature $H$, e.g., $\|H\|_2$.
In particular Proposition \ref{um4}, which asserts that $\lambda_{n+1}(\bar{\Delta}^E)$ is small for $n$-dimensional almost umbilical manifolds, plays an important role to prove Main Theorem 3.
In this subsection we do not assume neither $\Ric\geq -Kg$ nor $\diam(M)\leq D$.

\begin{notation}
Let \((M,g)\) be an $n$-dimensional Riemannian manifold and $\iota \colon (M,g)\to \mathbb{R}^{n+k}$ an isometric immersion.
Let $\nabla$ and $\nabla^0$ be the Levi-Civita connection  on $(M,g)$ and $\mathbb{R}^{n+1}$ respectively.
We define the second fundamental form $B\colon \Gamma(TM)\times \Gamma(TM)\to \Gamma(TM^\perp)$ to be
\begin{equation*}
\nabla ^0 _X Y= \nabla_X Y +B(X,Y)
\end{equation*}
holds for any vector fields $X,Y\in\Gamma(T M)$.
Note that the second fundamental form defines a section of $T^\ast M\otimes T^\ast M\otimes TM^\perp$, i.e., $B\in\Gamma(T^\ast M\otimes T^\ast M\otimes TM^\perp)$.
Define $H:=-\frac{1}{n}\tr_g B\in\Gamma(TM^\perp)$.
Let $\langle\cdot,\cdot\rangle_0$ denotes the standard inner product on $\mathbb{R}^{n+k}$.
The immersion $\iota \colon (M,g)\to \mathbb{R}^{n+k}$ is called totally umbilical if $B+g\otimes H=0$ holds.
\end{notation}

The following lemma is standard.
\begin{Lem}\label{um1}
Let \((M,g)\) be an $n$-dimensional Riemannian manifold and $\iota \colon (M,g)\to \mathbb{R}^{n+k}$ an isometric immersion.
Take a vector $v\in \mathbb{R}^{n+k}$ and define a function
$f_v\colon M\to \mathbb{R}$ by
\begin{equation*}
f_v(x):=\langle \iota(x), v \rangle_0.
\end{equation*}
Then, we have the following equations:
\begin{itemize}
\item[(i)] $\nabla^2 f_v = \langle B(\cdot,\cdot), v\rangle_0$.
\item[(ii)] $\Delta f_v=n\langle H, v\rangle_0$.
\end{itemize}
\end{Lem}

We first consider the relationship between $\Omega_1(g)$ and the extrinsic properties.

\begin{Prop}\label{um2}
Let \((M,g)\) be an $n$-dimensional closed Riemannian manifold and $\iota \colon (M,g)\to \mathbb{R}^{n+k}$ an isometric immersion.
Then, we have
\begin{equation*}
n^2\left(\frac{n-1}{n}-\Omega_1(g)\right)\|H\|_2^2\leq \|B+g\otimes H\|_2^2.
\end{equation*}

\end{Prop}
\begin{proof}
By the definition of $\Omega_1(g)$, Lemma \ref{p6b} and Lemma \ref{um1}, we have
\begin{equation*}
\begin{split}
n^2\left(\frac{n-1}{n}-\Omega_1(g)\right)\|\langle H, v \rangle_0\|_2^2&=
\left(\frac{n-1}{n}-\Omega_1(g)\right)\|\Delta f_v\|_2^2\\
&\leq\|\nabla^2 f_v+\frac{\Delta f_v}{n} g\|_2= \|\langle B+g\otimes H,v\rangle_0\|_2^2,
\end{split}
\end{equation*}
for all $v\in \mathbb{R}^{n+k}$.
Let $\{v_1,v_2,\ldots,v_{n+k}\}$ be the orthonormal basis on $\mathbb{R}^{n+k}$.
Since, we have
\begin{equation*}
\begin{split}
\sum_{i=1}^{n+k} \|\langle H, v_i \rangle_0\|_2^2&=\|H\|_2^2,\\
\sum_{i=1}^{n+k} \|\langle B + g\otimes H,v_i\rangle_0\|_2^2&=\|B + g\otimes H\|_2^2,
\end{split}
\end{equation*}
we get the proposition.
\end{proof}

We get the following corollary immediately.
\begin{Cor}\label{um3}
Let \((M,g)\) be an $n$-dimensional closed Riemannian manifold and $\iota \colon (M,g)\to \mathbb{R}^{n+k}$ a isometric immersion.
If $\|B+g\otimes H\|_2^2\leq \delta \|H\|_2^2$ for some $\delta>0$, then we have
\begin{equation*}
\Omega_1(g)\geq \frac{n-1}{n}-\frac{\delta}{n^2}.
\end{equation*}
\end{Cor}

We next consider the case of immersed hypersurfaces, i.e., $k=1$.
\begin{notation}
Let \((M,g)\) be an $n$-dimensional oriented closed Riemannian manifold and $\iota \colon (M,g)\to \mathbb{R}^{n+1}$ an isometric immersion.
Fix the unit normal vector field $\nu\in\Gamma(TM^\perp)$.
Define a section $A\in\Gamma(T^\ast M\otimes TM)$ to be
\begin{equation*}
g(A(X),Y)=-\langle B(X,Y),\nu_x \rangle_0
\end{equation*}
for any $x\in M$ and $X,Y\in T_x M$.
Define $h:=\frac{1}{n}\tr A\in C^\infty(M)$.
The immersion $\iota \colon (M,g)\to \mathbb{R}^{n+1}$ is totally umbilical if and only if
$A- h\Id =0$ holds.
\end{notation}

Let us show that $\lambda_{n+1}(\bar{\Delta}^E)$ is small for the almost umbilical manifold in the sense that the quantity $\|A-\Id \|_2$ is small.
\begin{Prop}\label{um4}
Let \((M,g)\) be an $n$-dimensional oriented closed Riemannian manifold and $\iota \colon (M,g)\to \mathbb{R}^{n+1}$ an isometric immersion.
Then, we have
\begin{equation*}
\lambda_{n+1}(\bar{\Delta}^E)\leq 2\|A- \Id \|_2^2.
\end{equation*}
\end{Prop}
\begin{proof}
For each vector $v\in \mathbb{R}^{n+1}$, we define a section $S_v\in \Gamma(E_M)$ by
\begin{equation*}
S_v:= \nabla f_v + \langle \nu, v \rangle_0 e.
\end{equation*}
Since we have $\nabla \langle \nu, v \rangle_0=\langle A(\cdot), v \rangle_0$ and $\nabla^2 f_v=-\langle \nu, v \rangle_0 A$,
we have
\begin{equation*}
\nabla^E S_v=-\langle \nu, v \rangle_0 (A-\Id)+\langle A(\cdot)-\Id(\cdot), v \rangle_0\otimes e.
\end{equation*}
Thus, we get
$\|\nabla^E S_v\|^2_2\leq 2|v|^2\|A-\Id\|^2_2$.
Since $|S_v|^2=|\nabla f_v|^2+\langle \nu, v \rangle_0^2=|\nabla^0 f_v|^2=|v|^2$,
we have $\|S_v\|_2^2=|v|^2$.
Therefore, we get
\begin{equation*}
\|\nabla^E S_v\|_2^2\leq 2\|A-\Id\|^2_2\|S_v\|_2^2.
\end{equation*}

The map $\mathbb{R}^{n+1}\to \Gamma(E_M),\, v\mapsto S_v$ is linear, and the image of the map is $(n+1)$-dimensional.
Thus, by the Rayleigh principle
$$\lambda_{n+1}(\overline{\Delta}^E)=\inf\left\{\sup_{S\in W\setminus\{0\}} \frac{\|\nabla^E S\|_2^2}{\|S\|_2^2}: \text{$W$ is an $(n+1)$-dimensional subspace of $\Gamma(E)$}\right\},$$
we get the proposition.
\end{proof}

\subsection{Application to the almost umbilical manifolds}
In this subsection, we show that the almost umbilical manifold is close to the standard sphere in several senses under some assumptions.

We first show the stability result under the condition that $\|A-\Id\|_2$ is small.
The proof of following theorem was inspired by \cite{RS}.
\begin{Thm}\label{um4.5}
Given an integer $n\geq 2$ and positive real numbers $\epsilon>0$ and $K>0$, there exists a positive constant $\delta(n,K,\epsilon)>0$ such that the following property holds.
Let $(M,g)$ be an $n$-dimensional oriented closed Riemannian manifold with $\Ric\geq -Kg$, and $\iota \colon (M,g)\to \mathbb{R}^{n+1}$ an isometric immersion.
If
\begin{equation*}
\|A- \Id \|_2\leq \delta,
\end{equation*}
then we have the following:
\begin{itemize}
\item[(i)] $d_{GH}(M,S^n)\leq \epsilon$ and $M$ is diffeomorphic to the $n$-dimensional standard sphere.
\item[(ii)] $d_{H}(\iota(M),S_M)\leq \epsilon$, where $S_M$ is defined by
\begin{equation*}
S_M:=\left\{a+\frac{1}{\Vol(M)}\int_M \iota(x) \,d\mu_g(x) \in\mathbb{R}^{n+1}\colon a\in \mathbb{R}^{n+1} \text{ with } |a|=1\right\}.
\end{equation*}
\item[(iii)] There exists a closed subset $E\subset M$ such that $\Vol(M\backslash E)\leq \epsilon \Vol(M)$ and $\iota$ is injective on $E$.
\end{itemize}
\end{Thm}
\begin{proof}
By Proposition \ref{um4}, we have $\lambda_{n+1}(\bar{\Delta}^E)\leq 2\delta^2$.

To apply Main Theorem 1, we estimate the diameter of $(M,g)$.
Since we have
\begin{equation*}
\|A-\Id\|_2^2=\|A-h\Id\|_2^2+n \|h-1\|_2^2,
\end{equation*}
we get
\begin{align}
\label{1umb} \|A-h\Id\|_2&\leq \delta,\\
\label{2umb} \|h-1\|_2&\leq \frac{1}{\sqrt{n}}\delta.
\end{align}
By the Gauss equation, we have
\begin{equation*}
\begin{split}
&\Ric(X,Y)-(n-1)h^2g(X,Y)\\
=& -g(A(X)-hX,A(Y)-hY)+(n-2)h g(A(X)-h X,Y)
\end{split}
\end{equation*}
for any $X,Y\in \Gamma(TM)$.
Thus, we get
\begin{equation}\label{3umb}
\Ric-(n-1)g\geq -(n-1)|h^2-1|-|A-h\Id|^2-(n-2)|h||A-h\Id|.
\end{equation}
Define a map $\rho\colon M\to \mathbb{R}_{\geq0}$ by
\begin{equation*}
\begin{split}
\rho(x):=&(\Ric-(n-1)g)_{-}(x)\\
:=&\inf \left\{r \in\mathbb{R}_{\geq 0}: \Ric(X,X)-(n-1)\geq - r \text{ for all } X\in T_x M \text{ with } |X|=1\right\}
\end{split}
\end{equation*}
for all $x\in M$.
We have $\rho\leq K+n-1$ by the assumption, and
\begin{equation*}
\rho\leq  (n-1)|h^2-1|+|A-h\Id|^2+(n-2)|h||A-h\Id|
\end{equation*}
by (\ref{3umb}).
Therefore, we get
\begin{equation}\label{4umb}
\begin{split}
\|\rho\|_n^n&\leq (K+n-1)^{n-1} \|\rho\|_1\\
&\leq C(n,K)\left( \|h-1\|_2\|h+1\|_2+\|A-h\Id\|_2^2+\|h\|_2\|A-h\Id\|_2\right)\\
&\leq C\delta
\end{split}
\end{equation}
by (\ref{1umb}) and (\ref{2umb}).

By the Aubry's theorem \cite[Theorem 1.2]{Au2}, taking $\delta$ small enough, we get
\begin{equation*}
\diam(M)\leq\left(\pi+C(n,K)\delta^\frac{1}{10}\right)\leq C(n,K).
\end{equation*}
Thus, taking $\delta$ sufficient small, we get that $M$ is diffeomorphic to $S^n$ and that $d_{GH}(M,S^n)\leq \tau(n,K|\delta)$ by Main Theorem 1 and Proposition \ref{um4}, where $\lim_{\delta\to 0}\tau(n,K|\delta)=0$. Thus, we get (i).

We next prove (ii).
Without loss of generality, we can assume that
\begin{equation*}
\int_M \iota(x) \,d\mu_g(x)=0.
\end{equation*}
By \cite[Proposition 1.5]{Au2}, (\ref{2umb}) and (\ref{4umb}), we have
\begin{equation}\label{rvumb}
\lambda_1(g)\geq n-C(n,K)\delta^\frac{1}{n}\geq n\|h\|_2-C(n,K)\delta^\frac{1}{n}.
\end{equation}
By \cite[Lemma 4.2]{AG0}, we have
\begin{equation*}
\Vol\left(M\setminus A_{C \delta^\frac{1}{8n}}\right)\leq C(n,K)\delta^\frac{1}{8n}\Vol(M),
\end{equation*}
where we put 
\begin{equation*}
A_\eta:= \{x\in M:\frac{1}{\|h\|_2}(1-\eta)\leq |\iota(x)|\leq \frac{1}{\|h\|_2}(1+\eta) \}
\end{equation*}
for $\eta>0$.
For all $x\in M$, there exists $y\in A_{C\delta^\frac{1}{8n}}$ such that $d(x,y)\leq C(n,K)\delta^\frac{1}{8n^2}$ by the Bishop-Gromov inequality (Theorem \ref{p2h}), where $d$ denotes the intrinsic metric on $M$.
Since we have $d(x,y)\geq d_{\mathbb{R}^{n+1}}(\iota(x),\iota(y))$ for all $x,y\in M$ and $|\|h\|_2-1|\leq C\delta$ by (\ref{2umb}), we get
\begin{equation}\label{5umb}
\left||\iota(x)|-1\right|\leq C\delta^\frac{1}{8n^2}
\end{equation}
for all $x\in M$.
In particular, we have $|\iota(x)|>0$ for all $x\in M$.
Define a map $\Psi\colon M\to S^n$ by
$\Psi:=\iota/|\iota|$. Let us calculate the degree of $\Psi$.
Define an orientation on $M$ to be
$(e_1,\ldots,e_n)$ ($e_i\in T_x M$) is positive if and only if $(\iota_\ast e_1,\ldots,\iota_\ast e_n,\nu(x))$ is positive in $\mathbb{R}^{n+1}$ for all $x\in M$.
Similarly, define an orientation on $S^n$ to be
$(e_1,\ldots,e_n)$ ($e_i\in\mathbb{R}^{n+1}$ is tangent to $S^n$ at $x\in S^n$) is positive if and only if $(e_1,\ldots,e_n,x)$ is positive in $\mathbb{R}^{n+1}$.
Take $x\in M$ and a positive orthonormal basis $(e_1,\ldots,e_n)$ on  $T_x M$.
Then, we have
\begin{equation*}
\begin{split}
\Psi_\ast e_1\wedge\ldots\wedge \Psi_\ast e_n\wedge \Psi(x)=&
\det d \Psi (x) \frac{\partial}{\partial x_1}\wedge\ldots\wedge\frac{\partial}{\partial x_{n+1}}.
\end{split}
\end{equation*}
Let $\iota(x)^T$ denotes the orthogonal projection of $\iota(x)$ on $\iota_\ast (T_x M)$.
Since we have
\begin{equation*}
\Psi_\ast e_i= \frac{\iota_\ast e_i}{|\iota(x)|}-\frac{\langle\iota_\ast e_i,\iota(x) 
\rangle_0}{|\iota(x)|^3}\iota(x)
\end{equation*}
and $|\iota(x)^T|\leq |\iota(x)|$, we get
\begin{equation*}
\begin{split}
&\left|\det d \Psi (x)-1 \right|\\
\leq& \left|\frac{1}{|\iota(x)|^n}\iota_\ast e_1\wedge\ldots\wedge \iota_\ast e_n\wedge \Psi(x)-\iota_\ast e_1\wedge\ldots\wedge \iota_\ast e_n\wedge \nu(x)\right| + C|\iota(x)^T|\\
\leq& |\Psi(x)-\nu(x)|+C\delta^\frac{1}{8n^2}+ C|\iota(x)^T|
\end{split}
\end{equation*}
by (\ref{5umb}).
Thus, by \cite[Lemma 4.1]{AG0}, (\ref{2umb}) and (\ref{5umb}), we get
\begin{equation}\label{6umb}
\begin{split}
\|\det d \Psi -1 \|_2\leq& \|\Psi-\nu\|_2+ C\delta^\frac{1}{8n^2}+C\|\iota^T\|_2\\
\leq &\|\Psi-\iota\|_2+\left\|\iota-\frac{h}{\|h\|_2^2}\nu\right\|_2+
\left\|\frac{h}{\|h\|_2^2}-1\right\|_2+C\delta^\frac{1}{8n^2}\\
\leq & C(n,K) \delta^\frac{1}{8n^2}.
\end{split}
\end{equation}
Since
\begin{equation*}
\deg \Psi 
=\frac{1}{\Vol(S^n)}\int_M \det d\Psi \,d\mu_g,
\end{equation*}
we get
\begin{equation*}
\begin{split}
|\deg \Psi -1|
\leq& \left|\deg \Psi -\frac{\Vol(M)}{\Vol(S^n)}\right|
+\left|\frac{\Vol(M)}{\Vol(S^n)}-1\right|\\
\leq & \frac{\Vol(M)}{\Vol(S^n)}\|\det d \Psi -1\|_2
+\left|\frac{\Vol(M)}{\Vol(S^n)}-1\right|.
\end{split}
\end{equation*}
By the volume convergence theorem \cite{Co3},
we have $|\Vol(M)-\Vol(S^n)|\leq \tau(n,K|\delta)$.
Thus, we get
$|\deg \Psi -1|\leq \tau(n,K|\delta)$, and so we get
$\deg \Psi=1$.
In particular, $\Psi$ is surjective.
Take arbitrary $y\in S^n$. Then, there exists $x\in M$ with $\Psi(x)=y$.
Since $|\Psi(x)-\iota(x)|\leq C\delta^\frac{1}{8n^2}$,
we get
$d_{\mathbb{R}^{n+1}}(y,\iota(x))\leq C\delta^\frac{1}{8n^2}$.
Thus, we get $d_{H}(\iota(M),S_M)\leq C(n,K)\delta^\frac{1}{8n^2}$.

Finally, we show (iii). 
By the area formula, we get
\begin{equation}\label{7umb}
\int_M|\det d\Psi|\,d\mu_g=\int_{S^n} H^0(\Psi^{-1}(y))\,d\mu_{S^n}(y),
\end{equation}
where $H^0$ denotes the counting measure.
Put
\begin{equation*}
\begin{split}
E_1:=&\{y\in S^n : H^0(\Psi^{-1}(y))=1\},\\
E_2:=&\{y\in S^n : H^0(\Psi^{-1}(y))\geq 2\}.
\end{split}
\end{equation*}
Then, we have
\begin{equation}\label{8umb}
\int_{S^n} H^0(\Psi^{-1}(y))\,d\mu_{S^n}
= \Vol(S^n)+\int_{E_2} \left(H^0(\Psi^{-1}(y))-1 \right)\,d\mu_{S^n}(y).
\end{equation}
Since we have
\begin{equation*}
\begin{split}
\int_M|\det d\Psi|\,d\mu_g
\leq \int_M|\det d\Psi-1|\,d\mu_g+\Vol(M)
\leq &\Vol(M)(1+\|\det d\Psi-1\|_2)\\
\leq& \Vol(S^n)+\tau(n,K|\delta)
\end{split}
\end{equation*}
by (\ref{6umb}), we get 
\begin{equation*}
\Vol(E_2)\leq \int_{E_2} \left(H^0(\Psi^{-1}(y))-1 \right)\,d\mu_{S^n}(y)\leq
\tau(n,K|\delta)
\end{equation*}
by (\ref{7umb}), (\ref{8umb}).
Thus, we get
\begin{equation*}
\int_{E_2} H^0(\Psi^{-1}(y))\,d\mu_{S^n}(y)=\int_{E_2} \left(H^0(\Psi^{-1}(y))-1 \right)\,d\mu_{S^n}(y)+\Vol(E_2)\leq \tau(n,K|\delta).
\end{equation*}
There exists a open subset $V\subset S^n$ such that $E_2\subset V$ and $\Vol(V\setminus E_2)\leq \delta$.
Then, we have
\begin{equation}\label{10umb}
\begin{split}
\int_{V} H^0(\Psi^{-1}(y))\,d\mu_{S^n}(y)=&\int_{E_2} H^0(\Psi^{-1}(y)) \,d\mu_{S^n}(y)+\Vol(V\setminus E_2)\\
\leq &\tau(n,K|\delta).
\end{split}
\end{equation}
By the area formula, we have
\begin{equation}\label{11umb}
\begin{split}
\int_{V} H^0(\Psi^{-1}(y))\,d\mu_{S^n}(y)
=&\int_{\Psi^{-1}(V)}|\det d\Psi|\,d\mu_g\\
\geq &\Vol(\Psi^{-1}(V))-\Vol(M)\|\det d\Psi-1\|_2.
\end{split}
\end{equation}
Thus, we get $\Vol(\Psi^{-1}(V))\leq \tau(n,K|\delta)$ by (\ref{6umb}), (\ref{10umb}) and (\ref{11umb}).
Since the map $\iota$ is injective on $E:= M\setminus \Psi^{-1}(V)$, we get (iii).
\end{proof}

The following lemma, which asserts that the mean curvature $h$ is almost constant in $L^2$ sense for the almost umbilical manifold in the sense that $\|A-h\Id\|_2$ is small under the assumption on Ricci curvature, was proved in \cite{CZ}.
We write down the proof here because it is short.
\begin{Lem}\label{um5}
Given an integer $n\geq 2$ and a positive real number $K>0$, there exists a positive constant $ C(n,K)>0$ such that the following property holds.
Let $(M,g)$ be an $n$-dimensional oriented closed Riemannian manifold with $\Ric\geq -K\lambda_1(g) g$, and $\iota \colon (M,g)\to \mathbb{R}^{n+1}$ an isometric immersion.
Put
\begin{equation*}
\overline{h}=\frac{1}{\Vol(M)}\int_M h\,d\mu_g.
\end{equation*}
Then, we have
\begin{itemize}
\item[(i)] $\|h-\overline{h} \|_2\leq C \|A-h\Id \|_2$.
\item[(ii)] $\|A-\overline{h}\Id\|_2\leq C \|A-h\Id \|_2$.
\end{itemize}
\end{Lem}
\begin{proof}
Take a function $f\in C^\infty(M)$ with $\int_M f\,d\mu_g=0$ and $\Delta f= h-\overline{h}$.
Since we have $\tr (\nabla_X A)=n X h$ for all $X\in \Gamma(TM)$, we get
\begin{equation*}
n \nabla h= -\nabla^\ast A
\end{equation*}
by the Codazzi equality, and so we get
\begin{equation*}
\nabla^\ast (A- hg)=-(n-1)\nabla h.
\end{equation*} 
Therefore, we have
\begin{equation}\label{12umb}
\begin{split}
\|h-\overline{h}\|_2^2
=&\frac{1}{\Vol(M)}\int_M (h-\overline{h})\Delta f\,d\mu_g\\
=&\frac{1}{\Vol(M)}\int_M \langle\nabla h,\nabla f\rangle\,d\mu_g\\
=&-\frac{1}{n-1} \frac{1}{\Vol(M)}\int_M \langle \nabla^\ast (A-h g), \nabla f\rangle\,d\mu_g\\
=&-\frac{1}{n-1} \frac{1}{\Vol(M)}\int_M \langle (A-h g), \nabla^2 f+\frac{\Delta f}{n} g\rangle\,d\mu_g\\
\leq&\frac{1}{n-1}\|A-h g\|_2\|\nabla^2 f+\frac{\Delta f}{n}g\|_2.
\end{split}
\end{equation}
By the Bochner formula, we have
\begin{equation}\label{13umb}
\begin{split}
\|\nabla^2 f+\frac{\Delta f}{n}g\|_2^2=&\frac{n-1}{n}\|\Delta f\|_2^2-\frac{1}{\Vol(M)}\int_M\Ric (\nabla f,\nabla f)\,d\mu_g\\
\leq &\frac{n-1}{n}\|\Delta f\|_2^2+K\lambda_1(g)\|\nabla f\|_2^2\\
\leq &\left(K+\frac{n-1}{n}\right)\|\Delta f\|_2^2=\left(K+\frac{n-1}{n}\right)\|h-\overline{h}\|_2^2.
\end{split}
\end{equation}
By (\ref{12umb}) and (\ref{13umb}), we get
\begin{equation*}
\|h-\overline{h}\|_2\leq\frac{1}{n-1}\sqrt{K+\frac{n-1}{n}}\|A-hg\|_2.
\end{equation*}
Therefore, we get
\begin{equation*}
\|A-\overline{h}\Id\|_2^2
= \|A-h Id\|_2^2+n\|h-\overline{h}\|_2^2
\leq \frac{n}{n-1}\left(\frac{K}{n-1}+1\right)\|A-h\Id\|_2^2.
\end{equation*}
\end{proof}

Let us normalize the mean curvature so that $\|h\|_2=1$ for simplicity.
\begin{Cor}\label{um6}
Given an integer $n\geq 2$ and positive real number $K>0$, there exists a positive constant $ C(n,K)>0$ such that the following property holds.
Let $(M,g)$ be an $n$-dimensional oriented closed Riemannian manifold with $\Ric\geq -K\lambda_1(g)g$, and $\iota \colon (M,g)\to \mathbb{R}^{n+1}$ an isometric immersion.
If $\|h\|_2=1$, we have
either
$\|A-\Id\|_2\leq C \|A-h g\|_2$ or $\|A+\Id\|_2\leq C \|A-h g\|_2$.
\end{Cor}
\begin{proof}
Since we have $|1-|\overline{h}||\leq C \|A-h \Id\|_2$ by Lemma \ref{um5} (i),
we get
\begin{equation*}
\|A-\Id \|_2\leq C \|A-h \Id\|_2.
\end{equation*}
if $\overline{h}\geq 0$, and
\begin{equation*}
\|A+\Id \|_2\leq C \|A-h \Id\|_2.
\end{equation*}
if $\overline{h}< 0$ by Lemma \ref{um5} (ii).
Thus, we get the corollary.
\end{proof}
\begin{notation}
Let $(M,g)$ be a Riemannian manifold.
In this subsection, let $(\widetilde{M},\tilde{g})$ be the orientable Riemannian covering of $(M,g)$ with two sheets if $M$ is not orientable, and
 $(\widetilde{M},\tilde{g})=(M,g)$ if $M$ is orientable.
\end{notation}

Let us show that the unorientable case cannot occur for the almost umbilical manifold in sense that $\|B+ g\otimes H\|_2$ is small under the assumption on the Ricci curvature.
\begin{Cor}\label{umb8}
Given an integer $n\geq 2$ and a positive real number and $K>0$, there exists $\delta(n,K)>0$ such that the following property holds.
Let $(M,g)$ be an $n$-dimensional closed Riemannian manifold and $\iota \colon (M,g)\to \mathbb{R}^{n+1}$ an isometric immersion.
If $\Ric\geq -K\lambda_1(\tilde{g})g$ and $\|B+ g\otimes H\|_2\leq \delta\|H\|_2$ holds,
then $M$ is orientable.
\end{Cor}
\begin{proof}
Suppose that $M$ is not orientable and let $\pi\colon \widetilde{M}\to M$ be a covering map.
Then, $\iota\circ \pi \colon \widetilde{M}\to \mathbb{R}^{n+1}$ is also isometric immersion.
Fix the unit normal vector field $\tilde{\nu}\in\Gamma(T\widetilde{M}^\perp)$.
Let $\widetilde{B}\in \Gamma( T^\ast \widetilde{M}\otimes T^\ast \widetilde{M}\otimes T \widetilde{M}^\perp)$ be the second fundamental form of $\widetilde{M}$ and $\widetilde{H}:=-\frac{1}{n}\tr\widetilde{B}$ the mean curvature.
Define a section $\widetilde{A}\in\Gamma(T^\ast \widetilde{M}\otimes T\widetilde{M})$ to be
\begin{equation*}
\tilde{g}(\widetilde{A}(X),Y)=-\langle \widetilde{B}(X,Y),\tilde{\nu}_x \rangle_0
\end{equation*}
for any $x\in \widetilde{M}$ and $X,Y\in T_x \widetilde{M}$.
Put $\tilde{h}:=\frac{1}{n}\tr \widetilde{A}\in C^\infty(\widetilde{M})$.
Consider the covering transformation $\Pi\colon\widetilde{M}\to\widetilde{M}$, which satisfies $\pi\circ \Pi=\pi$ and $\Pi^2=\Id$.
Then, we have $\tilde{\nu}(\Pi(x))=-\tilde{\nu}(x)\in \mathbb{R}^{n+1}$, and
$\widetilde{B}(X, Y)=\widetilde{B}(\Pi_\ast X,\Pi_\ast Y)\in \mathbb{R}^{n+1}$ for all $X,Y\in T_x \widetilde{M}$.
Thus, we have $\tilde{g}(\widetilde{A}(X),Y)=-\tilde{g}(\widetilde{A}(\Pi_\ast X),\Pi_\ast Y)$, and so $\tilde{h}\circ \Pi=-\tilde{h}$.
Therefore, we get
\begin{equation}\label{14umb}
\int_{\widetilde{M}} \tilde{h}\,d\mu_{\tilde{g}}=0.
\end{equation}
By the assumption,
we have $\|\widetilde{A}-\tilde{h}\Id\|_2\leq\delta \|\tilde{h}\|_2$.
Thus, by Lemma \ref{um5} and (\ref{14umb}), we get
\begin{equation*}
\|\tilde{h}\|_2\leq C\|\widetilde{A}-\tilde{h}\Id\|_2\leq C\delta \|\tilde{h}\|_2.
\end{equation*}
If $\delta>0$ is sufficiently small, this inequality cannot occur.
Therefore, we get the corollary.
\end{proof}

We prove Main Theorem 3 under slightly weaker assumptions (without assuming orientability).
We show that the almost umbilical manifold in sense that $\|B+ g\otimes H\|_2$ is small is close to the standard sphere under the assumption on the Ricci curvature.
\begin{Thm}\label{um7}
Given an integer $n\geq 2$ and positive real numbers $\epsilon>0$ and $K>0$, there exists $\delta(n,K,\epsilon)>0$ such that the following property holds.
Let $(M,g)$ be an $n$-dimensional closed Riemannian manifold and $\iota \colon (M,g)\to \mathbb{R}^{n+1}$ an isometric immersion.
If $\Ric\geq -K\lambda_1(\tilde{g})g$ and $\|B+ g\otimes H\|_2\leq \delta\|H\|_2$ holds,
then we have the following:
\begin{itemize}
\item[(i)] $d_{GH}(M,S^n(1/\|H\|_2))\leq \epsilon/\|H\|_2$ and $M$ is diffeomorphic to the $n$-dimensional standard sphere.
\item[(ii)] $d_{H}(\iota(M),S_M)\leq \epsilon/\|H\|_2$, where $S_M$ is defined by
\begin{equation*}
S_M:=\left\{a+\frac{1}{\Vol(M)}\int_M \iota(x) \,d\mu_g(x) \in\mathbb{R}^{n+1}\colon a\in \mathbb{R}^{n+1} \text{ with } |a|=\frac{1}{\|H\|_2} \right\}.
\end{equation*}
\item[(iii)] There exists a closed subset $E\subset M$ such that $\Vol(M\backslash E)\leq \epsilon \Vol(M)$ and $\iota$ is injective on $E$.
\end{itemize}
\end{Thm}
\begin{proof}
By Corollary \ref{umb8}, $M$ is orientable.
By the scaling, we can assume $\|h\|_2=1$.
Note that $\|H\|_2=\|h\|_2$ and $\|B+ g\otimes H\|_2=\|A-h\Id\|_2$.
Considering $-\nu$ if necessary, we have
\begin{equation*}
\|A-\Id \|_2\leq C \|A-h g\|_2\leq C\delta
\end{equation*}
by Corollary \ref{um6}.
Moreover, by the Reilly inequality \cite{Rei}, we have $\lambda_1(g)\leq n$, and so $\Ric\geq -n K g$.
Thus, we get the theorem by Theorem \ref{um4.5}.
\end{proof}

\begin{Rem}
If $\Ric_g \geq -Kg$ and $\diam(M)\leq D$, then we have $\Ric_{\tilde{g}}\geq -K \tilde{g}$ and $\diam(\widetilde{M})\leq 2D$, and so $\Ric_{\tilde{g}}\geq -C(n,K,D)\lambda_1(\tilde{g}) \tilde{g}$ by the Li-Yau estimate $\lambda_1(\tilde{g})\geq C(n,K,D)>0$ \cite[p.116]{SY}.
\end{Rem}

\section{Examples and the converse of Main Theorem 1}
In this section, we give some examples of the collapsing sequence of Riemannian manifolds such that the values of $\lambda_k(\bar{\Delta}^E)$ are not continuous, and consider the converse of Main Theorem 1 for non-collapsing cases.
\begin{Example}
Take positive integers $n_1, n_2\in \mathbb{Z}_{>0}$ and positive real numbers $K>0$ and $D>0$. Put $n=n_1+n_2$.
Let $(N_i,g_i)$ $(i=1,2)$ be closed Riemannian manifolds of dimension $n_i$ with $\Ric_{g_1}\geq -Kg_1$, $\diam (N_1,g_1)\leq D$ and $\Ric_{g_2}\geq 0$.
For each $r>0$, we consider the metric $G_r=g_1+r^2 g_2$ on $M=N_1\times N_2$.
Then, we have $\Ric_{G_r}\geq -K G_r$ for any $r>0$, and $(M,G_r)$ converges to $(N_1,g_1)$ as $r\to 0$ in the measured Gromov-Hausdorff sense.
By \cite[Proposition 2.14]{Ai}, we have 
\begin{equation*}
\Omega_1(G_r)\leq \max\left\{\frac{n_1-1}{n_1},\frac{n_2-1}{n_2}\right\}\leq \frac{n-2}{n-1}.
\end{equation*}
Thus, by Lemma \ref{p6b} and (\ref{1a}), we have
\begin{equation*}
\|(\nabla^{G_r})^2 f+ f G_r\|_2^2\geq \|(\nabla^{G_r})^2 f+ \frac{\Delta^{G_r} f}{n} G_r\|_2^2\geq \frac{1}{n(n-1)}\|\Delta^{G_r} f\|_2^2
\end{equation*}
for all $f\in C^\infty (M)$.
Therefore, by Proposition \ref{p41c}, $\liminf_{r\to 0} \lambda_1(\bar{\Delta}^{E},(M,G_r))\geq C(n,K,D)>0$ holds.
In particular, if $(N_1,g_1)$ is isometric to the standard sphere of radius 1, then
\begin{equation*}
\liminf_{r\to 0} \lambda_1(\bar{\Delta}^{E},(M,G_r))\geq C(n)>\lambda_1(\bar{\Delta}^{E},(N_1,g_1))=0.
\end{equation*}
\end{Example}
\begin{Example}
Take an integers $n\geq 2$ and a positive real number $K>0$.
Let $\{(M_i,g_i)\}_{i\in \mathbb{N}}$ be arbitrary sequence of $n$-dimensional closed Riemannian manifolds such that $\Ric_{g_i}\geq - Kg_i$ and $(M_i,g_i)$ converges to $S^{n-1}$ in the measured Gromov-Hausdorff sense.
Then, by Main Theorem 1, we have
\begin{equation*}
\liminf_{r\to 0} \lambda_n(\bar{\Delta}^{E},M_i)\geq C(n,K)>\lambda_n(\bar{\Delta}^{E},S^{n-1})=0.
\end{equation*}
\end{Example}
The above examples tells us that the continuity of the value of $\lambda_k(\bar{\Delta}^E)$  in the measured Gromov-Hausdorff topology does not hold in general.
However, under the non-collapsing assumption, we have the following property using the methods in \cite{Ho2} (see Appendix B for details).
For any non-collapsed sequence of $n$-dimensional closed Riemannian manifolds $\{(M_i,g_i)\}_{i\in \mathbb{N}}$ with $|\Ric_{g_i}|\leq K$ and $\diam(M_i)\leq D$, and its Gromov-Hausdorff limit $X$, we can define $\lambda_k(\bar{\Delta}^{E},X)$, and we have $\lim_{i\to \infty}\lambda_k(\bar{\Delta}^{E},M_i)=\lambda_k(\bar{\Delta}^{E},X)$ for all $k\in \mathbb{N}$.
In particular, we get the following.
\begin{Prop}\label{p51a}
Given an integer $n\geq 2$ and positive real numbers $\epsilon>0$ and $K>0$, there exists a positive constant $\delta(n,K,\epsilon)>0$ such that if $(M,g)$ is an $n$-dimensional closed Riemannian manifold with $|\Ric_g|\leq K$ and $d_{GH}(M,S^n)\leq \delta$, then $\lambda_{n+1}(\bar{\Delta}^{E},M)\leq \epsilon$.
\end{Prop}
Since the standard sphere is smooth, the above proposition also follows from \cite[Theorem 7.2]{CC1}.
The question whether we need to assume the upper bound of Ricci curvature in Proposition \ref{p51a} is related to (Q5.4) and (Q5.5) in \cite{Ho2}.
We give the following conjecture.
\begin{Conj}
Given an integer $n\geq 2$ and positive real numbers $\epsilon>0$ and $K>0$, there exists a positive constant $\delta(n,K,\epsilon)>0$ such that if $(M,g)$ is an $n$-dimensional closed Riemannian manifold with $\Ric_g \geq -Kg$ and $d_{GH}(M,S^n)\leq \delta$, then $\lambda_{n+1}(\bar{\Delta}^{E},M)\leq \epsilon$.
\end{Conj}


\appendix
\section{Spherical multi-suspension}

In this appendix we show the following multi-suspension theorem, which gives an approximation of the shape of the Riemannian manifold under the pinching condition on $\lambda_k(\bar{\Delta}^E)$, based on the methods of \cite{Ho}. 
As a consequence, we give another proof of Main Theorem 1.
\begin{Thm}\label{apa1}
Given integers $n\geq 2$ and $1\leq k\leq n+1$, and positive real numbers $\epsilon>0$, $K>0$ and $D>0$, there exists $\delta(n,K,D,\epsilon)>0$ such that if $(M,g)$ is an $n$-dimensional closed Riemannian manifold with  $\Ric \geq-K g$, $\diam(M)\leq D$ and $\lambda_k(\bar{\Delta}^E)\leq\delta$, then we have one of the following:
\begin{itemize}
\item[(i)] $d_{GH}(M,S^{k-1})\leq \epsilon$,
\item[(ii)] $d_{GH}(M,S^k)\leq\epsilon$,
\item[(iii)] There exists a compact geodesic space $Z$ such that
$d_{GH}(M,S^{k-1}\ast Z)\leq \epsilon$, where $S^{k-1}\ast Z$ denotes the $k$-fold spherical suspension of $Z$ $($see the definition below$)$.
\end{itemize}
\end{Thm}

In the following, we show that each $B_{ij}$ in Proposition \ref{p2k} is close to the spherical suspension, and prove Theorem \ref{apa1} by iteration.

We first recall some basic definitions.
\begin{Def}\label{dapa1}
Let $Z$ be a metric space.
\begin{itemize}
\item[(i)] We say $Z$ is a geodesic space if for each $z_1,z_2\in Z$, there exists a minimal geodesic $c\colon[0,d(z_1,z_2)]\to Z$ connecting them, i.e., $c(0)=z_1$, $c(d(z_1,z_2))=z_2$ and $d(c(t_1),c(t_2))=|t_1-t_2|$ holds for all $t_1,t_2\in[0,d(z_1,z_2)]$.
Given points $z_1,z_2\in Z$, $\gamma_{z_1,z_2}$ denotes one of minimal geodesics connecting them.
\item[(ii)] We define an equivalence relation $\sim$ on $[0,\pi]\times Z$ by $(0,z_1)\sim (0,z_2)$ and $(\pi,z_1)\sim (\pi,z_2)$ for all $z_1,z_2\in Z$.
Define a metric $d$ on $([0,\pi]\times Z)/\sim$ to be
\begin{equation*}
\cos d([t_1,z_1],[t_2,z_2])=\cos t_1\cos t_2+\sin t_1\sin t_2\cos\min\{d(z_1,z_2),\pi\}
\end{equation*}
and $0\leq d([t_1,z_1],[t_2,z_2])\leq\pi$ for all $[t_1,z_1],[t_2,z_2]\in ([0,\pi]\times Z)/\sim$.
The metric space $(([0,\pi]\times Z)/\sim,d)$ is called the spherical suspension of $Z$ and denoted by $S^0\ast Z$.
Define $0^\ast := [0,z]\in S^0\ast Z$ and $\pi^\ast:=[\pi,z]\in S^0\ast Z$.
\item[(iii)] For any positive integer $k\in \mathbb{Z}_{>0}$, we define $S^k\ast Z:=S^0\ast (S^{k-1}\ast Z)$ and $[t_1,\ldots,t_{k+1},z]:=[t_1,[t_2,\ldots, t_{k+1},z]]\in S^k\ast Z$ inductively.
We call $S^k\ast Z$ the $k$-fold spherical suspension of $Z$.
\item[(iv)] We define an equivalence relation $\sim'$ on $[0,\infty) \times Z$ by $(0,z_1)\sim' (0,z_2)$ for all $z_1,z_2\in Z$.
Define a metric $d$ on $([0,\infty)\times Z)/\sim'$ by
\begin{equation*}
d([t_1,z_1],[t_2,z_2]):=\left(t_1^2+t_2^2-2t_1 t_2 \cos \min\{d(z_1,z_2),\pi\}\right)^\frac{1}{2}.
\end{equation*}
The metric space $(([0,\infty)\times Z)/\sim',d)$ is called the metric cone of $Z$ and denoted by $C(Z)$.
Put $0^\ast =[0,z]\in C(Z)$.
\end{itemize}
\end{Def}

The spherical suspension is the generalization of the warped product metric with warping function $\sin t$.
For an $(n-1)$-dimensional Riemannian manifold $(Z,g_{n-1})$, the distance on $(0,\pi)\times Z$ induced by the warped product metric $g=d t^2+ \sin^2 t g_{n-1}$ coincides with the spherical suspension metric. 
If the Riemannian metric extends smoothly to the endpoints $0^\ast$ and $\pi^\ast$ in $([0,\pi]\times Z)/\sim$, then $(Z,g_{n-1})$ is isometric to the $(n-1)$-dimensional standard sphere of radius $1$.
Note that the warped product metric on $([0,\pi]\times S^{n-1})/\sim$ with warping function $\sin t$ coincides with the metric on the $n$-dimensional standard sphere, and so we have $S^n=S^0\ast S^{n-1}$.
Inductively we have $S^n=S^{n-1}\ast \{0,\pi\}$.
\begin{Def}\label{dapa2}
Let $(X,x,d)$ be a pointed metric space. 
\begin{itemize}
\item[(i)] We say a sequence of pointed metric spaces $\{(X_i,x_i,d_i)\}_{i\in \mathbb{N}}$ converges to $(X,x,d)$ in the pointed Gromov-Hausdorff sense if
for all $R>0$, there exists a sequence of positive real numbers $\{\epsilon_i\}_{i\in \mathbb{N}}$ such that $\lim_{i\to\infty}\epsilon_i= 0$ and a sequence of $\epsilon_i$-Hausdorff approximation maps $\{\psi_i\colon B_R(x)\to B_R(x_i)\}_{i\in\mathbb{N}}$.
\item[(ii)] We say a pointed metric space $(Y,y,d')$ is the tangent cone of $X$ at $x$ if there exists a sequence of positive real numbers $\{r_i\}_{i\in\mathbb{N}}$ such that $\lim_{i\to\infty}r_i= 0$ and $\{(X,x,r_i^{-1} d)\}_{i\in\mathbb{N}}$ converges to $(Y,y,d')$ in the pointed Gromov-Hausdorff sense. 
\end{itemize}
\end{Def}
Note that if a (pointed) metric space $X$ (or $(X,x)$) is a (pointed) Gromov-Hausdorff limit of a sequence of geodesic spaces, then $X$ is also geodesic space.

We next recall some facts and definitions about Ricci limit spaces.
Let $\{(M_i,g_i,p_i)\}$ be a sequence of pointed closed Riemannian manifolds with $\Ric_i\geq -K g_i$ and $\diam(M_i)\leq D$ ($K,D>0$).
Then, there exist a subsequence (denote it again by $\{(M_i,g_i,p_i)\}$) and a pointed geodesic space $(X,p,\nu)$ with a Radon measure $\nu$ such that
$\{(M_i,g_i)\}$ converges to $(X,p)$ in the pointed Gromov-Hausdorff sense and
$$
\lim_{i\to\infty} \frac{\Vol(B_r(x_i))}{\Vol(B_1(p_i))}=\nu(B_r(x))
$$
for all $r>0$, $x_i\in M_i$ and $x\in X$ with $\lim_{i\to \infty} d(\psi_i(x_i),x)=0$, where $\psi_i$ is a Hausdorff approximation map that we used for the definition of the pointed Gromov-Hausdorff convergence
(see Theorem 1.6 and Theorem 1.10 of \cite{CC1}).
We call such $\nu$ a limit measure.
We can consider the cotangent bundle $\pi \colon T^\ast X \to X$ with a canonical inner product by \cite{Ch0} and \cite{CC3} (see also \cite[Section 2]{Ho1} for a short review).
We have $\nu(X\setminus \pi(T^\ast X))=0$ and $T^\ast_x X:=\pi^{-1}(x)$ is a finite dimensional vector space with an inner product for all $x\in \pi(T^\ast X)$.
For all Lipschitz function $f$ on $X$, we can define $d f(x)\in T_x^\ast X$ for almost all $x\in X$, and we have $d f\in L^\infty(T^\ast X)$.
Let $\LIP(X)$ be the set of the Lipschitz functions on $X$. For all $f\in \LIP(X)$, we define $\|f\|_{H^{1,2}}^2=\|f\|_2^2+\|d f\|_2^2$.
Let $H^{1,2}(X)$ be the completion of $\LIP(X)$ with respect to this norm.
Define
\begin{equation*}
\begin{split}
\mathcal{D}^2(\Delta,\nu):=\Big\{f\in H^{1,2}(X)& : \text{there exists $F\in L^2(X)$ such that}\\
&\int_X \langle df, dh \rangle\,d \nu=\int_X F h\,d \nu \text{ for all $h\in H^{1,2}(X)$} \Big\}.
\end{split}
\end{equation*}
For any $f\in \mathcal{D}^2(\Delta,\nu)$, the function $F\in L^2(X)$ is uniquely determined. Thus, we define $\Delta f:=F$.
For all sequence $f_i\in L^2(M_i)$ and $f\in L^2(X)$, we say that $f_i$ converges to $f$ weakly in $L^2$ (see \cite{Ho1}) if
\begin{equation*}
\sup_{i\in \mathbb{N}}\|f_i\|_2<\infty,
\end{equation*}
and for all $r>0$, $x_i\in M_i$ and $x\in X$ with $\lim_{i\to \infty}d(\psi_i(x_i),x)=0$, we have
\begin{equation*}
\lim_{i\to \infty}\frac{1}{\Vol(B_1(p_i))}\int_{B_r(x_i)} f_i \,d \mu_{g_i}=\int_{B_r(x)} f \,d \nu.
\end{equation*}
We say that $f_i$ converges to $f$ strongly in $L^2$ if $f_i$ converges to $f$ weakly in $L^2$, and
\begin{equation*}
\limsup_{i\to \infty} \|f_i\|_2\leq \|f\|_2
\end{equation*}
holds.
Note that if $f_i$ converges to $f$ weakly in $L^2$, then
\begin{equation*}
\liminf_{i\to \infty}\|f_i\|_2\geq \|f\|_2
\end{equation*}
by \cite[Proposition 3.29]{Ho1}.

We need the following easy lemma.
\begin{Lem}\label{apa2}
Let $(M,g)$ be a complete Riemannian manifold.
Take arbitrary $\epsilon>0$.
If points $p,a,b\in M$ satisfies $d(p,a)+d(a,b)\leq d(p,b)+\epsilon$ and $\gamma_{a,b}(s)\in I_p\setminus \{p\}$ for almost all $s\in [0,d(a,b)]$,
then we have $d(p,a)+d(a,\gamma_{a,b}(s))\leq d(p,\gamma_{a,b}(s))+\epsilon$ for all $s\in [0,d(a,b)]$, and
\begin{equation*}
\int_0^{d(a,b)}\left(1-\langle \dot{\gamma}_{a,b}(s),\frac{\partial}{\partial r}\rangle\right)\, d s \leq \epsilon,
\end{equation*}
where $\frac{\partial}{\partial r}$ denotes $\dot{\gamma}_{p,x}(d(p,x))\in T_x M$ for all $x\in I_p\setminus \{p\}$.
\end{Lem}
\begin{Rem}
Since $\left|\dot{\gamma}_{a,b}(s)-\frac{\partial}{\partial r}\right|^2 =2\left(1-\langle \dot{\gamma}_{a,b}(s),\frac{\partial}{\partial r}\rangle\right)$, we have
\begin{equation*}
\int_0^{d(a,b)}\left|\dot{\gamma}_{a,b}(s)-\frac{\partial}{\partial r}\right|\, d s \leq (2 d(a,b)\epsilon)^{\frac{1}{2}}.
\end{equation*}
\end{Rem}
\begin{proof}
The first assertion is trivial.
The map $d(\gamma_{a,b}(\cdot),p)\colon [0,d(a,b)]\to \mathbb{R}_{>0}$ is Lipschitz continuous and $(d(\gamma_{a,b}(s),p))'=\langle \dot{\gamma}_{a,b}(s),{\partial}/{\partial r}\rangle$ for all $s$ with $\gamma_{a,b}(s)\in I_p\setminus \{p\}$.
Since
\begin{equation*}
\int_0^{d(a,b)} (d(\gamma_{a,b}(s),p))' \,d s=d(b,p)-d(a,p),
\end{equation*} 
we get the lemma.
\end{proof}
The following proposition asserts that each $B_{ij}$ in Proposition \ref{p2k} is close to the spherical suspension.
\begin{Prop}\label{apa3}
Given an integers $n\geq 2$, and positive real numbers $K>0$ and $D>0$, then there exists $\eta(n,K,D)>0$ such that the following property holds.
Take a positive real number $0<\delta\leq \eta$.
Let $(M,g)$ be an $n$-dimensional closed Riemannian manifold with $\Ric\geq -Kg$ and $\diam(M)\leq D$.
Suppose that a non-zero function $f\in C^\infty(M)$ satisfies $\|\nabla^2 f+f g\|_2\leq \delta\|f\|_2$.
Under the notation of Proposition \ref{p2k}, for all $i,j$ with $|d(x_i,x_j)-\pi|\leq \delta^\frac{1}{200n^2}$, we have $d_{GH}(B_{ij},S^0\ast Z_{ij})\leq C(n,K,D) \delta^{\frac{1}{4000n^2}}$, where we put $Z_{ij}:=\{x\in B_{ij}:d(x_i,x)=d(x_j,x)\}$.
\end{Prop}
\begin{proof}
We first suppose that $\delta\leq \delta_{\Db}$.
Since we have either $m_i=m_j+1$ or $m_i=m_j-1$, 
we can assume that $x_i\in A_{m_i}$ and $x_j\in A_{m_i+1}$.

Let us apply Lemma \ref{p2g} to each $y_1\in Q$ and $y_2\in D(y_1)$.
\begin{Clm}\label{clap1}
For all $y_1\in Q$ and $y_2\in D(y_1)$ $($see Lemma \ref{p2e}$)$, we have
\begin{equation*}
\begin{split}
\big|\cos d(p,y_1) - &\cos d(p,y_2) \cos d(y_1,y_2) \\
 &+\langle \nabla f_1 (y_2),\dot{\gamma}_{y_1,y_2}\rangle \sin d(y_1,y_2)\big|\leq C(n,K,D)\delta^{\frac{1}{48n}}.
\end{split}
\end{equation*}
\end{Clm}
\begin{proof}[Proof of Claim \ref{clap1}]
Take arbitrary $y_1\in Q$ and $y_2\in D(y_1)$.
Since we have
\begin{equation*}
\int_0^{d(y_1,y_2)} \left|\frac{\partial^2}{\partial s^2} (f_1\circ \gamma_{y_1,y_2}(s)) + f_1 \circ \gamma_{y_1,y_2}(s)\right| \, d s \leq \delta^{\frac{1}{6}},
\end{equation*}
we get $\left|f_1\circ \gamma_{y_1,y_2}(d(y_1,y_2)-s) - f_1(y_2) \cos s + \langle \nabla f_1 (y_2),\dot{\gamma}_{y_1,y_2}\rangle \sin s\right|\leq \sinh D \delta^{\frac{1}{6}}$ for all $s\in [0,d(y_1,y_2)]$ by applying Lemma \ref{p2g} to $f_1\circ \gamma_{y_1,y_2}(d(y_1,y_2)-s)$.
Putting $s=d(y_1,y_2)$,
we get the claim by $\|f_1-\cos d(p,\cdot)\|_{\infty}\leq C\delta^{\frac{1}{48n}}$ (see Proposition \ref{p2i}).
\end{proof}
Let us replace $\nabla f$ by $\nabla \cos d(p,\cdot)$ in Claim \ref{clap1} using $\|\nabla f_1-\nabla\cos d(p,\cdot)\|_2\leq C\delta^{\frac{1}{48n}}$ (see Proposition \ref{p2i}),
and get the integral pinching condition.
\begin{Clm}\label{clap2}
For all $y_1\in Q$, we define $F_{y_1}\colon M\to \mathbb{R}_{\geq0}$ by
\begin{equation*}
F_{y_1}(y_2):=\Big|\cos d(p,y_1) - \cos d(p,y_2) \cos d(y_1,y_2) - \langle \dot{\gamma}_{y_1,y_2},\frac{\partial}{\partial r} \rangle \sin d(p,y_2) \sin d(y_1,y_2)\Big|
\end{equation*} for all $y_2 \in I_p\cap I_{y_1}\setminus\{p,y_1\}$ and $F_{y_1}(y_2)=0$ otherwise, where $\frac{\partial}{\partial r}$ denotes $\dot{\gamma}_{p,x}(d(p,x))\in T_x M$ for all $x\in I_p\setminus \{p\}$.
Then, we have
\begin{equation*}
\frac{1}{\Vol(M)}\int_M F_{y_1} \,d\mu_g\leq C(n,K,D)\delta^{\frac{1}{48n}}.
\end{equation*}
\end{Clm}
\begin{proof}[Proof of Claim \ref{clap2}]
Take arbitrary $y_1\in Q$.
Since have $\Vol(M\setminus D(y_1))\leq \delta^\frac{1}{6}\Vol(M)$ and $F_{y_1}\leq 3$, we get
\begin{equation*}
\begin{split}
\frac{1}{\Vol(M)}\int_M F_{y_1} \,d\mu_g
\leq&\frac{1}{\Vol(M)}\int_{D(y_1)} \Big|\cos d(p,y_1) - \cos d(p,y_2) \cos d(y_1,y_2)\\
& \quad\qquad\quad\qquad\qquad\quad+ \langle \nabla f_1 (y_2),\dot{\gamma}_{y_1,y_2} \rangle \sin d(y_1,y_2)\Big|\, d y_2\\
&\quad+\|\nabla f_1+\sin d(p,\cdot)\frac{\partial}{\partial r}\|_1+3 \delta^\frac{1}{6}\\
\leq &C\delta^{\frac{1}{48n}}
\end{split}
\end{equation*}
by Claim \ref{clap1} and Proposition \ref{p2i}.
Thus, we get the claim.
\end{proof}
Similarly to Lemma \ref{p2e}, let us get the pinching condition on the segments using the integral pinching condition on the manifold.
\begin{Clm}\label{clap3}
For each $y_1\in Q$ and $y_2\in M$, we define
\begin{equation*}
\begin{split}
E_{y_1}(y_2):=\Big\{y_3\in I_{y_2}&\setminus\{y_2\}:  \int_0^{d(y_2,y_3)} F_{y_1} \circ\gamma_{y_2,y_3}(s)\, d s\leq\delta^\frac{1}{144n},\text{ and we have} \\
&\gamma_{y_2,y_3}(s)\in I_p\cap I_{y_1}\setminus \{p,y_1\} \text{ for almost all $s\in[0,d(y_2,y_3)]$}
\Big\},
\end{split}
\end{equation*}
and
\begin{equation*}
R_{y_1}:=\{y_2\in M: \Vol(E_{y_1}(y_2))\geq(1-\delta^\frac{1}{144n})\Vol(M)\}.
\end{equation*}
Then, we have
\begin{equation*}
\Vol(R_{y_1})\geq(1- C(n,K,D)\delta^\frac{1}{144n})\Vol(M).
\end{equation*}
\end{Clm}
\begin{proof}[Proof of Claim \ref{clap3}]
By applying the segment inequality Theorem \ref{p2f} to functions $F_{y_1}$ and $1-\chi_{I_p\cap I_{y_1}\setminus \{p,y_1\}}$ (here $\chi_{I_p\cap I_{y_1}\setminus \{p,y_1\}}(x)=1$ for all $x\in I_p \cap I_{y_1}\setminus \{p,y_1\}$ and $\chi_{I_p\cap I_{y_1}\setminus \{p,y_1\}}(x)=0$ otherwise),
we get the claim by Claim \ref{clap2} similarly to Lemma \ref{p2e}.
\end{proof}
Now, we define an approximation map.
Define $\phi_{ij}\colon B_{ij}\to Z_{ij}$ as follows.
If $x\in B_{ij}$ satisfies $d(x,x_i)\leq d(x,x_j)$, then
\begin{align*}
d(x_i,\gamma_{x,x_j}(0))\leq& d(x_j,\gamma_{x,x_j}(0)),\\
d(x_i,\gamma_{x,x_j}(d(x,x_j)))\geq &d(x_j,\gamma_{x,x_j}(d(x,x_j)))=0.
\end{align*}
Thus, we can choose $s\in[0,d(x,x_j)]$ with $d(x_i,\gamma_{x,x_j}(s))= d(x_j,\gamma_{x,x_j}(s))$, and define $\phi_{ij}(x):=\gamma_{x,x_j}(s)\in \Imag \gamma_{x,x_j}\cap Z_{ij}$.
Similarly, if $d(x,x_i)> d(x,x_j)$, we define $\phi_{ij}(x)$ to be $\phi_{ij}(x)\in \Imag \gamma_{x,x_i}\cap Z_{ij}$. We define an approximation map $\psi_{ij}\colon B_{ij}\to S^0\ast Z_{ij}$ by
\begin{empheq}[left={\psi_{i j} (x) :=\empheqlbrace}]{align*}
&\qquad 0^\ast && (d(x,x_i)\leq\delta^\frac{1}{2000n^2}),\\
&\quad [d(x,x_i),\phi_{ij}(x)]&& (d(x,x_i)>\delta^\frac{1}{2000n^2} \text{ and } d(x,x_j)>\delta^\frac{1}{2000n^2}),\\
&\qquad\pi^\ast &&(d(x,x_j)\leq \delta^\frac{1}{2000n^2}).
\end{empheq}

We list some basic properties of the metric on $B_{i j}$.
\begin{Clm}\label{clap4}
We have the following properties.
\begin{itemize}
\item[(i)]
For all $x\in Z_{ij}$, we have $|d(x,x_i)-\pi/2|\leq C\delta^\frac{1}{250n^2}$ and $|d(x,x_j)-\pi/2|\leq C\delta^\frac{1}{250n^2}$.
\item[(ii)] For all $x\in B_{ij}$, we have $|d(p,x)-(m_i \pi +d(x,x_i))|\leq C\delta^\frac{1}{250n^2}$ and $|d(p,x)-((m_i +1)\pi -d(x,x_j))|\leq C\delta^\frac{1}{250n^2}$.
\item[(iii)] For all $x\in B_{ij}$ with $d(x,x_i)\leq d(x,x_j)$, we have $d(p,x)+d(x,\phi_{ij}(x))\leq d(p,\phi_{ij}(x))+ C\delta^\frac{1}{250n^2}$.
\item[(iv)] For all $x\in B_{ij}$ with $d(x,x_i)> d(x,x_j)$, we have $d(p,\phi_{ij}(x))+d(\phi_{ij}(x),x)\leq d(p,x)+ C\delta^\frac{1}{250n^2}$.
\end{itemize}
\end{Clm}
Claim \ref{clap4} is easy consequence of the inequalities $|d(p,x_i)-m_i\pi|\leq \delta^\frac{1}{100n}$, $|d(p,x_j)-(m_i+1)\pi|\leq \delta^\frac{1}{100n}$, $|d(x_i,x_j)-\pi|\leq \delta^\frac{1}{200n^2}$ and $d(x,x_i)+d(x,x_j)\leq d(x_i,x_j)+\delta^\frac{1}{250n^2}$ ($x\in B_{i j}$) in Proposition \ref{p2k}.

Let us compare the cosine of the metric on $B_{i j}$ and that of the spherical suspension metric (see Definition \ref{dapa1} (ii)).
\begin{Clm}\label{clap5}
Define $G_{i j}:=\{x\in B_{i j}: d(x,x_i)>\delta^\frac{1}{2000n^2}, d(x,x_j)>\delta^\frac{1}{2000n^2}\}$.
Then, we have
\begin{equation*}
\begin{split}
\Big|\cos d(y_1,y_2)&-\cos d(x_i,y_1)\cos d(x_i,y_2) \\
&-\sin d(x_i,y_1)\sin d(x_i,y_2)\cos d(\phi_{i j}(y_1),\phi_{i j}(y_2)) \Big|\leq C(n,K,D)\delta^\frac{1}{1000n^2}
\end{split}
\end{equation*}
for all $y_1,y_2\in G_{ij}$.
\end{Clm}
\begin{proof}[Proof of Claim \ref{clap5}]
Take arbitrary $y_1, y_2\in G_{i j}$.
Put $y_3:=\phi_{ij}(y_2)$.
By Lemma \ref{p2e}, Claim \ref{clap3} and the Bishop-Gromov inequality (Theorem \ref{p2h}), there exists points $\tilde{y}_1\in Q$, $\tilde{y}_2\in R_{\tilde{y}_1}$ and $\tilde{y}_3 \in E_{\tilde{y}_1}(\tilde{y}_2)$ with $d(y_k,\tilde{y}_k)\leq C\delta^\frac{1}{144n^2}$ for $k=1,2,3$.
Put $l(s)=d(\tilde{y}_1,\gamma_{\tilde{y}_2,\tilde{y}_3}(s))$.
Then, by the first variation formula, we have $l'(s)=\langle \dot{\gamma}_{\tilde{y}_1,\gamma_{\tilde{y_2},\tilde{y}_3}(s)}(l(s)), \dot{\gamma}_{\tilde{y_2},\tilde{y}_3}(s)\rangle$ for all $s\in I_{\tilde{y_1}}\setminus \{\tilde{y_1}\}$.
Suppose that $d(y_2,x_i)\leq d(y_2,x_j)$.
Then, we have $d(p,\tilde{y}_2)+d(\tilde{y}_2,\tilde{y}_3)\leq d(p,\tilde{y}_3)+C\delta^\frac{1}{250n^2}$ by Claim \ref{clap4} (iii).
Thus, by Lemma \ref{apa2} and the definition of $E_{\tilde{y}_1}(\tilde{y}_2)$, we get
\begin{equation}\label{apec1}
\begin{split}
\int_0^{d(\tilde{y}_2,\tilde{y}_3)} |\cos d(p,\tilde{y_1})&-\cos (d(p,\tilde{y}_2)+s)\cos l(s)\\
&+ \sin  (d(p,\tilde{y}_2)+s) (\cos l(s))'|\,ds\leq C\delta^\frac{1}{500n^2}.
\end{split}
\end{equation}
Here, we have
\begin{equation}\label{apec2}
\begin{split}
&\cos d(p,\tilde{y_1})-\cos (d(p,\tilde{y}_2)+s)\cos l(s)+ \sin  (d(p,\tilde{y}_2)+s) (\cos l(s))'\\
=&\sin (d(p,\tilde{y}_2)+s) ( \cos l(s)-\cos d(p,\tilde{y_1})\cos (d(p,\tilde{y}_2)+s) )'\\
&\qquad -(\sin (d(p,\tilde{y}_2)+s))' ( \cos l(s)-\cos d(p,\tilde{y_1})\cos (d(p,\tilde{y}_2)+s))\\
=&(\sin (d(p,\tilde{y}_2)+s))^2\left(\frac{\cos l(s)-\cos d(p,\tilde{y_1})\cos (d(p,\tilde{y}_2)+s)}{\sin (d(p,\tilde{y}_2)+s)}\right)'.
\end{split}
\end{equation}
Since we have  $m_i\pi+\delta^\frac{1}{2000n^2}-C\delta^\frac{1}{250n^2} \leq d(p,\tilde{y}_2)+s\leq (m_i+1/2)\pi+C \delta^\frac{1}{250n^2}$ for all $s\in[0,d(\tilde{y}_2,\tilde{y}_3)]$,
we can assume that $m_i\pi+\frac{1}{2}\delta^\frac{1}{2000n^2}\leq d(p,\tilde{y}_2)+s\leq (m_i+1)\pi-\frac{1}{2}\delta^\frac{1}{2000n^2}$ by taking $\delta$ sufficient small.
Thus, we have
\begin{equation}\label{apec3}
\frac{1}{(\sin (d(p,\tilde{y}_2)+s))^2}\leq C \delta^{-\frac{1}{1000n^2}}.
\end{equation}
By (\ref{apec1}), (\ref{apec2}), (\ref{apec3}) and Claim \ref{clap4}, we get
\begin{equation}\label{apec4}
\begin{split}
\Big|\cos d(y_1,y_2)&-\cos d(p,y_1)\cos d(p,y_2)\\
&-(-1)^{m_i}\sin d(p,y_2)\cos d(y_1,\phi_{i j}(y_2)) \Big|\leq C\delta^\frac{1}{1000n^2}.
\end{split}
\end{equation}
Similarly, we have (\ref{apec4}) for the case $d(y_2,x_i)> d(y_2,x_j)$.
Using (\ref{apec4}) for the pair $(\phi_{i j}(y_2),y_1)$, we get
\begin{equation}\label{apec5}
\begin{split}
\Big|\cos d(\phi_{i j}(y_2),y_1)&-\cos d(p,\phi_{i j}(y_2))\cos d(p,y_1) \\
&-(-1)^{m_i}\sin d(p,y_1)\cos d(\phi_{i j}(y_2),\phi_{i j}(y_1)) \Big|\leq C\delta^\frac{1}{1000n^2}.
\end{split}
\end{equation}
By Claim \ref{clap4} (i) we have $|\cos d(p,\phi_{i j}(y_2))|\leq C\delta^\frac{1}{250n^2}$.
Combining this, (\ref{apec4}) and (\ref{apec5}), we get the claim.
\end{proof}
The following claim shows that $\psi_{i j}$ satisfies one of the property of the Hausdorff approximation map (Definition \ref{hap} (i)).
\begin{Clm}\label{hapi}
We have
\begin{equation*}
|d(y_1,y_2)-d(\psi_{i j}(y_1),\psi_{i j}(y_2)) |\leq C(n,K,D)\delta^\frac{1}{2000n^2}
\end{equation*}
for all $y_1,y_2\in B_{i j}$.
\end{Clm}
\begin{proof}[Proof of Claim \ref{hapi}]
Since $d(y_1,y_2)\leq \pi +C\delta^\frac{1}{250n^2}$ and $d(\phi_{i j}(y_2),\phi_{i j}(y_1))\leq \pi +C\delta^\frac{1}{250n^2}$, we get the claim by Lemma \ref{p2j} and Claim \ref{clap5}.
\end{proof}
Finally, we show that $\psi_{i j}$ satisfies the other property of the Hausdorff approximation map.
\begin{Clm}\label{hapii}
$\psi_{ij}(B_{i j})$ is $C(n,K,D)\delta^\frac{1}{4000n^2}$-dense in $S^0\ast Z_{i j}$ $($see Definition \ref{hap}$)$.
\end{Clm}
\begin{proof}[Proof of Claim \ref{hapii}]
Take arbitrary $[t,z]\in S^0\ast Z_{i j}$ with $2\delta^\frac{1}{2000n^2}<t<\pi-2\delta^\frac{1}{2000n^2}$.
Suppose that $t\geq \pi/2$.
Put $x=\gamma_{z,x_j}(t-\pi/2)$.
Since we have $d(x,z)=t-\pi/2$, $|d(x_i,z)-\pi/2|\leq C\delta^\frac{1}{250n^2}$ and $|d(x_i,x)-t|\leq C\delta^\frac{1}{250n^2}$,
we get $x\in G_{i j}$ by taking $\delta$ sufficient small, and by Claim \ref{clap5},
\begin{equation*}
\left|\cos \left(t-\frac{\pi}{2}\right)-\sin t \cos d(\phi_{i j}(x),z) \right|\leq C\delta^\frac{1}{1000n^2}.
\end{equation*}
Since $\cos \left(t-\pi/2\right)=\sin t$ and $1/|\sin t|\leq C\delta^{-\frac{1}{2000n^2}}$, 
\begin{equation*}
\left|1- \cos d(\phi_{i j}(x),z) \right|\leq C\delta^\frac{1}{2000n^2}.
\end{equation*}
Thus, we get
\begin{equation*}
|\cos d(\psi_{i j}(x),[t,z])-1|\leq C\delta^\frac{1}{2000n^2}.
\end{equation*}
By Lemma \ref{p2j}, we get $d(\psi_{i j}(x),[t,z])\leq C\delta^\frac{1}{4000n^2}$.
Similarly, putting $x=\gamma_{z,x_i}(\pi/2-t)$, we have $d(\psi_{i j}(x),[t,z])\leq C\delta^\frac{1}{4000n^2}$ for the case $t<\pi/2$.
Since 
\begin{equation*}
\left\{[t,z]\in S^0\ast Z_{i j}: 2\delta^\frac{1}{2000n^2}<t<\pi-2\delta^\frac{1}{2000n^2} \text{ and } z\in Z_{i j}\right\}
\end{equation*} is $2\delta^\frac{1}{2000n^2}$-dense in $S^0\ast Z_{i j}$,
we get the claim.
\end{proof}
By Claim \ref{hapi} and Claim \ref{hapii}, we get that the map $\psi_{ij}\colon B_{ij}\to S^0\ast Z_{ij}$ is a $C(n,k,D)\delta^\frac{1}{4000n^2}$-Hausdorff approximation.
Thus, we get the proposition.
\end{proof}
Rewording Proposition \ref{p2k} and Proposition \ref{apa3} by using the limit space, we immediately get the following corollary.
\begin{Cor}\label{CHCO}
Take an integer $n\geq 2$, positive real numbers $K>0$ and $D>0$ and a sequence of positive real numbers $\{\delta_i\}$ $(i\in \mathbb{N})$ with $\delta_i\to 0$.
Suppose that a sequence of $n$-dimensional closed Riemannian manifolds $\{(M_i,g_i)\}$ $(i\in\mathbb{N})$ with $\Ric_i\geq -Kg_i$ and $\diam (M_i)\leq D$ converges to a geodesic space $(X,d)$ in the Gromov-Hausdorff topology, and there exists a non-zero function $f_i\in C^\infty(M_i)$ with $\|(\nabla^{g_i})^2 f_i+f_i g_i\|_2\leq \delta_i \|f_i\|_2$ for each $i$.
Then, there exist finite points $x_0,\ldots,x_N\in X$ $(N\leq N_3)$ such that the following properties holds.
\begin{itemize}
\item[(i)] For each $i,j=0,\ldots, N$, there exists an integer $k_{ij}\in\mathbb{Z}_{\geq0}$ with $d(x_i,x_j)=\pi k_{ij}$.
Moreover, $k_{ij}$ is even if and only if $(-1)^{k_{0i}}=(-1)^{k_{0j}}$ holds.
\item[(ii)] For each $i,j=0,\ldots, N$ with $d(x_i,x_j)=\pi$, we define
\begin{equation*}
\begin{split}
B_{ij}:=&\{x\in X: d(x_i,x)+d(x,x_j)= d(x_i,x_j)\}\\
\widetilde{B}_{ij}:=&B_{ij}\setminus \{x_i,x_j\}.
\end{split}
\end{equation*}
Then, $X=\bigcup_{ij} B_{ij}$ holds, and $\widetilde{B}_{ij}\cap B_{kl}=\emptyset$ if $\{k,l\}\neq \{i,j\}$.
\item[(iii)] Suppose that $i,j,k=0,\ldots, N$ satisfy $j\neq k$, $d(x_i,x_j)=\pi$ and $d(x_i,x_k)=\pi$.
Then, for any $x\in B_{ij}$ and $y\in B_{ik}$, we have $d(x,y)=d(x,x_i)+d(x_i,y)$.
\item[(iv)] For each $i,j=0,\ldots, N$ with $d(x_i,x_j)=\pi$, there exists a metric space $Z_{ij}$ such that $B_{ij}$ is isometric to the spherical suspension $S^0\ast Z_{ij}$.
\end{itemize}
\end{Cor}

More detailed consideration enable us to exclude the case $N\geq 2$.
\begin{Cor}\label{CHCO2}
Take an integer $n\geq 2$, positive real numbers $K>0$ and $D>0$ and a sequence of positive real numbers $\{\delta_i\}$ $(i\in \mathbb{N})$ with $\delta_i\to 0$.
Suppose that a sequence of $n$-dimensional closed Riemannian manifolds $\{(M_i,g_i)\}$ $(i\in\mathbb{N})$ with $\Ric_i\geq -Kg_i$ and $\diam (M_i)\leq D$ converges to a geodesic space $(X,d)$ in the Gromov-Hausdorff topology, and there exists a non-zero function $f_i\in C^\infty(M_i)$ with $\|(\nabla^{g_i})^2 f_i+f_i g_i\|_2\leq \delta_i \|f_i\|_2$ for each $i$.
Then, we always have $N=1$ in Corollary \ref{CHCO}, and there exists a compact metric space $Z$ such that $X=S^0\ast Z$.
\end{Cor}
\begin{proof}
If $N=1$ holds in Corollary \ref{CHCO}, we get the corollary.

In the following, we suppose that $N\geq 2$ holds in Corollary \ref{CHCO} and show a contradiction.
\begin{Clm}\label{imp1}
Take arbitrary $i=0,\ldots,N$.
Take $j_1,\ldots,j_l=0,\ldots, N$ such that $j_a\neq j_b$ (if $a\neq b$) and
\begin{equation*}
\{x_{j_1},\ldots,x_{j_l}\}=\{x_k:d(x_i,x_k)=\pi\}.
\end{equation*}
If $l\geq 2$, then we have $l=2$ and $\Card Z_{i j_1}= \Card Z_{i j_2}=1$.
\end{Clm}
\begin{proof}[Proof of Claim \ref{imp1}]
The tangent cone $X_{x_i}$ at $x_i$ is isometric to $$(C(Z_{i j_1}) \coprod \cdots \coprod C(Z_{i j_l}))/\sim,$$ where $C(Z_{ij_1}),\ldots, C(Z_{i j_l})$ denote the metric cones, and $\sim$ identify the vertexes.
Let $0^\ast$ denotes the vertex.
The metric $d$ on $(C(Z_{i j_1}) \coprod \cdots \coprod C(Z_{i j_l}))/\sim$ satisfies $d(c_1,c_2)=d(0^\ast,c_1)+d(0^\ast,c_2)$ for all $c_1\in C(Z_{i j_a})$ and $c_2\in C(Z_{i j_b})$ ($a\neq b$).

Suppose that we have either $l\geq 3$ or $\Card Z_{i j_a}\geq 2$ for some $a$.
Then, there exist $a\in \{1,\ldots,l\}$ and points $x_1\in Z_{i j_1}$, $x_2\in Z_{i j_2}$, $x_3\in Z_{i j_a}$
such that $x_3\notin \{x_1,x_2\}$.
If necessary, we can exchange $j_1$ and $j_2$ and assume $a\neq 1$.
Consider a line $\gamma\colon \mathbb{R}\to X_{x_i}$ defined by
\begin{empheq}[left={\gamma(t)=\empheqlbrace}]{align*}
[-t, x_1&]\in C(Z_{i j_1}) &&\quad (t<0),\\
[t, x_2&]\in C(Z_{i j_2}) && \quad(t>0),\\
&0^\ast && \quad (t=0).
\end{empheq}
By the splitting theorem \cite[Theorem 9.27]{Ch}, $X_{x_i}$ splits into $\mathbb{R}\times Z$ for some geodesic space $Z$, and there exists a point $y_0\in Z$ such that $\gamma(t)$
corresponds to $(t,y_0)$ for each $t\in\mathbb{R}$.
In particular, under this isometry, we have
\begin{empheq}[left={\empheqlbrace}]{align*}
[1,x_1]\in C(Z_{i j_1}) &\mapsto (-1,y_0),\\
[1, x_2]\in C(Z_{i j_2}) &\mapsto (1,y_0),\\
0^\ast &\mapsto (0,y_0).
\end{empheq}
Take $(s,y)\in\mathbb{R}\times Z$ corresponding to $[1,x_3]$.
Then, we have
\begin{align*}
1=&d(0^\ast,[1,x_3])^2=s^2+d(y,y_0)^2,\\
4=&d([1,x_1],[1,x_3])^2=(s+1)^2+d(y,y_0)^2.
\end{align*}
Thus, we get $s=1$ and $d(y,y_0)=0$.
This contradicts to $x_2\neq x_3$.
Therefore, we get $l=2$ and $\Card Z_{i j_1}= \Card Z_{i j_2}$=1.
\end{proof}
By the connectedness, there exists $i$ such that $\Card\{x_k:d(x_i,x_k)=\pi\}=2$, and we can connect $x_i$ to $x_j$ for any $j$, i.e., there exists $x_{j_0},\ldots, x_{j_l}$ such that
$j_0=i$, $j_l=j$ and $d(x_{j_a},x_{j_{a+1}})=\pi$ for all $a=0,\ldots,l-1$.
Therefore, we get $X=[0,N\pi]$ or $X=S^1((N+1)/2)$ (c.f. \cite[Theorem 1.1]{Chen} and \cite[Theorem 1.1]{KL}).
Moreover, if $X=S^1((N+1)/2)$, then $N$ is odd by the following argument.
Suppose that $N$ is even. Then, $(N+1)/2$ is not a integer, and there exists $x,x_i,x_j\in X$ with $d(x_i,x_j)=\pi$, $d(x_0,x)=(N+1)\pi/2$ and $x\in B_{i j}$.
We get $d(x_0,x_i)=d(x_0,x_j)=N\pi/2$, and this contradicts to Corollary \ref{CHCO} (i).
Thus, we have that $N$ is odd.

For each $i\in\mathbb{N}$, take $p_i\in M_i$ of Proposition \ref{p2i} for $f_i$.
By taking a subsequence, we can assume that there exists a point $p\in X$ such that $p_i\stackrel{GH}{\to} p$, and $X$ has a limit measure $\nu$.
By \cite[Theorem 1.1]{KL} (see also \cite[Definition 2.1]{KL}), the limit measure $\nu$ on $X$ is of the form
$$\nu=\phi d H^1,$$
where $H^1$ denotes the $1$-dimensional Hausdorff measure and $\phi\colon X\to \mathbb{R}_{\geq 0}\cup \{\infty\}$ satisfies
\begin{equation}
\begin{split}\label{impe}
\phi^{\frac{1}{n}}(\gamma_{x,y}(t))
\geq \frac{1}{\sinh \left(d(x,y)\sqrt{\frac{K}{n}}\right)}
&\Biggr(\sinh \left(d(\gamma_{x,y}(t),y)\sqrt{\frac{K}{n}}\right)\phi^{\frac{1}{n}}(x)\\
&\qquad+\sinh\left(d(x,\gamma_{x,y}(t))\sqrt{\frac{K}{n}}\right)\phi^{\frac{1}{n}}(y)\Biggr)
\end{split}
\end{equation}
for all $x,y\in X$ and $0\leq t\leq d(x,y)$.
In particular, $\phi$ is a locally Lipschitz function on $(0,N \pi)$ if $X=[0,N\pi]$, and $\phi$ is a Lipschitz function if $X=S^1(m)$.
Set $\widetilde{f}_i:=T_{(n,\delta)}(f_i)\in C^\infty(M_i)$ and normalize it so that $\|\widetilde{f}_i\|_2^2=\frac{1}{n+1}$.
Taking a subsequence, we can assume that there exists $f\in\mathcal{D}(X,\nu)$ such that $\widetilde{f}_i$ converges to $f$ strongly in $L^2$, and $\Delta \widetilde{f}_i$ converges to $\Delta f$ weakly in $L^2$ by Theorem 1.3 and Theorem 4.9 of \cite{Ho1}.
Moreover, $\cos d_i(p_i,\cdot)$ converges to $\cos d(p,\cdot)$ strongly in $L^2$ because we can easily verify the property of \cite[Definition 3.7]{Ho1}(see also \cite[Proposition 3.32]{Ho1}). 
Thus, we get
$$
\|f-\cos d(p,\cdot)\|_2\leq \liminf_{i\to\infty}\|\widetilde{f}_i-\cos d(p_i,\cdot)\|_2=0,
$$
and so $f=\cos d(p,\cdot)$.
Since $\Delta \widetilde{f}_i-n \widetilde{f}_i$ converges to $\Delta f-n f$ weakly in $L^2$,
we get
$$
\|\Delta f-n f\|_2\leq \liminf_{i\to \infty}\|\Delta \widetilde{f}_i-n \widetilde{f}_i\|_2=0,
$$
and so $\Delta f=n f$.

If $X=[0,N\pi]$, we parameterize $(0,N\pi)\subset [0,N\pi]$ by the identity map $(0,N\pi)\to (0,N\pi),\, \theta\mapsto \theta$.
Since $0\in [0,N\pi]$ corresponds to some $x_i$ of Corollary \ref{CHCO}, we have $f(\theta)=\pm\sin \theta$ under the parameterization.
Considering $-f$, if necessary, we can assume that $f(\theta)=\sin\theta$.
If $X=S^1(m)$ and $p=m \exp(\sqrt{-1}\theta_0)$, we parameterize $S^1(m)\setminus \{p\}\subset S^1(m)$ by the map $$(0,2m\pi)\to S^1(m)\setminus \{p\},\,\theta\mapsto m \exp(\sqrt{-1}(\theta_0+\theta/m)).$$
For both cases, we have $f(\theta)=\sin \theta$ under our parameterization.

Take arbitrary $\psi\in C^\infty_0((0,N\pi))$ if $X=[0,N\pi]$ and $\psi\in C^\infty_0(S^1(m)\setminus\{p\})$ if $X=S^1(m)$, where $C^\infty_0$ denotes the space of smooth functions with compact support.
Then, we have
\begin{align*}
\int_X \langle d f, d\psi \rangle\,d \nu
=&-\int \sin \theta \psi' (\theta) \phi (\theta)\,d\theta\\
=&-\int (\sin \theta \psi (\theta) \phi (\theta))'\,d\theta
+\int (\sin \theta  \phi (\theta))'\psi(\theta)\,d\theta\\
=&\int (\cos \theta  \phi (\theta)+ \sin\theta \phi'(\theta))\psi(\theta)\,d\theta
\end{align*}
By the definition of the Laplacian $\Delta$, we have
\begin{align*}
\int_X \langle d f, d\psi \rangle\,d \nu
=\int_X \Delta f \psi\,d\nu
=\int n \cos \theta \phi (\theta)\psi (\theta) \,d\theta.
\end{align*}
Thus, we get
$$
\sin \theta\phi' (\theta)=(n-1)\cos \theta \phi(\theta)
$$
for almost all $\theta$.
This gives
$$
\sin \theta\left(\phi (\theta)-\phi\left(\frac{\pi}{2}\right)\sin^{n-1}\theta\right)'=(n-1)\cos \theta \left(\phi (\theta)-\phi\left(\frac{\pi}{2}\right)\sin^{n-1}\theta\right)
$$
for almost all $\theta\in (0,\pi)$.
Since $\sin \theta>0$ for any $\theta \in(0,\pi)$, we get
$\phi (\theta)=\phi\left(\frac{\pi}{2}\right)\sin^{n-1}\theta$ on $(0,\pi)$.
Similarly, we have
$\phi (\theta)=\phi\left(\frac{3\pi}{2}\right)\sin^{n-1}(\theta-\pi)$ on $(\pi,2\pi)$.
By the Lipschitz continuity of $\phi$, we get $\phi(\pi)=0$.
Putting $x=\pi/2$, $y=3\pi/2$ and $t=\pi$ into (\ref{impe}), we get
$$
\phi(\pi)\geq  \frac{1}{\sinh \left(\pi\sqrt{\frac{K}{n}}\right)}
\Biggr(\sinh \left(\frac{\pi}{2} \sqrt{\frac{K}{n}}\right)\phi^{\frac{1}{n}}\left(\frac{\pi}{2}\right)
+\sinh \left(\frac{\pi}{2} \sqrt{\frac{K}{n}}\right)\phi^{\frac{1}{n}}\left(\frac{3\pi}{2}\right)\Biggr)>0.
$$
This is a contradiction. Thus, we get $N=1$.
\end{proof}

By Corollary \ref{CHCO2}, we always have $N=1$ in Proposition \ref{p2k} if we take $\delta$ small enough.
\begin{Prop}\label{p2ki}
Given an integer $n\geq 2$ and positive real numbers $K>0$ and $D>0$, there exists a positive constant $\eta(n,K,D)>0$ such that the following property hold.
Take a positive real number $0<\delta\leq \eta$.
Let $(M,g)$ be an $n$-dimensional closed Riemannian manifold with $\Ric\geq -Kg$ and $\diam(M)\leq D$.
Suppose that a non-zero function $f\in C^\infty(M)$ satisfies $\|\nabla^2 f+f g\|_2\leq \delta\|f\|_2$.
Take a point $p\in M$ of Proposition \ref{p2i}.
Then, there exists a point $q\in M$ such that
$|d(p,q)-\pi|\leq \delta^{\frac{1}{200n^2}}$ and
$d(p,x)+d(x,q)\leq d(p,q)+\delta^\frac{1}{250n^2}$ holds for all $x\in M$.
\end{Prop}

To prove Theorem \ref{apa1}, let us turn to the consideration of the $\delta$-pinching condition for a subspace $V\subset C^\infty(M)$.
For the standard sphere $S^n\subset \mathbb{R}^{n+1}$, the height functions are the first eigenfunction, and if $f_1$ and $f_2$ are height functions that are orthogonal to each other in $L^2$ sense,
then the distance between the maximum points of these functions is equal to $\pi/2$.
The following lemma asserts that such a property almost holds under our $\delta$-pinching condition.
\begin{Lem}\label{p31d}
Given an integer $n\geq 2$ and positive real numbers $K>0$ and $D>0$, there exists a positive constant $\eta(n,K,D)>0$ such that the following properties hold.
Take a positive real number $0<\delta\leq \eta$.
Let $(M,g)$ be an $n$-dimensional closed Riemannian manifold with $\Ric\geq -Kg$ and $\diam(M)\leq D$.
Suppose that a 2-dimensional subspace $V\subset C^\infty(M)$ satisfies the $\delta$-pinching condition.
Take $f_1,f_2\in T_{(n,\delta)}(V)$ such that
\begin{equation*}
\|f_1\|_2^2=\|f_2\|_2^2=\frac{1}{n+1},\quad \int_M f_1 f_2 \,d\mu_g=0.
\end{equation*}
For each $f_i$ $(i=1,2)$, we use the notation $p_i,q_i$ of Proposition \ref{p2ki}.
Then, we have
\begin{equation*}
\left|d(p_1,p_2)-\frac{\pi}{2}\right|\leq C\delta^\frac{1}{96n},\quad
\left|d(p_1,q_2)-\frac{\pi}{2}\right|\leq C\delta^\frac{1}{250n^2}.
\end{equation*}
\end{Lem}
\begin{proof}
Set $f_0:=f_1+f_2$ and use the notation $p_0$ of Proposition \ref{p2i} for $f_0$.
Note that $\|f_0\|_2^2=\frac{2}{n+1}$.

By Proposition \ref{p2i}, we have
\begin{equation}\label{31b}
\begin{split}
|f_0(x)-\sqrt{2}\cos d(p_0,x)|\leq &C \delta^\frac{1}{48n},\\
|f_1(x)-\cos d(p_1,x)|\leq &C \delta^\frac{1}{48n},\\
|f_2(x)-\cos d(p_2,x)|\leq& C\delta^\frac{1}{48n}
\end{split}
\end{equation}
for all $x\in M$.
Combining (\ref{31b}) with $f_0=f_1+f_2$, we get
\begin{equation}\label{31da}
\left|\sqrt{2}\cos d(p_0,x)-\cos d(p_1,x)-\cos d(p_2,x)\right|\leq C \delta^\frac{1}{48n}
\end{equation}
for all $x\in M$.
Putting $x=p_0, p_1,p_2$ into (\ref{31da}), we get
\begin{equation*}
\begin{split}
\left|\sqrt{2}-\cos d(p_0,p_1)-\cos d(p_0,p_2)\right|\leq& C \delta^\frac{1}{48n},\\
\left|\sqrt{2}\cos d(p_0,p_1)-1-\cos d(p_1,p_2)\right|\leq&C \delta^\frac{1}{48n},\\
\left|\sqrt{2}\cos d(p_0,p_2)-1-\cos d(p_1,p_2)\right|\leq&C \delta^\frac{1}{48n}.
\end{split}
\end{equation*}
Thus, we get
\begin{align*}
\left|\cos d(p_1,p_2)\right|\leq&C\delta^\frac{1}{48n}.
\end{align*}
Since we have $d(p_1,p_2)\leq \pi+2\delta^{\frac{1}{250n^2}}$ by Proposition \ref{p4a},
we get
\begin{equation*}
\left|d(p_1,p_2)-\frac{\pi}{2}\right|\leq C\delta^\frac{1}{96n}.
\end{equation*}
by Lemma \ref{p2j}.
Thus, we get
\begin{equation*}
\left|d(p_1,q_2)-\frac{\pi}{2}\right|\leq C\delta^\frac{1}{250n^2}
\end{equation*}
by Proposition \ref{p2ki}.
\end{proof}

Now, we are in position to prove Theorem \ref{apa1}.
\begin{proof}[Proof of Theorem \ref{apa1}]
Let $\{(M_i,g_i)\}_{i\in \mathbb{N}}$ be a sequence of $n$-dimensional closed Riemannian manifolds such that $\Ric_{g_i}\geq -K g_i$, $\diam(M_i)\leq D$, $\lim_{i\to \infty}\lambda_k(\bar{\Delta}^E,M_i)= 0$, and $\{(M_i,g_i)\}_{i\in \mathbb{N}}$ converges to a geodesic space $X$.

By Corollary \ref{CHCO2} and Lemma \ref{p31d}, we have $\diam(X)=\pi$, and there exist pairs of points $(p_1,q_1),\ldots,(p_k,q_k)$ such that the following properties hold:
\begin{itemize}
\item[(i)] $d(p_i,q_i)=\pi$ for all $i$.
\item[(ii)] $d(p_i,p_j)=d(p_i,q_j)=\frac{\pi}{2}$ if $i\neq j$.
\item[(iii)] Define $Z_i:=\{x\in X:d(x,p_i)=d(x,q_i)\}$. For all $i$, we define $\phi_i \colon X\to Z_i$ as follows.
If $x\in X$ satisfies $d(x,p_i)\leq d(x,q_i)$, we define $\phi_i(x)$ to be $\phi_i(x)\in \Imag \gamma_{x,q_i}\cap Z_{i}$.
If $d(x,p_i)> d(x,q_i)$, we define $\phi_{i}(x)$ to be $\phi_{i}(x)\in \Imag \gamma_{x,p_i}\cap Z_{i}$.
Then, we have 
\begin{equation}\label{apec6}
\cos d(x,y)=\cos d(p_i,x)\cos d(p_i,y)+\sin d(p_i,x)\sin d(p_i,y)\cos d(\phi_i(x),\phi_i(y))
\end{equation}
 for all $x,y\in X$.
\end{itemize}
Note that for all $i$, we have an isometry $X\cong S^0\ast Z_i$ such that $p_i\mapsto 0^\ast$ and $q_i\mapsto \pi^\ast$.
Under this identification, we have $\phi_i([t,z])=[\pi/2,z]$ for all $[t,z]\in S^0\ast Z_i$ with $0<t<\pi$.
\begin{Clm}\label{clap6}
Take arbitrary $i$ and $x,y\in Z_i$ with $d(x,y)<\pi$.
Then, we have $\gamma_{x,y}(s)\in Z_i$ for all $s\in [0,d(x,y)]$.
\end{Clm}
\begin{proof}[Proof of Claim \ref{clap6}]
We have
\begin{equation*}
\begin{split}
\cos d(x,\gamma_{x,y}(s))=&\sin d(p_i,\gamma_{x,y}(s))\cos d(x,\phi_i(\gamma_{x,y}(s))),\\
\cos d(y,\gamma_{x,y}(s))=&\sin d(p_i,\gamma_{x,y}(s))\cos d(y,\phi_i(\gamma_{x,y}(s))).
\end{split}
\end{equation*}
We first show that $\gamma_{x,y}(t)\in Z_i$ for all $t\in[0,d(x,y)]$ with $d(x,\gamma_{x,y}(t))<\frac{\pi}{2}$ and $d(y,\gamma_{x,y}(t))<\frac{\pi}{2}$.
Take arbitrary $t\in[0,d(x,y)]$ with $d(x,\gamma_{x,y}(t))<\frac{\pi}{2}$ and $d(y,\gamma_{x,y}(t))<\frac{\pi}{2}$.
Suppose that $\gamma_{x,y}(t)\notin Z_i$.
Then, we have
\begin{equation*}
\begin{split}
0<\cos d(x,\gamma_{x,y}(t))<&\cos d(x,\phi_i(\gamma_{x,y}(t))),\\
0<\cos d(y,\gamma_{x,y}(t))<&\cos d(y,\phi_i(\gamma_{x,y}(t))),
\end{split}
\end{equation*}
and so
\begin{equation*}
\begin{split}
d(x,\gamma_{x,y}(t))>& d(x,\phi_i(\gamma_{x,y}(t))),\\
d(y,\gamma_{x,y}(t))>&d(y,\phi_i(\gamma_{x,y}(t))).
\end{split}
\end{equation*}
Thus, we get
$d(x,y)=d(x,\gamma_{x,y}(t))+d(y,\gamma_{x,y}(t))> d(x,\phi_i(\gamma_{x,y}(t)))+ d(y,\phi_i(\gamma_{x,y}(t)))$.
This is a contradiction.
Therefore, we get $\gamma_{x,y}(t)\in Z_i$.

We next show that $\gamma_{x,y}(t)\in Z_i$ for the general case.
Put $t_0:=\frac{1}{2}d(x,y)$.
Then $d(x,\gamma_{x,y}(t_0))<\frac{\pi}{2}$ and $d(y,\gamma_{x,y}(t_0))<\frac{\pi}{2}$.
Thus, we get $\gamma_{x,y}(t_0)\in Z_i$.
For all $t\in[0,t_0]$, we have $d(x,\gamma_{x,y}(t))<\frac{\pi}{2}$ and $d(\gamma_{x,y}(t_0),\gamma_{x,y}(t))<\frac{\pi}{2}$, and so we get $\gamma_{x,y}(t)\in Z_i$.
Similarly, we get $\gamma_{x,y}(t)\in Z_i$ for all $t\in[t_0,d(x,y)]$.
\end{proof}

We have shown that $X=S^0\ast Z_1$.
To carry out the iteration process, we investigate the structure of $Z_1\cap \cdots \cap Z_{i-1}$.
\begin{Clm}\label{clap7}
Suppose that $k\geq 2$, $2\leq i\leq k$ and
\begin{equation*}
Z_1\cap \cdots \cap Z_{i-1}\setminus\{p_{i},q_{i}\}\neq \emptyset.
\end{equation*}
Then, we have $Z_{1}\cap \cdots \cap Z_{i}\neq \emptyset$ and
\begin{equation*}
Z_1\cap \cdots \cap Z_{i-1}=S^0\ast(Z_{1}\cap \cdots \cap Z_{i}).
\end{equation*}
\end{Clm}
\begin{proof}[Proof of Claim \ref{clap7}]
For all $x\in Z_1\cap \cdots \cap Z_{i-1}\setminus\{p_{i},q_{i}\}$, we have $d(x,p_{i})<\pi$ and $d(x,q_{i})<\pi$, and so we get $\phi_i(x)\in Z_{1}\cap \cdots \cap Z_{i}$ by Claim \ref{clap6}.
Thus, we have $Z_{1}\cap \cdots \cap Z_{i}\neq \emptyset$.
Since we have (\ref{apec6})
for all $x,y \in Z_1\cap \cdots \cap Z_{i-1}\setminus\{p_{i},q_{i}\}$,
the map $$Z_1\cap \cdots \cap Z_{i-1}\to S^0\ast(Z_{1}\cap \cdots \cap Z_{i}),\,x\mapsto [d(p_i,x),\phi_i(x)]$$
gives the isomorphism.
Note that the surjectivity of this map also follows from Claim \ref{clap6}.
\end{proof}
For all $2\leq i<k$, we have $p_{k}\in Z_1\cap \cdots \cap Z_{i-1}\setminus\{p_i,q_i\}$,
and so $Z_1\cap \cdots \cap Z_{i-1}\setminus\{p_i,q_i\}\neq \emptyset$.
Since $Z_1\cap \cdots \cap Z_{k}\subset Z_1\cap \cdots \cap Z_{k-1}\setminus\{p_k,q_k\}$,
we get
\begin{align*}
Z_1\cap \cdots \cap Z_{k}= \emptyset &\iff Z_1\cap \cdots \cap Z_{k-1}\setminus\{p_k,q_k\}=\emptyset\\
&\iff Z_1\cap \cdots \cap Z_{k-1}=\{p_k,q_k\}
\end{align*}
by Claim \ref{clap7}.
Therefore, we get the following inductively by Claim \ref{clap7}:
\begin{empheq}[left={X=\empheqlbrace}]{align*}
&\qquad S^{k-1} && (Z_1\cap \cdots \cap Z_{k}= \emptyset),\\
&S^{k-1}\ast(Z_1\cap \cdots \cap Z_{k})&&(Z_1\cap \cdots \cap Z_{k}\neq \emptyset).
\end{empheq}
Finally, we investigate the structure of $Z_1\cap \cdots \cap Z_{k}$ when $Z_1\cap \cdots \cap Z_{k}\neq \emptyset$.
\begin{Clm}\label{clap8}
Suppose that $Z_1\cap \cdots \cap Z_{k}\neq \emptyset$.
Then, the tangent cone $X_{p_k}$ of $X$ at $p_k$ is isometric to $\mathbb{R}^{k-1}\times C(Z_1\cap \cdots \cap Z_{k})$.
Moreover, if $Z_1\cap \cdots \cap Z_{k}$ is not connected, then we have $\Card(Z_1\cap \cdots \cap Z_{k})=2$.
\end{Clm}
\begin{proof}[Proof of Claim \ref{clap8}]
For each $i=1,\ldots, k$, we define a map 
$\psi_i\colon X\to Z_1\cap \cdots \cap Z_{i}$ by $\psi_i:=\phi_i\circ\cdots \circ \phi_1$. Define $\psi_0:=\Id_X\colon X\to X$.
For each $r>0$,
we define $\Psi_r \colon X\to \mathbb{R}^k\times C(Z_1\cap \cdots \cap Z_{k})$ by
\begin{equation*}
\begin{split}
&\Psi_r(x)\\
:=&\left(\frac{1}{r}\left(d(p_1,\psi_0(x))-\frac{\pi}{2}\right),\ldots,\frac{1}{r}\left(d(p_{k-1},\psi_{k-2}(x))-\frac{\pi}{2}\right), \left[\frac{1}{r}d(p_k,\psi_{k-1}(x)), \psi_{k}(x)\right]\right).
\end{split}
\end{equation*}
Note that we have $\Psi_r(p_k)=(0,\ldots,0,0^\ast)$.
In the following, we show that $(X,p_k,r^{-1}d)$ converges to $(\mathbb{R}^k\times C(Z_1\cap \cdots \cap Z_{k}),(0,\ldots,0,0^\ast))$ in the pointed Gromov-Hausdorff sense through the map $\Psi_r$.

Fix $R>0$.
Let us show that
$$
\Psi_r\colon \left(B_{rR}(p_k),r^{-1}d \right)\to B_{R+r}\left((0,\ldots,0,0^\ast)\right)\subset \mathbb{R}^k\times C(Z_1\cap \cdots \cap Z_{k})
$$
is a Hausdorff approximation map for sufficiently small $r>0$.
Here, $B_{rR}(p_k)$ denotes the open metric ball for the original metric on $(X,d)$.
Note that if this was proved, we can easily modify the map $\Psi_r$ to a Hausdorff approximation map $\left(B_{rR}(p_k),r^{-1} d\right)\to B_{R}\left((0,\ldots,0,0^\ast)\right)$.

For all $x,y\in B_{r R}(p_k)$, we have
\begin{equation*}
d(\phi_1(x),\phi_1(y))\leq d(x,\phi_1(x))+d(x,y)+d(y,\phi_1(y))\leq 4r R.
\end{equation*}
Since $\psi_i(p_k)=p_k$ for $i=0,\ldots,k-1$, we have
\begin{equation*}
d(\psi_i(x),\psi_i(y))\leq  C(R,k)r
\end{equation*}
for all $x,y\in B_{r R}(p_k)$ and $i=0,\ldots,k-1$ inductively.
In particular, $\psi_i(x),\psi_i(y)\in B_{C(R,k)r}(p_k)$.
For all $a>0$, we use the following notation:
\begin{equation*}
f=O(r^a):\iff |f|\leq C(R,k)r^a.
\end{equation*}
For all $1\leq i \leq k-1$ and $x,y\in B_{r R}(p_k)$, we have
\begin{equation}\label{apec7}
\begin{split}
&\cos d(\psi_{i-1} (x),\psi_{i-1}(y))\\
=&\cos d(p_{i},\psi_{i-1} (x))\cos d(p_i,\psi_{i-1} (y))\\
&\qquad \qquad\qquad \qquad+\sin d(p_i,\psi_{i-1} (x)) \sin d(p_i,\psi_{i-1} (y))\cos d(\psi_{i} (x),\psi_{i} (x))\\
=&\left(\frac{\pi}{2}-d(p_i, \psi_{i-1}(x))+O(r^3)\right)\left(\frac{\pi}{2}-d(p_i, \psi_{i-1}(y))+O(r^3)\right)\\
&+\left(1-\frac{1}{2}\left(d(p_i, \psi_{i-1}(x))-\frac{\pi}{2}\right)^2+O(r^4)\right) \times\\
&\,\left(1-\frac{1}{2}\left(d(p_i, \psi_{i-1}(y))-\frac{\pi}{2}\right)^2+O(r^4)\right)\left(1-\frac{1}{2}d(\psi_i(x), \psi_i(y))^2+O(r^4)\right)\\
=&1-\frac{1}{2}\big(d(p_i,\psi_{i-1}(x))-d(p_i,\psi_{i-1}(y))\big)^2-\frac{1}{2}d(\psi_i(x), \psi_i(y))^2+O(r^4)
\end{split}
\end{equation}
by (\ref{apec6}), $\frac{\pi}{2}-d(p_i, \psi_{i-1}(x))=O(r)$ and $\frac{\pi}{2}-d(p_i, \psi_{i-1}(y))=O(r)$.
By (\ref{apec7}) and
\begin{equation*}
\cos d(\psi_{i-1} (x),\psi_{i-1}(y))=1-\frac{1}{2}d(\psi_{i-1} (x),\psi_{i-1}(y))^2+O(r^4),
\end{equation*}
we get
\begin{equation*}
d(\psi_{i-1} (x),\psi_{i-1}(y))^2=\big(d(p_i,\psi_{i-1}(x))-d(p_i,\psi_{i-1}(y))\big)^2+d(\psi_i(x), \psi_i(y))^2+O(r^4).
\end{equation*}
By induction, we get
\begin{equation}\label{apec8}
\begin{split}
d(x,y)^2
=\sum_{i=1}^{k-1}\Big(d(p_i,&\psi_{i-1}(x))-d(p_i,\psi_{i-1}(y))\Big)^2\\
&+d(\psi_{k-1}(x), \psi_{k-1}(y))^2+O(r^4).
\end{split}
\end{equation}
We have
\begin{equation}\label{apec9}
\begin{split}
&\cos d(\psi_{k-1}(x), \psi_{k-1}(y))\\
=&\cos d(p_{k},\psi_{k-1} (x))\cos d(p_k,\psi_{k-1} (y))\\
&\qquad \qquad\qquad +\sin d(p_k,\psi_{k-1} (x)) \sin d(p_k,\psi_{k-1} (y))\cos d(\psi_{k} (x),\psi_{k} (y))\\
=&\left(1-\frac{1}{2}d(p_k, \psi_{k-1}(x))^2+O(r^4)\right)\left(1-\frac{1}{2}d(p_k, \psi_{k-1}(y))^2+O(r^4)\right)\\
&+\left(d(p_k, \psi_{k-1}(x))+O(r^3)\right)\left(d(p_k, \psi_{k-1}(y))+O(r^3)\right)
\cos d(\psi_{k}(x), \psi_{k}(y))\\
=& 1-\frac{1}{2}d(p_k, \psi_{k-1}(x))^2-\frac{1}{2}d(p_k, \psi_{k-1}(y))^2\\
&+d(p_k, \psi_{k-1}(x)) d(p_k, \psi_{k-1}(y))\cos d(\psi_{k}(x), \psi_{k}(y))+O(r^4).
\end{split}
\end{equation}
By (\ref{apec9}) and
\begin{equation*}
\cos d(\psi_{k-1} (x),\psi_{k-1}(y))=1-\frac{1}{2}d(\psi_{k-1} (x),\psi_{k-1}(y))^2+O(r^4),
\end{equation*}
we get
\begin{equation}\label{apec10}
\begin{split}
&d(\psi_{k-1} (x),\psi_{k-1}(y))^2\\
=&d(p_k, \psi_{k-1}(x))^2+d(p_k, \psi_{k-1}(y))^2\\
&\qquad-2d(p_k, \psi_{k-1}(x)) d(p_k, \psi_{k-1}(y))\cos d(\psi_{k}(x), \psi_{k}(y))+O(r^4).
\end{split}
\end{equation}
By (\ref{apec8}) and (\ref{apec10}), we get
\begin{equation}\label{apec11}
\left(\frac{1}{r}d(x,y)\right)^2
=d(\Psi_r(x),\Psi_r(y))^2+O(r^2).
\end{equation}
In particular, we have $\Psi_r(x),\Psi_r(y)\in B_{R+r}\left((0,\ldots,0,0^\ast)\right)\subset \mathbb{R}^k\times C(Z_1\cap \cdots \cap Z_{k})$ for sufficient small $r>0$.

For all $(t_1,\ldots,t_{k-1},[t_k,z])\in \mathbb{R}^{k-1}\times C(Z_1\cap \cdots \cap Z_{k})$, we have
\begin{equation*}\label{apec12}
\Psi_r\left(\left[r t_1+\frac{\pi}{2},\ldots,rt_{k-1}+\frac{\pi}{2},r t_k,z\right]\right)=(t_1,\ldots,t_{k-1},[t_k,z]).
\end{equation*}
If 
$$(t_1,\ldots,t_{k-1},[t_k,z])\in B_{R-r}\left((0,\ldots,0,0^\ast)\right)\subset\mathbb{R}^{k-1}\times C(Z_1\cap \cdots \cap Z_{k}),$$
then $\left[r t_1+\frac{\pi}{2},\ldots,rt_{k-1}+\frac{\pi}{2},r t_k,z\right]\in B_{C(R,k)r}(p_k),$ and so $$\left[r t_1+\frac{\pi}{2},\ldots,rt_{k-1}+\frac{\pi}{2},r t_k,z\right]\in B_{rR}(p_k)$$ by (\ref{apec11}) if $r>0$ is small enough. This shows that the map $\Psi_r\colon \left(B_{rR}(p_k),r^{-1}d \right)\to B_{R+r}\left((0,\ldots,0,0^\ast)\right)$ is $2r$-dense.
Combining this and (\ref{apec11}), we have that $\Psi_r$ is a $2r$-Hausdorff approximation map, and so we get the first assertion.

We next show the second assertion.
The proof is similar to Corollary \ref{CHCO2}.
Suppose that there exist more than two different connected components of $Z_1\cap \cdots \cap Z_{k}$ (let $A$ be one of the connected components), and there exist two points $b,c\in Z_1\cap \cdots \cap Z_{k}\setminus A$ ($b\neq c$).
Take arbitrary $a\in A$.
Then,
$\gamma\colon (-\infty,\infty)\to C(Z_1\cap \cdots \cap Z_{k})$ defined by
$\gamma(t)=[-t,a]$ for $t\leq 0$, and $\gamma(t)=[t,b]$ for $t>0$, is a line in $C(Z_1\cap \cdots \cap Z_{k})$.
Thus, $(0,\gamma)$ in $X_{p_k}=\mathbb{R}^{k-1}\times C(Z_1\cap \cdots \cap Z_{k})$ is a line.
By the splitting theorem \cite[Theorem 9.27]{Ch},
there exists a geodesic space $Y$, a point $y_0\in Y$ and an isometry $X_{p_k}\to \mathbb{R}\times Y$ such that $(0,\gamma(t))\in X_{p_k}$ corresponds to $(t,y_0)\in\mathbb{R}\times Y$ for all $t\in \mathbb{R}$.
Take a point $(s,y)\in \mathbb{R}\times Y$ that corresponds to $(0,[1,c])\in X_{p_k}$.
Since $d((0,0^\ast), (0,[1,c]))=1$, we have $s^2+d(y,y_0)^2=1$.
Since $d((0,[1,a]), (0,[1,c]))=2$, and $(0,[1,a])$ corresponds to $(-1,y_0)$, we have $(s+1)^2+d(y,y_0)^2=4$.
Thus, we get that $s=1$ and $d(y,y_0)=0$.
However, $(0,[1,b])$ corresponds to and $(1,y_0)$.
Since we assumed that $c\neq b$, this is a contradiction.
Thus, we get the second assertion.
\end{proof}
If $\Card (Z_1\cap \cdots \cap Z_{k})=2$, then $\diam(Z_1\cap \cdots \cap Z_{k})=\pi$ by Claim \ref{clap6}, and so we get $X=S^{k-1}\ast \{0,\pi \}=S^k$.

If  $Z_1\cap \cdots \cap Z_{k}$ is connected,
we define a metric $d_L$ on $Z_1\cap \cdots \cap Z_{k}$ by
\begin{equation*}
\begin{split}
d_L(x,y):=\inf \Big\{\sum_{j=1}d(&x_{j-1},x_j): N\in\mathbb{Z}_{>0},\, x_j\in Z_1\cap \cdots \cap Z_{k} \text{ for all $j=0,\ldots,N$,}\\
& x_0=x,\, x_N=y \text{ and } d(x_{j-1},x_j)<\pi \text{ for all $j=1,\ldots,N$}\Big\}.
\end{split}
\end{equation*}
Then, $(Z_1\cap \cdots \cap Z_{k},d_L)$ is a geodesic space, and $X=S^{k-1}\ast(Z_1\cap \cdots \cap Z_{k},d_L)$ holds.

By the Gromov's pre-compactness theorem and the above argument, we get the theorem.
\end{proof}
\begin{proof}[Another proof of Main Theorem 1]
Let $\{(M_i,g_i)\}_{i\in \mathbb{N}}$ be a sequence of $n$-dimensional closed Riemannian manifolds such that $\Ric_{g_i}\geq -K g_i$, $\diam(M_i)\leq D$ for all $i$ and $\lim_{i\to \infty}\lambda_n(\bar{\Delta}^E,M_i)= 0$.
For each $i$, define an orientable $n$-dimensional closed Riemannian manifold $(N_i,\tilde{g}_i)$ to be the orientable Riemannian covering of $(M_{i},g_{i})$ with two sheets if $M_i$ is not orientable, and
 $(N_i,\tilde{g}_i)=(M_i,g_i)$ if $M_i$ is orientable.
Then, we have $\lim_{i\to \infty}\lambda_{n+1}(\bar{\Delta}^E,N_i)= 0$ by Corollary \ref{p42c}.
Take a subsequence $i(j)$ such that $(M_{i(j)},g_{i(j)})$, $(N_{i(j)},\tilde{g}_{i(j)})$ converges to some geodesic spaces $X,Y$ in Gromov-Hausdorff topology, respectively.
Since $\dim Y\leq n$ holds, we have $Y=S^n$ by Theorem \ref{apa1} and \cite[Proposition 5.6]{Ho}, where $\dim$ denotes the Hausdorff dimension.
In particular, we get $\{(M_{i(j)},g_{i(j)})\}$ is a non-collapsing sequence.
Thus, we get $X=S^n$ or $X=S^{n-1}\ast Z$ for some geodesic space $Z$ by Theorem \ref{apa1}.
By \cite[Proposition 5.6]{Ho}, we get that $\Card Z=1$ if $X=S^{n-1}\ast Z$.
However, $S^{n-1}\ast \{\text{point}\}=S^n_{+}$ ($n$-dimensional hemisphere), and this contradict to \cite[Theorem 6.2]{CC1}. Thus, we get $X=S^n$ and Main Theorem 1.
\end{proof}


\section{Continuity of the eigenvalues}
In this appendix we prove the continuity of the eigenvalues of the Laplacian defined in Definition \ref{def1} for a non-collapsed Gromov-Hausdorff convergent sequence of $n$-dimensional closed Riemannian manifolds with a uniform $2$-sided bound on the Ricci curvature.
As an application, on such a limit space, we consider the Obata equation $\nabla^2 f+fg=0$ and generalize our main theorem.

Take an integer $n\geq 2$, real numbers $K_1,K_2\in \mathbb{R}$ with $K_1<K_2$ and positive real numbers $D>0$ and $v>0$.
Let $\mathcal{M}=\mathcal{M}(n,K_1,K_2,D,v)$ be the set of isometry classes of $n$-dimensional closed Riemannian manifolds $(M,g)$ with $K_1g\leq \Ric_g \leq K_2 g$, $\diam(M)\leq D$ and $\Vol(M)\geq v$.
Let $\overline{\mathcal{M}}=\overline{\mathcal{M}}(n,K_1,K_2,D,v)$ be the closure of $\mathcal{M}$ in the Gromov-Hausdorff topology.
If $X_i\in\overline{\mathcal{M}}$ ($i\in \mathbb{N}$) converges to $X\in\overline{\mathcal{M}}$ in the Gromov-Hausdorff topology, then there exist a sequence of positive real numbers $\{\epsilon_i\}_{i\in \mathbb{N}}$ with $\lim_{i\to \infty}\epsilon_i=0$, and a sequence of $\epsilon_i$-Hausdorff approximation maps $\phi_i \colon X_i\to X$. Fix such a sequence. We say a sequence $x_i\in X_i$ converges to $x\in X$ if $\lim_{i\to \infty}\phi_i(x_i)=x$ (denote it by $x_i\stackrel{GH}{\to} x$).
By the volume convergence theorem \cite[Theorem 5.9]{CC1}, $(X_i,H^n)$ converges to $(X,H^n)$ in the measured Gromov-Hausdorff sense, i.e., for all $r>0$ and all sequence $x_i\in X_i$ that converges to $x\in X$, we have $\lim_{i\to \infty}H^n(B_r (x_i))=H^n(B_r(x))$, where $H^n$ denotes the $n$-dimensional Hausdorff measure.
In particular, for all $X\in \overline{\mathcal{M}}$, we have $v\leq H^n(X)\leq C(n,K_1,D)$.

For all $X\in \overline{\mathcal{M}}$, we can consider the cotangent bundle $\pi \colon T^\ast X \to X$ with a canonical inner product by \cite{Ch0} and \cite{CC3} (see also \cite[Section 2]{Ho1} for a short review).
We have $H^n(X\setminus \pi(T^\ast X))=0$ and $T^\ast_x X:=\pi^{-1}(x)$ is an $n$-dimensional vector space for all $x\in \pi(T^\ast X)$.
For all Lipschitz function $f$ on $X$, we can define $d f(x)\in T_x^\ast X$ for almost all $x\in X$, and we have $d f\in L^\infty(T^\ast X)$.
Let $TX$ be the dual bundle of $T^\ast X$.
Let $\langle\cdot,\cdot \rangle$ denotes the inner product on $T^\ast X$, $TX$ and tensor bundles.
We also denote the inner product on $TX$ by $g=g_X\in L^\infty (T^\ast X\otimes T^\ast X)$.
We consider the vector bundle $E(X):=T^\ast X\oplus \mathbb{R}e$ on $X$ with the product metric $\langle\cdot,\cdot\rangle=\langle\cdot,\cdot\rangle_E$.
We have an identification $L^2(E(X))=L^2(T^\ast X)\oplus L^2(X)$ by considering the map $\omega+f e\mapsto (\omega,f)$.

Honda \cite{Ho1} (see also page 1595--1596 of \cite{Ho2}) defined the concepts of $L^p$ weakly convergence and $L^p$ strong convergence for functions and tensors fields ($p\in (1,\infty)$).
Most of important properties are summarized in page 1596--1598 of \cite{Ho2}.
Similarly, we give the following definition.
\begin{Def}
Let $\{X_i\}_{i\in \mathbb{N}}$ be a sequence in $\mathcal{M}$ and let $X\in\overline{\mathcal{M}}$ be the Gromov-Hausdorff limit.
Take a sequence $T_i\in L^2(E(X_i))$ and $T\in L^2(E(X))$.
\begin{itemize}
\item[(i)] We say that $T_i$ converges to $T$ weakly in $L^2$ if 
\begin{equation*}
\sup_{i\in \mathbb{N}}\|T_i\|_2<\infty,
\end{equation*}
and for all $r>0$, $y_i,z_i\in X_i$ and $y,z\in X$ with $y_i\stackrel{GH}{\to} y$ and $z_i\stackrel{GH}\to z$, we have
\begin{equation*}
\begin{split}
\lim_{i\to \infty}\int_{B_r(y_i)} \langle T_i, d r_{z_i}\rangle_E \,d H^n&=\int_{B_r(y)} \langle T, d r_{z}\rangle_E \,d H^n,\\
\lim_{i\to \infty}\int_{B_r(y_i)} \langle T_i, e\rangle_E \,d H^n&=\int_{B_r(y)} \langle T, e\rangle_E \,d H^n,
\end{split}
\end{equation*}
where $r_z(x):=d(z,x)$ for all $x\in X$.
\item[(ii)] We say that $T_i$ converges to $T$ strongly in $L^2$ if $T_i$ converges to $T$ weakly in $L^2$, and
\begin{equation*}
\limsup_{i\to \infty} \|T_i\|_2\leq \|T\|_2
\end{equation*}
holds.
\end{itemize}
\end{Def}
We have that $T_i=\omega_i+f_i e\in L^2(E(X_i))$ converges to $T=\omega+f e\in L^2(E(X))$ weakly in $L^2$ if and only if $\omega_i$, $f_i$ converges to $\omega$, $f$ weakly in $L^2$, respectively.
Similarly, $T_i\in L^2(E(X_i))$ converges to $T\in L^2(E(X))$ strongly in $L^2$ if and only if $\omega_i$, $f_i$ converges to $\omega$, $f$ strongly in $L^2$, respectively.
One of the important property for the weakly convergent sequence is the lower semi-continuity of the norm:
\begin{equation*}
\liminf_{i\to \infty}\|T_i\|_2\geq \|T\|_2
\end{equation*}
holds if $T_i\in L^2(X_i)$ converges to $T\in L^2(X)$ weakly in $L^2$ (see \cite[Proposition 3.64]{Ho1}).
Thus, $T_i$ converges to $T$ strongly in $L^2$ if and only if $T_i$ converges to $T$ weakly in $L^2$ and $\lim_{i\to \infty} \|T_i\|_2=\|T\|_2$ holds.
For the linearity of the strong convergence, see \cite[Corollary 3.59]{Ho1}.

We next consider the space $L^2(T^\ast X \oplus E(X))$.
We also have an identification $L^2(T^\ast X \oplus E(X))=L^2(T^\ast X\otimes T^\ast X)\oplus L^2(TX)$, and define the concepts of convergence for $L^2(T^\ast X\oplus E(X))$.
\begin{Def}
Let $\{X_i\}_{i\in \mathbb{N}}$ be a sequence in $\mathcal{M}$ and let $X\in\overline{\mathcal{M}}$ be the Gromov-Hausdorff limit.
Take a sequence $T_i\in L^2(T^\ast X_i\oplus E(X_i))$ and $T\in L^2(T^\ast X\oplus E(X))$.
\begin{itemize}
\item[(i)] We say that $T_i$ converges to $T$ weakly in $L^2$ if 
\begin{equation*}
\sup_{i\in \mathbb{N}}\|T_i\|_2<\infty,
\end{equation*}
and for all $r>0$, $y_i,z_i^1,z_i^2\in X_i$ and $y,z^1,z^2\in X$ with $y_i\stackrel{GH}{\to} y$, $z_i^j\stackrel{GH}\to z$ ($j=1,2$), we have
\begin{equation*}
\begin{split}
\lim_{i\to \infty}\int_{B_r(y_i)} \langle T_i, d r_{z_i^1}\otimes d r_{z_i^2}\rangle_E \,d H^n&=\int_{B_r(y)} \langle T, d r_{z^1}\otimes d r_{z^2}\rangle_E \,d H^n,\\
\lim_{i\to \infty}\int_{B_r(y_i)} \langle T_i, d r_{z_i^1}\otimes e\rangle_E \,d H^n&=\int_{B_r(y)} \langle T, dr_{z^1}\otimes e\rangle_E \,d H^n.
\end{split}
\end{equation*}
\item[(ii)] We say that $T_i$ converges to $T$ strongly in $L^2$ if $T_i$ converges to $T$ weakly in $L^2$ and
\begin{equation*}
\limsup_{i\to \infty} \|T_i\|_2\leq \|T\|_2
\end{equation*}
holds.
\end{itemize}
\end{Def}
We have that $T_i=S_i+\omega_i\otimes e\in L^2(T^\ast X_i\otimes E(X_i))$ converges to $T=S+\omega\otimes e\in L^2(T^\ast X\otimes E(X))$ weakly in $L^2$ if and only if $S_i$, $\omega_i$ converges to $S$, $\omega$ weakly in $L^2$, respectively.
Similarly, $T_i\in L^2(T^\ast X_i\otimes E(X_i))$ converges to $T\in L^2(T^\ast X\otimes E(X))$ strongly in $L^2$ if and only if $S_i$, $\omega_i$ converges to $S$, $\omega$ strongly in $L^2$, respectively.

We list the definitions of important functional spaces and the eigenvalues of Laplacian on limit spaces.
Some of them are first introduced by Gigli \cite{Gig}.
\begin{Def}\label{defapb}
Let $X\in \overline{\mathcal{M}}$.
\begin{itemize}
\item[(i)] Let $\LIP(X)$ be the set of the Lipschitz functions on $X$. For all $f\in \LIP(X)$, we define $\|f\|_{H^{1,2}}^2=\|f\|_2^2+\|d f\|_2^2$.
Let $H^{1,2}(X)$ be the completion of $\LIP(X)$ with respect to this norm.
\item[(ii)] Define
\begin{equation*}
\begin{split}
\mathcal{D}^2(\Delta,X):=\Big\{f\in H^{1,2}(X)& : \text{there exists $F\in L^2(X)$ such that}\\
&\int_X \langle df, dh \rangle\,d H^n=\int_X F h\,d H^n \text{ for all $h\in H^{1,2}(X)$} \Big\}.
\end{split}
\end{equation*}
For any $f\in\mathcal{D}^2(\Delta,X)$, the function $F\in L^2(X)$ is uniquely determined.
Thus, we define $\Delta f:=F$.
\item[(iii)] Define
\begin{equation*}
\begin{split}
\Test F(X):=&\left\{f\in\mathcal{D}^2(\Delta,X)\cap \LIP(X):\Delta f\in H^{1,2}(X)\right\},\\
\Test T^\ast X:=&\left\{\sum_{i=1}^N f_i d h_i: N\in \mathbb{N},\, f_i, h_i\in \Test F(X)\right\},\\
\Test E(X):=&\left\{\sum_{i=1}^N \psi_i d\phi_i+fe: N\in \mathbb{N},\, \psi_i, \phi_i,f\in \Test F(X)\right\}.
\end{split}
\end{equation*}
We have an identification $\Test E(X)=\Test T^\ast X\oplus \Test F(X)$.
\item[(vi)] The operator
\begin{equation*}
\nabla\colon \Test T^\ast X\to L^2(T^\ast X \otimes T^\ast X)
\end{equation*}
is defined by $\nabla \sum_{i=1}^N f_i d h_i:=\sum_{i=1}^N \left(d f_i\otimes d h_i+ f_i \nabla^2 h_i \right)$, where $\nabla^2$ denotes the Hessian $\Hess^g$ defined in \cite{Ho0}. Note that $\nabla \omega\in L^2(T^\ast X \otimes T^\ast X)$ for all $\omega\in \Test T^\ast X$ by \cite[Theorem 4.11]{Ho3}.
We define
\begin{equation*}
\nabla^E\colon \Test E(X)\to L^2(T^\ast X\otimes E(X))
\end{equation*}
by $\nabla^E (\omega+f e):=\nabla \omega +f g+(d f-\omega)\otimes e$.
\item[(v)] We define the norms
\begin{equation*}
\begin{split}
\|\omega\|_{H_C^{1,2}}^2&:=\|\omega\|_2^2+\|\nabla \omega\|_2^2\quad (\omega\in \Test T^\ast X),\\
\|T\|_{H_E^{1,2}}^2&:=\|T\|_2^2+\|\nabla^E T\|_2^2\quad (T\in \Test E(X)).
\end{split}
\end{equation*}
Let $H^{1,2}_C(T^\ast X)$ and $H^{1,2}_E(E(X))$ be the completion of $\Test T^\ast X$ and $\Test E(X)$ with respect to the norms $\|\cdot\|_{H^{1,2}_C}$ and $\|\cdot\|_{H^{1,2}_E}$.
\item[(vi)] Define
\begin{equation*}
\begin{split}
\mathcal{D}^2(\Delta_{C,1},X):=\Big\{\omega \in& H^{1,2}_C(T^\ast X) : \text{there exists $\hat{\omega}\in L^2(T^\ast X)$ such that}\\
&\int_X \langle \nabla \omega, \nabla \eta \rangle\,d H^n=\int_X \langle\hat{\omega}, \eta\rangle \,d H^n \text{ for all $\eta\in H_C^{1,2}(T^\ast X)$} \Big\}.
\end{split}
\end{equation*}
For any $\omega\in\mathcal{D}^2(\Delta_{C,1},X)$, the form $\hat{\omega}\in L^2(T^\ast X)$ is uniquely determined.
Thus, we put $\Delta_{C,1} \omega:=\hat{\omega}$.
\item[(vii)] Define
\begin{equation*}
\begin{split}
\mathcal{D}^2(\bar{\Delta}^E,X):=\Big\{T\in& H^{1,2}_E(E(X)) : \text{there exists $\hat{T}\in L^2(E(X))$ such that}\\
&\int_X \langle \nabla^E T, \nabla^E S \rangle\,d H^n=\int_X \langle\hat{T}, S\rangle \,d H^n \text{ for all $S\in H_E^{1,2}(E(X))$} \Big\}.
\end{split}
\end{equation*}
For any $T\in\mathcal{D}^2(\bar{\Delta}^E,X)$, the section $\hat{T}\in L^2(E(X))$ is uniquely determined.
Thus, we define $\bar{\Delta}^E T:=\hat{T}$.
\item[(viii)] For all $k\in \mathbb{Z}_{>0}$, we define
\begin{equation*}
\begin{split}
\lambda_k(\bar{\Delta}^{E}):=\inf\left\{\sup_{T\in \mathcal{E}_k\setminus \{0\}}\frac{\|\nabla^E T\|^2_2}{\|T\|^2_2}: \mathcal{E}_k\subset H^{1,2}_E(E(X))\text{ is a $k$-dimensional subspace}\right\}.
\end{split}
\end{equation*}
\end{itemize}
\end{Def}
We give some easy lemmas about Definition \ref{defapb}.

We first investigate the relationship between $H^{1,2}_C(T^\ast X)$, $H^{1,2}(X)$ and $H^{1,2}(E(X))$.
\begin{Lem}\label{ap1}
Let $X\in \overline{\mathcal{M}}$.
For all $T=\omega + f e\in \Test E(X)$, we have
\begin{equation*}
\frac{1}{2(n+1)}(\|\omega\|_{H^{1,2}_C}^2+\|f\|_{H^{1,2}}^2)\leq \|T\|_{H^{1,2}_E}^2\leq 2(n+1)(\|\omega\|_{H^{1,2}_C}^2+\|f\|_{H^{1,2}}^2),
\end{equation*}
and so we have an identification $H^{1,2}_E(E(X))=H^{1,2}_C(T^\ast X)\oplus H^{1,2}(X)$.
\end{Lem}
\begin{proof}
Take arbitrary $T=\omega + f e\in \Test E(X)$.
We have
\begin{equation*}
\begin{split}
\|T\|_{H^{1,2}_E}=&\|T\|_2^2+\|\nabla^E T\|_2^2\\
\leq&\|\omega\|_2^2+\|f\|_2^2+(\|\nabla \omega\|_2+\sqrt{n}\|f\|_2)^2+(\|d f\|_2+\|\omega\|_2)^2\\
\leq&2(n+1)(\|\omega\|_{H^{1,2}_C}^2+\|f\|_{H^{1,2}}^2).
\end{split}
\end{equation*}
On the other hand, we have
\begin{equation*}
\begin{split}
&(n+1)\|T\|_{H^{1,2}_E}\\
\geq&\|\omega\|_2^2+\|f\|_2^2+(\sqrt{n}\|f\|_2)^2+(\|\nabla \omega\|_2-\sqrt{n}\|f\|_2)^2
+\|\omega\|_2^2+(\|d f\|_2-\|\omega\|_2)^2\\
\geq &\|\omega\|_2^2+\|f\|_2^2+\frac{1}{2}\|\nabla \omega\|_2^2+\frac{1}{2}\|d f\|_2^2\\
\geq &\frac{1}{2}(\|\omega\|_{H^{1,2}_C}^2+\|f\|_{H^{1,2}}^2).
\end{split}
\end{equation*}
Thus, we get the lemma.
\end{proof}
We next investigate the relationship between $\mathcal{D}^2(\Delta_{C,1},X)$, $\mathcal{D}^2(\Delta,X)$ and $\mathcal{D}^2(\bar{\Delta}^E,X)$.
\begin{Lem}\label{ap2}
Let $X\in \overline{\mathcal{M}}$.
We have an identification 
\begin{equation*}
\mathcal{D}^2(\bar{\Delta}^E,X)=\mathcal{D}^2(\Delta_{C,1},X)\oplus \mathcal{D}^2(\Delta,X).
\end{equation*}
Moreover, for all $T=\omega+f e\in \mathcal{D}^2(\bar{\Delta}^E,X)$, we have 
$\bar{\Delta}^E T=\Delta_{C,1} \omega-2d f +\omega+(\Delta f -2\delta \omega +n f)e$.
\end{Lem}
\begin{proof}
For all $T=\omega+f e,S=\eta+he\in H^{1,2}_E(E(X))$, we have
\begin{equation}\label{apeqb1}
\begin{split}
&\int_X \langle \nabla^E T, \nabla^E S \rangle\,d H^n\\
=&\int_{X}\Big( \langle \nabla \omega,\nabla \eta\rangle
-f\nabla^\ast \eta - h\nabla^\ast \omega+n f h +\langle d f-\omega,d h-\eta \rangle \Big) \,d H^n,
\end{split}
\end{equation}
where we defined $\nabla^\ast \omega:=-\tr \nabla \omega\in L^2(X)$.
Since $\nabla^\ast \omega=\delta \omega$ holds for the smooth case, we have $\nabla^\ast \omega=\delta \omega$ for the general case by \cite[Theorem 3.74]{Ho1} and the smooth approximation theorem \cite[Theorem 3.5]{Ho2}, where $\delta \omega\in L^2(X)$ is characterized by satisfying
\begin{equation*}
\int_X  f \delta \omega\,dH^n=\int_X \langle \omega, d f\rangle \,dH^n,
\end{equation*}
for all $f\in \Test F(X)$.

We first prove $\mathcal{D}^2(\Delta_{C,1},X)\oplus \mathcal{D}^2(\Delta,X)\subset \mathcal{D}^2(\bar{\Delta}^E,X)$ and the second assertion.
Take arbitrary $T=\omega+f e\in \mathcal{D}^2(\Delta_{C,1},X)\oplus \mathcal{D}^2(\Delta,X)$.
For all $S=\eta+he\in H^{1,2}_E(E(X))$, by (\ref{apeqb1}) we have
\begin{equation*}
\begin{split}
&\int_X \langle \nabla^E T, \nabla^E S \rangle\,d H^n\\
=&\int_{X}\Big( \langle \Delta_{C,1} \omega,\eta\rangle
-\langle d f, \eta\rangle - h\delta \omega+n f h +\langle \Delta f-\delta \omega, h\rangle-\langle d f-\omega,\eta \rangle \Big) \,d H^n\\
=&\int_{X} \langle \Delta_{C,1} \omega-2d f +\omega+(\Delta f -2\delta \omega +n f)e,S\rangle\,d H^n.
\end{split}
\end{equation*}
Thus, we get $T\in \mathcal{D}^2(\Delta,X)$ and $\bar{\Delta}^E T=\Delta_{C,1} \omega-2d f +\omega+(\Delta f -2\delta \omega +n f)e$.

We next prove $\mathcal{D}^2(\bar{\Delta}^E,X)\subset\mathcal{D}^2(\Delta_{C,1},X)\oplus \mathcal{D}^2(\Delta,X)$.
Take arbitrary $T=\omega+f e\in \mathcal{D}^2(\bar{\Delta}^E,X)$.
Take $\hat{\omega}\in L^2(T^\ast X)$ and $\hat{f}\in L^2(X)$ with $\bar{\Delta}^E T=\hat{\omega}+\hat{f} e$.
For all $\eta \in H^{1,2}_C(T^\ast X)$, by (\ref{apeqb1}) we have
\begin{equation*}
\begin{split}
\int_X \langle \hat{\omega},\eta\rangle \,d H^n
=&\int_X \langle \nabla^E T, \nabla^E \eta \rangle\,d H^n\\
=&\int_{X}\Big( \langle \nabla \omega,\nabla \eta\rangle
-\langle d f, \eta\rangle +\langle d f-\omega,-\eta \rangle \Big) \,d H^n,
\end{split}
\end{equation*}
and so
\begin{equation*}
\int_{X} \langle \nabla \omega,\nabla \eta\rangle\,d H^n
=\int_X \langle \hat{\omega}+2d f -\omega ,\eta\rangle \,d H^n.
\end{equation*}
Thus, we have $\omega\in \mathcal{D}^2(\Delta_{C,1},X)$.
Similarly, we have $f\in\mathcal{D}^2(\Delta,X)$.
\end{proof}

We mention the convergence of the image of the connection $\nabla^E$.
\begin{Lem}\label{ap3}
Let $\{X_i\}_{i\in \mathbb{N}}$ be a sequence in $\mathcal{M}$ and let $X\in\overline{\mathcal{M}}$ be the Gromov-Hausdorff limit.
Suppose that a sequence $T_i=\omega_i+f_i e\in H^{1,2}_E(E(X_i))$ converges to $T=\omega +f e\in H^{1,2}(E(X))$ strongly in $L^2$. Then, we have that $\nabla^E T_i$ converges to $\nabla^E T$ weakly in $L^2$ if and only if $\nabla \omega_i,d f_i$ converges to $\nabla\omega,d f$ weakly in $L^2$, respectively. We also have that $\nabla^E T_i$ converges to $\nabla^E T$ strongly in $L^2$ if and only if
$\nabla \omega_i,d f_i$ converges to $\nabla\omega,d f$ strongly in $L^2$, respectively.
\end{Lem}
\begin{proof}
Since $\nabla^E T_i= \nabla \omega_i+f_i g_{X_i} + (d f_i-\omega_i)\otimes e$ and $f_i g_{X_i}, \omega_i$ converges to $f g,\omega$ strongly in $L^2$, respectively by Proposition 3.44 and Proposition 3.70 in \cite{Ho1}, we get the lemma.
\end{proof}

We show basic properties about the eigenvalues of $\bar{\Delta}^E$ on limit spaces.
\begin{Thm}\label{ap4}
Let $X\in \overline{\mathcal{M}}$.
\begin{itemize}
\item[(i)] We have
\begin{equation*}
0\leq \lambda_1(\bar{\Delta}^{E})\leq\lambda_2(\bar{\Delta}^{E})\leq \cdots  \to \infty.
\end{equation*}
\item[(ii)] There exists a complete orthonormal system of eigensection $\{T_k\}$ in $L^2(E(X))$, i.e., for all $k\in \mathbb{Z}_{>0}$, we have $T_k\in \mathcal{D}^2(\bar{\Delta}^E,X)$ and
\begin{equation*}
\bar{\Delta}^{E} T_k= \lambda_k(\bar{\Delta}^{E}) T_k.
\end{equation*}
\item[(iii)] For all $\lambda \in \mathbb{R}$, we have
\begin{equation*}
\begin{split}
\{T\in \mathcal{D}^2(\bar{\Delta}^E,X): \bar{\Delta}^{E} T= \lambda T\}
=&\bigoplus_{\lambda_k(\bar{\Delta}^{E}) =\lambda}\mathbb{R}T_k\\
:=&\Span_{\mathbb{R}}\{T_k : k\in\mathbb{Z}_{>0} \text{ with } \lambda_k(\bar{\Delta}^E) = \lambda\}.
\end{split}
\end{equation*}
\end{itemize}
\end{Thm}
\begin{proof}
Take a sequence $S_i\in H^{1,2}(E(X))$ ($i\in \mathbb{N}$) such that $\|S_i\|_2=1$ and
$\lim_{i\to \infty}\|\nabla^E S_i\|_2^2=\lambda_1(\bar{\Delta}^E,X)$ hold.
Since $\sup_{i}\|\nabla^E S_i\|_2<\infty$, there exist a subsequence (denote it again by $S_i$) and $T_1\in H^{1,2}(E(X))$ such that 
$S_i$ converges to $T_1$ strongly in $L^2$ and $\nabla^E S_i$ converges to $\nabla^E T_1$ weakly in $L^2$ by the weak compactness of the closed ball in Hilbert spaces and the Rellich theorem for limit spaces \cite[Theorem 4.9]{Ho1}.
Then, we have $\|T_1\|_2=1$ and $\|\nabla^E T_1\|_2^2\leq \liminf \|\nabla^E S_i\|_2^2=\lambda_1(\bar{\Delta}^E,X)$.
By the definition of $\lambda_1(\bar{\Delta}^E,X)$, we have $\|\nabla^E T_1\|_2^2\geq\lambda_1(\bar{\Delta}^E,X)$, and so $\|\nabla^E T_1\|_2^2=\lambda_1(\bar{\Delta}^E,X)$.

For all $t\in\mathbb{R}$ and $S\in H^{1,2}(E(X))$ with $\int_X \langle T_1,S\rangle \,d H^n=0$, we have
\begin{equation*}
\begin{split}
\lambda_1(\bar{\Delta}^E,X)(1+t^2\|S\|_2^2)
=&\lambda_1(\bar{\Delta}^E,X)\|T_1+t S\|_2^2\\
\leq& \|\nabla^E T_1+t \nabla^E S\|_2^2\\
= & \lambda_1(\bar{\Delta}^E,X)+\frac{2t}{H^n(X)}\int_X\langle \nabla^E T_1,\nabla^E S\rangle \,d H^n+t^2 \|\nabla^E S\|_2^2.
\end{split}
\end{equation*}
Thus, we have $\int_X\langle \nabla^E T_1,\nabla^E S\rangle \,d H^n=0$.
Combining this and $\int_X\langle \nabla^E T_1,\nabla^E T_1\rangle \,d H^n=\lambda_1(\bar{\Delta}^E,X) \int_X\langle T_1,T_1\rangle \,d H^n$, we get $\int_X\langle \nabla^E T_1,\nabla^E S\rangle \,d H^n=\lambda_1(\bar{\Delta}^E,X) \int_X\langle T_1,S\rangle \,d H^n$ for all $S\in H^{1,2}(E(X))$.
Therefore, we get $T_1\in \mathcal{D}^2(\bar{\Delta}^E,X)$ and $\bar{\Delta}^{E} T_1= \lambda_1(\bar{\Delta}^{E}) T_1$.

Suppose that we have chosen $T_1,\ldots,T_k\in \mathcal{D}^2(\bar{\Delta}^E,X)$ ($k\in \mathbb{Z}_{>0}$) such that $\frac{1}{H^n(X)}\int_X\langle  T_i, T_j\rangle \,d H^n=\delta_{i j}$ and $\bar{\Delta}^{E} T_i= \lambda_i(\bar{\Delta}^{E}) T_i$ hold.
Considering
\begin{equation*}
\begin{split}
\widetilde{\lambda}_{k+1}:=\inf\Big\{\|\nabla^E S\|_2^2: S\in& H^{1,2}_E(E(X))\text{ with $\|S\|_2=1$ and}\\
& \int_X\langle S,T_i \rangle\,dH^n=0 \text{ for all $i=1,\ldots,k$}\Big\},
\end{split}
\end{equation*}
we can take $T_{k+1}\in \mathcal{D}^2(\bar{\Delta}^E,X)$ such that $\|T_{k+1}\|_2=1$,  $\bar{\Delta}^{E} T_{k+1}= \widetilde{\lambda}_{k+1} T_{k+1}$ and $\int_X\langle T_{k+1},T_i \rangle\,dH^n=0$ for all $i=1,\ldots,k$.

Let us show $\widetilde{\lambda}_{k+1}=\lambda_{k+1}(\bar{\Delta}^{E})$.
Define $\widetilde{\mathcal{E}}_{i}:=\Span_\mathbb{R}\{T_1,\ldots,T_{i},T_{k+1}\}$ for $i=1,\ldots,k$.
If $\widetilde{\lambda}_{k+1}\geq \lambda_{i}(\bar{\Delta}^{E})$, we get
\begin{equation*}
\widetilde{\lambda}_{k+1}=\sup_{T\in\widetilde{\mathcal{E}}_{i}\setminus\{0\}} \frac{\|\nabla^E T\|_2^2}{\|T\|_2^2}\geq \lambda_{i+1}(\bar{\Delta}^{E}).
\end{equation*}
Since we have $\widetilde{\lambda}_{k+1}\geq \lambda_{1}(\bar{\Delta}^{E})$, we get $\widetilde{\lambda}_{k+1}\geq \lambda_{k+1}(\bar{\Delta}^{E})$ by induction.
To prove $\widetilde{\lambda}_{k+1}\leq \lambda_{k+1}(\bar{\Delta}^{E})$, we take arbitrary $(k+1)$-dimensional subspace $\mathcal{E}_{k+1}\subset H^{1,2}_E(E(X))$.
Then, there exists an element $S\in \mathcal{E}_{k+1}$ with $\|S\|_2=1$ and $\int_X\langle S,T_i \rangle\,dH^n=0$ for all $i=1,\ldots,k$.
Thus, we get
\begin{equation*}
\sup_{T\in\mathcal{E}_{k+1}\setminus\{0\}} \frac{\|\nabla^E T\|_2^2}{\|T\|_2^2}\geq \widetilde{\lambda}_{k+1}.
\end{equation*}
This shows that $\widetilde{\lambda}_{k+1}\leq \lambda_{k+1}(\bar{\Delta}^{E})$.
Therefore, $\widetilde{\lambda}_{k+1}= \lambda_{k+1}(\bar{\Delta}^{E})$.

By induction, we get a sequence $\{T_k\}_{k\in\mathbb{Z}_{>0}}\subset \mathcal{D}^2(\bar{\Delta}^E,X)$ such that $\bar{\Delta}^{E} T_k= \lambda_k(\bar{\Delta}^{E}) T_k$ and $\frac{1}{H^n(X)}\int_X\langle T_i, T_j\rangle \,d H^n=\delta_{i j}$.

We next prove that $\lim_{k\to \infty}\lambda_k(\bar{\Delta}^{E})=\infty$.
Suppose that $\lim_{k\to \infty}\lambda_k(\bar{\Delta}^{E})=\lambda<\infty$.
Since $\sup_{i}\|\nabla^E T_i\|_2\leq \lambda$, there exist a subsequence $i(j)$ and $T\in H^{1,2}(E(X))$ such that 
$T_{i(j)}$ converges to $T$ strongly in $L^2$ and $\nabla^E T_{i(j)}$ converges to $\nabla^E T$ weakly in $L^2$. We have $\|T\|_2=1$.
For all $k\in\mathbb{Z}_{>0}$, we have
\begin{equation*}
\int_X\langle T_k, T \rangle\,d H^n=\lim_{l\to \infty}\int_X\langle T_k, T_l \rangle\,d H^n=0.
\end{equation*}
Thus, we get
\begin{equation*}
\int_X\langle T, T \rangle\,d H^n=\lim_{k\to \infty}\int_X\langle T_k, T \rangle\,d H^n=0.
\end{equation*}
This contradicts to $\|T\|_2=1$.
Therefore, we get $\lim_{k\to \infty}\lambda_k(\bar{\Delta}^{E})=\infty$ and (i).

Suppose that $\bigoplus_{k=1}^\infty \mathbb{R} T_k$ is not dense in $H^{1,2}_E(E(X))$.
Then, there exists $T\in H^{1,2}_E(E(X))$ such that  $\|T\|_2=1$ and $\int_X\langle T_k, T \rangle\,d H^n=0$ for all $k\in\mathbb{Z}_{>0}$.
By the definition of $\widetilde{\lambda}_k$, we get $\|\nabla^E T\|_2^2\geq \widetilde{\lambda}_k= \lambda_k(\bar{\Delta}^{E})$ for all $k\in\mathbb{Z}_{>0}$.
This contradicts to $\lim_{k\to \infty}\lambda_k(\bar{\Delta}^{E})=\infty$.
Since $H^{1,2}_E(E(X))$ is dense in $L^2(E(X))$ (see \cite[Claim 3.2]{Ho2}), we get that $\{T_k\}$ is complete orthonormal system in $L^2(E(X))$.
Thus, we get (ii).

Similarly, we have (iii).
\end{proof}

Let us give the $L^\infty$ estimate for the eigensections of the Laplacian $\bar{\Delta}^E$.
\begin{Lem}\label{ap5}
Let $(M,g)\in \mathcal{M}$.
Take positive real numbers $\beta>0$ and $0\leq \alpha \leq \beta$.
Then, for any $T\in \Gamma(E(M))$ with $\bar{\Delta}^E T=\alpha T$, we have
\begin{equation*}
\|T\|_{\infty}\leq C(n,K_1,D,\beta) \|T\|_2.
\end{equation*} 
\end{Lem}
\begin{proof}
Since we have
\begin{equation*}
\Delta |T|^2=2\langle \bar{\Delta}^E T, T\rangle-2|\nabla^E T|^2\leq 2 \beta |T|^2,
\end{equation*}
we get the lemma by  \cite[Proposition 9.2.7]{Pe3} (see also Proposition 7.1.13 and Proposition 7.1.17 in \cite{Pe3}).
Note that our sign convention of the Laplacian is different from \cite{Pe3}.
\end{proof}

We investigate the convergence of eigensections.
The following proposition plays an important role to prove Theorem \ref{ap7}.
\begin{Prop}\label{ap6}
Let $\{X_i\}_{i\in \mathbb{N}}$ be a sequence in $\mathcal{M}$ and let $X\in\overline{\mathcal{M}}$ be the Gromov-Hausdorff limit.
Take a sequence $\{\lambda_i\}$ in $\mathbb{R}_{\geq 0}$ and $T_i\in \Gamma(E(X_i))$ with $\bar{\Delta}^{E} T_i =\lambda_i T_i$ and $\|T_i\|_2=1$.
Suppose that $T_i$ converges to $T\in L^2(E(X))$ weakly in $L^2$, and there exists $\mu>0$ with $\sup_{i}\lambda \leq \mu$.
Then, we have the following properties.
\begin{itemize}
\item[(i)]  There exist $\lambda \in \mathbb{R}_{\geq 0}$ with $\lambda=\lim_{i\to \infty} \lambda_i$.
\item[(ii)] We have $T\in \mathcal{D}^2(\bar{\Delta}^E,X)\cap L^\infty (E(X))$,
$\|T\|_{\infty}\leq C(n,K,D,\mu)$ and $\bar{\Delta}^E T=\lambda T$.
\item[(iii)] $T_i,\nabla^E T_i$ converges $T,\nabla^E T$ strongly in $L^2$, respectively.
\end{itemize}
\end{Prop}
\begin{proof}
Take $\omega_i\in \Gamma(T^\ast X_i)$ and $f_i\in C^\infty(X_i)$ with $T_i=\omega_i+f_i e$.
Since $\|T_i\|_{\infty}\leq C(n,K,D,\beta)$, we have 
$\|\omega_i\|_{\infty} \leq C(n,K,D,\beta)$.
By \cite[Theorem 4.9]{Ho1} and \cite[Theorem 6.11]{Ho3} (see also \cite[Proposition 4.2]{Ho2}), there exist a subsequence $i(j)$, $\omega\in L^2(T^\ast X)$ and $f\in H^{1,2}(X)$ such that $\omega$ is differentiable at almost all point in X, $\nabla \omega\in L^2(T^\ast X\otimes T^\ast X)$, $\omega_i$, $f_i$ converges to $\omega$, $f$ strongly in $L^2$ and $\nabla \omega_i$, $d f_i$ converges to $\nabla\omega$, $d f$ weakly in $L^2$, respectively.
By the lower semi-continuity of $L^\infty$ norm \cite[Proposition 3.64]{Ho1}, we have $\|f\|_{\infty}\leq C(n,K,D,\mu)$, $\|\omega\|_{\infty}\leq C(n,K,D,\mu)$ and $\|T\|_{\infty}\leq C(n,K,D,\mu)$.
As with the proof of \cite[Proposition 4.8 (ii)]{Ho2}, we have $\langle\omega,d h\rangle\in H^{1,2}(X)$ for all $h\in \mathcal{D}^2(\Delta,X)$ with $\Delta h\in L^{\infty}(X)$.
By \cite[Proposition 4.5]{Ho2}, we get $\omega\in H^{1,2}_C(T^\ast X)$.

Take arbitrary $S=\eta+he\in \Test E(X)$.
Then, there exist sequences $\eta_i\in \Gamma(T^\ast X_i)$ and $h_i\in C^\infty(X_i)$ such that $\eta_i$, $\nabla \eta_i$, $\delta \eta_i$, $h_i$, $d h_i$ converges to $\eta$, $\nabla \eta$, $\delta \eta$, $h$, $d h$ strongly in $L^2$, respectively by \cite[Proposition 3.5]{Ho2}.
Set $S_i:= \eta_i+ h_i e\in \Gamma(E(X_i))$.
Take a subsequence $i(j)$ such that the limit $\lambda=\lim_{j\to\infty} \lambda_{i(j)}$ exists.
Then, we get 
\begin{equation*}
\begin{split}
\lim_{j\to \infty}\int_{X_{i(j)}}\langle \nabla^E T_{i(j)},\nabla^E S_{i(j)}\rangle \,d H^n=&
\lim_{j\to \infty}\lambda_{i(j)}\int_{X_{i(j)}}\langle T_{i(j)}, S_{i(j)}\rangle \,d H^n\\
=&\lambda\int_{X}\langle T, S \rangle \,d H^n.
\end{split}
\end{equation*}
We have
\begin{equation*}
\begin{split}
&\lim_{j\to \infty}\int_{X_{i(j)}}\langle \nabla^E T_{i(j)},\nabla^E S_{i(j)}\rangle \,d H^n\\
=&\lim_{j\to \infty}\int_{X_{i(j)}}\Big( \langle \nabla \omega_{i(j)},\nabla \eta_{i(j)}\rangle
-f_{i(j)}\delta \eta_{i(j)} - h_{i(j)}\delta \omega_{i(j)}+n f_{i(j)} h_{i(j)}\\
&\qquad\qquad\qquad\qquad\qquad\qquad\qquad+\langle d f_{i(j)}-\omega_{i(j)},d h_{i(j)}-\eta_{i(j)}\rangle \Big) \,d H^n\\
=&\int_{X}\Big( \langle \nabla \omega,\nabla \eta\rangle
-f\delta \eta - h\delta \omega+n f h +\langle d f-\omega,d h-\eta \rangle \Big) \,d H^n\\
=&\int_{X}\langle \nabla^E T,\nabla^E S\rangle \,d H^n.
\end{split}
\end{equation*}
Thus, we get
\begin{equation*}
\int_{X}\langle \nabla^E T,\nabla^E S\rangle \,d H^n=\lambda\int_{X}\langle T, S \rangle \,d H^n.
\end{equation*}
This shows that $T\in \mathcal{D}^2(\bar{\Delta}^E,X)\cap L^\infty (E(X))$ and $\bar{\Delta}^E T=\lambda T$.
Since $\lambda$ is uniquely determined, we have $\lambda=\lim_{i\to \infty} \lambda_i$.
Thus, we get (i) and (ii)

Since we have
\begin{equation*}
\lim_{i\to \infty}\|\nabla^E T_i\|_2^2=\lim_{i\to \infty}\lambda_i=\lambda=\|\nabla^E T\|_2^2,
\end{equation*}
$\nabla^E T_i$ converges to $\nabla^E T$ strongly in $L^2$.
\end{proof}

Let us show the continuity of the eigenvalues of the Laplacian $\bar{\Delta}^E$ under our setting.
The following theorem is the main goal of this appendix.
\begin{Thm}\label{ap7}
Let $\{X_i\}_{i\in \mathbb{N}}$ be a sequence in $\overline{\mathcal{M}}$ and let $X\in\overline{\mathcal{M}}$ be the Gromov-Hausdorff limit.
Then, we have 
\begin{equation*}
\lim_{i\to \infty}\lambda_k(\bar{\Delta}^E,X_i)=\lambda_k(\bar{\Delta}^E,X)
\end{equation*}
for all $k\in\mathbb{Z}_{>0}$.
\end{Thm}
\begin{proof}
We consider the case when $X_i\in \mathcal{M}$ for all $i$.
The general case is its easy consequence.

Take arbitrary $k\in \mathbb{Z}_{>0}$.

We first show the lower semi-continuity
\begin{equation*}
\liminf_{i\to \infty} \lambda_k(\bar{\Delta}^E,X_i)\geq\lambda_k(\bar{\Delta}^E,X).
\end{equation*}
If $\liminf_{i\to \infty} \lambda_k(\bar{\Delta}^E,X_i)=\infty$, this is trivial.
We assume that $\liminf_{i\to \infty} \lambda_k(\bar{\Delta}^E,X_i)<\infty$.
For each $i$, let $\{T_{i,j}\}_{j\in\mathbb{Z}_{>0}}$ be the complete orthonormal system of eigensections in $L^2(E(X_i))$:
\begin{equation*}
\bar{\Delta}^E T_{i,j}=\lambda_j(\bar{\Delta}^E,X_i)T_{i,j}.
\end{equation*}
We can take a subsequence $i(l)$ such that $\liminf_{i\to \infty} \lambda_k(\bar{\Delta}^E,X_i)=\lim_{l\to \infty } \lambda_k(\bar{\Delta}^E,X_{i(l)})$ and $T_{i(l),j}$ converges to some $T_j\in L^2(E(X))$ weakly in $L^2$ as $l\to \infty$ for all $j=1,\ldots,k$ by \cite[Proposition 3.50]{Ho1}.
By Proposition \ref{ap6}, there exist $\lambda_1,\ldots,\lambda_k$ with $\lambda_j=\lim_{l\to \infty}\lambda_j(\bar{\Delta}^E,X_{i(l)})$ for all $j=1,\ldots,k$, and we have $T_j\in \mathcal{D}^2(\bar{\Delta}^E,X)$ and $\bar{\Delta}^E T_{j}=\lambda_j T_{j}$.
Define $\mathcal{E}_{k}:=\Span_{\mathbb{R}}\{T_1,\ldots,T_k\}$.
Then, we get
\begin{equation*}
\lambda_k(\bar{\Delta}^E,X)\leq \sup_{T\in\mathcal{E}_{k}\setminus\{0\}}\frac{\|\nabla^E T\|_2^2}{\|T\|_2^2}=\liminf_{i\to \infty} \lambda_k(\bar{\Delta}^E,X_i).
\end{equation*}

We next show the upper semi-continuity
\begin{equation*}
\limsup_{i\to\infty} \lambda_k(\bar{\Delta}^E,X_i)\leq\lambda_k(\bar{\Delta}^E,X).
\end{equation*}
Let  $\{T_{j}\}_{j\in\mathbb{Z}_{>0}}$ be the complete orthonormal system of eigensections in $L^2(E(X))$:
\begin{equation*}
\bar{\Delta}^E T_{j}=\lambda_j(\bar{\Delta}^E,X)T_{j}.
\end{equation*}
For each $j$, take $\omega_j\in H^{1,2}_C(T^\ast X)$ and $f_j\in H^{1,2}(X)$ such that $T_j=\omega_j+f_j e$.
By \cite[Theorem 3.5]{Ho2}, there exist $\omega_{i,j}\in \Gamma(T^\ast X_i)$ and $f_{i,j}\in C^\infty(X_i)$ such that
$\omega_{i,j}$, $\nabla \omega_{i,j}$, $f_{i,j}$, $d f_{i,j}$ converges to $\omega_j$, $\nabla \omega_{j}$, $f_{j}$, $d f_{j}$ strongly in $L^2$, respectively.
Define $T_{i,j}:=\omega_{i,j}+f_{i,j}e\in\Gamma(E(X_i))$ and $\mathcal{E}_{i,k}:=\Span_{\mathbb{R}}\{T_{i,1},\ldots,T_{i,k}\}\subset \Gamma(E(X_i))$.
Then, $\dim\mathcal{E}_{i,k}=k$ for sufficient large $i$, and
\begin{equation*}
\lambda_k(\bar{\Delta}^E,X)=\lim_{i\to\infty}\sup_{T\in \mathcal{E}_{i,k}\setminus\{0\}}\frac{\|\nabla^E T\|_2^2}{\|T\|_2^2}.
\end{equation*}  
Thus, we get
\begin{equation*}
\limsup_{i\to\infty}\lambda_k(\bar{\Delta}^E,X_i)\leq \limsup_{i\to\infty}\sup_{T\in \mathcal{E}_{i,k}\setminus\{0\}}\frac{\|\nabla^E T\|_2^2}{\|T\|_2^2}=\lambda_k(\bar{\Delta}^E,X).
\end{equation*}

\end{proof}

Let us investigate the relationship between $\lambda_k(\bar{\Delta}^E)$ and the Obata equation $\nabla^2 f + f g=0$.
\begin{Thm}\label{ap8}
Let $X\in\overline{\mathcal{M}}$.
For all $k\in\mathbb{Z}_{>0}$ the following conditions are mutually equivalent.
\begin{itemize}
\item[(i)] $\lambda_k(\bar{\Delta}^E,X)=0$.
\item[(ii)] There exists a $k$-dimensional subspace $V\subset \mathcal{D}^2(\Delta,X)$ such that $\nabla^2 f+f g=0$ holds for all $f\in V$.
\end{itemize}
\end{Thm}
\begin{proof}
We first prove that (ii) implies (i).
For all $f\in \mathcal{D}^2(\Delta,X)$, we have $d f\in H^{1,2}_C(T^\ast X)$ (see \cite[Theorem 1.9]{Ho3} and the definition of $H^{1,2}_C(T^\ast X)$), and so $d f + f e\in H^{1,2}_E(E(X))$ by Lemma \ref{ap1}.
Define $\widetilde{V}:=\{d f + f e: f\in V\}\subset H^{1,2}_E(E(X))$.
Since $\nabla^E T=0$ for all $T\in\widetilde{V}$,
we have $\widetilde{V}\subset \{T\in \mathcal{D}^2(\bar{\Delta}^E,X):\bar{\Delta}^E T=0\}$.
Thus, we get $\lambda_k(\bar{\Delta}^E,X)=0$ by Theorem \ref{ap4}.

We next prove that (i) implies (ii).
For all $T=\omega+f e\in \mathcal{D}^2(\bar{\Delta}^E,X)$ with $\bar{\Delta}^E T=0$, we have $f\in\mathcal{D}^2(\Delta,X)$ (see Lemma \ref{ap2}), $\nabla \omega +f g=0$ and $d f-\omega=0$, and so $\nabla^2 f+ fg =0$.
Since the map $\{T\in \mathcal{D}^2(\bar{\Delta}^E,X):\bar{\Delta}^E T=0\}\to \mathcal{D}^2(\Delta,X),\, T=\omega+fe \mapsto f$ is injective, the image $V\subset \mathcal{D}^2(\Delta,X)$ satisfies $\dim V\geq k$.
Thus, we get (ii).
\end{proof}

\begin{Cor}\label{ap9}
Let $X\in\overline{\mathcal{M}}$.
If there exists a non-zero element $f\in \mathcal{D}^2(\Delta,X)$ with $\nabla^2 f+f g=0$,
then there exists a compact geodesic space $Z$ such that $X$ is isometric to the spherical suspension $S^0\ast Z$.
\end{Cor}
\begin{proof}
By Theorem \ref{ap8}, we have $\lambda_1(\bar{\Delta}^E,X)=0$.
Take a sequence $\{X_i\}$ in $\mathcal{M}$ with $X_i$ converges to $X$ in the Gromov-Hausdorff topology.
Then, we have $\lim_{i\to\infty}\lambda_1(\bar{\Delta}^E,X_i)=0$.
Thus, we get the corollary by Theorem \ref{apa1}.
\end{proof}

We get the following three corollaries immediately by Theorem \ref{apa1}, Theorem \ref{ap6} and Main Theorem 1.
\begin{Cor}\label{ap11}
Given a positive real number $\epsilon>0$ and a integer $1\leq k\leq n-1$, there exists $\delta(n,K_1,D,v,\epsilon)>0$ such that if $X\in\overline{\mathcal{M}}$ satisfies $\lambda_k(\bar{\Delta}^E,X)\leq\delta$, then $d_{GH}(X,S^{k-1}\ast Z)\leq \epsilon$ for some compact geodesic space $Z$.
\end{Cor}
\begin{Cor}
Let $\{X_i\}_{i\in \mathbb{N}}$ be a sequence in $\overline{\mathcal{M}}$.
Then the following three conditions are mutually equivalent.
\begin{itemize}
\item[(i)] $\lambda_n(\bar{\Delta}^E,X_i)\to 0$ as $i\to\infty$.
\item[(ii)] $\lambda_{n+1}(\bar{\Delta}^E,X_i)\to 0$ as $i\to\infty$.
\item[(iii)] $d_{GH}(X_i,S^n)\to 0$ as $i\to \infty$.
\end{itemize}
\end{Cor}
\begin{Cor}\label{ap10}
Given a positive real number $\epsilon>0$, there exists $\delta(n,K_1,D,\epsilon)>0$ such that if $X\in\overline{\mathcal{M}}$ satisfies $\lambda_n(\bar{\Delta}^E,X)\leq\delta$, then $d_{GH}(X,S^n)\leq \epsilon$.
\end{Cor}
Note that $\delta$ in Corollary \ref{ap10} does not depend on $K_2$ and $v$.
If $\epsilon$ is sufficiently small in Corollary \ref{ap10}, then $X$ is bi-H\"{o}lder equivalent to $S^n$ by Theorem 5.9, Theorem A.1.2, Theorem A.1.3, Theorem A.1.5 in \cite{CC1}. In particular, $X$ is homeomorphic to $S^n$.

\bibliographystyle{amsbook}

\end{document}